\documentclass[a4paper,
			titlepage, 
			final,11pt]{book}

 \usepackage{amsmath,amsthm,amssymb,latexsym,enumerate, color}
 \usepackage{amsmath,amssymb,verbatim,mathrsfs,latexsym,paralist}
 \usepackage{mathrsfs}
 \usepackage[T1]{fontenc}
 \usepackage{lmodern}
 \usepackage[utf8]{inputenc}
 \usepackage[english]{babel}
 \usepackage{hyperref}
 \usepackage{bookmark}
\usepackage{graphicx}
\usepackage{afterpage}
\usepackage{float}
\usepackage{epstopdf}
\usepackage{indentfirst}
\usepackage[a4paper]{geometry}

\usepackage{anyfontsize}

\usepackage{fancyhdr}

\usepackage[nottoc]{tocbibind}

\usepackage{emptypage}

\theoremstyle{theorem}
\newtheorem{thm}{Theorem}[chapter]
\newtheorem{prop}[thm]{Proposition}
\newtheorem{lem}[thm]{Lemma}
\newtheorem{cor}[thm]{Corollary}
\newtheorem{rem}[thm]{Remark}

\newcommand{\NN}{\mathbb{N}}

\newcommand{\RR}{\mathbb{R}}
\newcommand{\CC}{\mathbb{C}}
\renewcommand{\SS}{\mathbb{S}}

\def\Xint#1{\mathchoice
{\XXint\displaystyle\textstyle{#1}}%
{\XXint\textstyle\scriptstyle{#1}}%
{\XXint\scriptstyle\scriptscriptstyle{#1}}%
{\XXint\scriptscriptstyle\scriptscriptstyle{#1}}%
\!\int}
\def\XXint#1#2#3{{\setbox0=\hbox{$#1{#2#3}{\int}$}
\vcenter{\hbox{$#2#3$}}\kern-.5\wd0}}

\def\dashint{\Xint-}

\newcommand{\Erfc}{\mathop{\mathrm{Erfc}}}    
    
\newcommand{\re}{\mathop{\mathrm{Re}}}   
   
\newcommand{\dist}{\mathop{\mathrm{dist}}}  
\newcommand{\link}{\mathop{\circ\kern-.35em -}}

\newcommand{\ol}{\overline}
\newcommand{\pa}{\partial}
\newcommand{\ul}{\underline}

\newcommand{\lan}{\langle}
\newcommand{\ran}{\rangle}
\newcommand{\tr}{\mathop{\mathrm{tr}}}

\newcommand{\na}{\nabla}

\newcommand{\al}{\alpha}
\newcommand{\be}{\beta}
\newcommand{\ga}{\gamma}  
\newcommand{\Ga}{\Gamma}
\newcommand{\de}{\delta}
\newcommand{\De}{\Delta}
\newcommand{\ve}{\varepsilon}
 
\newcommand{\la}{\lambda}
\newcommand{\La}{\Lambda}    
\newcommand{\ka}{\kappa}
\newcommand{\si}{\sigma}
\newcommand{\Si}{\Sigma}
\newcommand{\te}{\theta}
\newcommand{\zi}{\zeta}
\newcommand{\om}{\omega}
\newcommand{\Om}{\Omega}

\newcommand{\cH}{{\mathcal H}}

\newcommand{\cK}{{\mathcal K}}

\newcommand{\cS}{{\mathcal S}}

\newcommand{\cX}{\mathcal{X}}

\begin{document}
\linespread{1.1}

\thispagestyle{empty}

\newgeometry{top=3cm,bottom=2cm}
\begin{center}
	\vspace{-4cm}
		\includegraphics[scale=0.465]{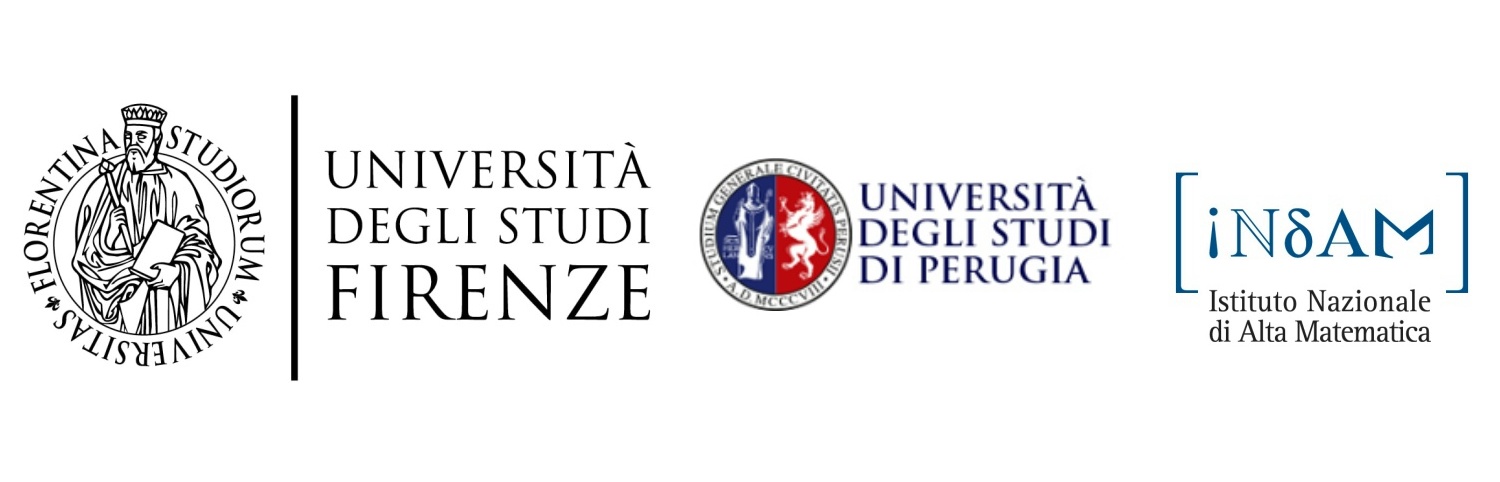}
	\par
	 { %\small %Consorzio  Interuniversitario per l'alta Formazione in Matematica\\
	  Universit\`a di Firenze, Universit\`a di Perugia, INdAM  consorziate nel {\small CIAFM}\\}
	 %Sede amministrativa Universit\`a di Firenze\\
	 \vspace{1mm}
\vskip .8cm
	%\vspace{1mm}
	%\LARGE
	%  UNIVERSIT\`A DEGLI STUDI DI FIRENZE
	\par \vspace{2mm}
	%\hrule
	%\par \vspace{4mm}
	%\large
	%Facolt\`a di Scienze Matematiche, Fisiche e Naturali
	%\\
	%Dipartimento di Matematica e Informatica ``Ulisse Dini''
	%\\
	\large
	\textbf{DOTTORATO DI RICERCA\\
	IN MATEMATICA, INFORMATICA, STATISTICA}\\
		 \vskip.2cm
		 CURRICULUM IN MATEMATICA\\
	CICLO XXXI
	\par \vspace{5mm}
	\large
	
	%Tesi di Dottorato
		 
	 {\bf Sede amministrativa Universit\`a degli Studi di Firenze}\\
	Coordinatore Prof.~Graziano Gentili
	\par \vspace{8mm}
	\huge
	\textbf{Asymptotic analysis of solutions\\related to the \\ game-theoretic $p$-laplacian}
	\par \vspace{5mm}
	
	\large
	Settore Scientifico Disciplinare MAT/05
\end{center}
\par \vspace{10mm}

\normalsize
\hspace{1cm}\begin{minipage}{0.42\linewidth}
	\textbf{Dottorando}:
	\\
	{Diego~Berti}
\end{minipage}
\hspace{2cm}
\begin{minipage}{0.42\linewidth}
	\textbf{Tutore}
	\\
	{Prof.~Rolando~Magnanini}
	
\end{minipage}
\par \vspace{10mm}

\begin{center}
	\begin{minipage}{0.30\linewidth}
		\textbf{Coordinatore}
		\\
		{Prof.~Graziano Gentili}
	\end{minipage}
\end{center}
\par \vspace{9mm}
\begin{center}
	\hrule
	\par \vspace{5mm}
	%\large
	Anni 2015/2018
	
\end{center}
\restoregeometry

%\title{Asymptotic analysis of solutions related to the game-theoretic $p$-laplacian}
%\author{Diego Berti}
%\date{}
%
%\maketitle
\tableofcontents

\pagestyle{fancy}
\fancyhf{}
\fancyhead[LE]{\leftmark}
\fancyhead[RO]{\rightmark}
\fancyhead[RE,LO]{\thepage}

%\setlength{\parskip}{1em}
%\renewcommand{\baselinestretch}{1.5}
%\linespread{1.6}

%\afterpage{
%\thispagestyle{empty}
%}

%\chapter*{Acknowledgments}
%
%I would like to express special thanks to my advisor Professor~Rolando~Magnanini for his precious impulse. His experience and knowledge have been very stimulating and his attentions have been crucial in some occasions.
%
%Also, I want to thank the institution of the ``Dottorato in Matematica'' at Università di Firenze. These three-years PhD program were been a priceless opportunity. In particular, a special mention is due to the coordinator of XXXI cicle, Professor~Graziano~Gentili, as well as the current coordinator Professor~Paolo~Salani.
%
%I would like to express my gratitude to Professor~Shigeru~Sakaguchi. I visited in two different moments the GSIS center at the Tohoku University in Sendai. 
%
%I have also to mention the companions (PhD students or not) with who I have shared the T8 Room. They have helped me or simply shared with me the most part of these last three years. In particular, I mention Simon, Elisa, Mauro, Andrea, Giorgio, Nico, Giulia, Giulio, Tommaso and Simone.
%
%Finally, I thanks my family for the support with special mention for Masoomeh.
\chapter*{Acknowledgments}
Here, I mention those who have foremost contributed to the writing of this work as
the conclusion of my Ph.D.
\par
First of all, I express my gratitude to my advisor Prof. Rolando Magnanini for his constant
 help in the production of this thesis and in the progress of my academic figure.
He first introduced me to the topics and questions I have studied in this thesis and his
competence, passion and experience have been crucial in more than one step.
\par
I want to thank the members of the research group in Firenze, for involving me and
for their precious support through these years. In particular, a quote is due to Prof.
Paolo Salani, Prof. Chiara Bianchini, Prof. Andrea Colesanti and Prof. Elisa Francini.
Also, for having shared most of the steps of the journey with good advice, I thank my
academic brother Dr. Giorgio Poggesi.
\par

The visits to the prestigious Research Center for Pure and Applied Mathematics
 at the Tohoku University in Sendai have been 
extremely significant for all my work. I truly thank Prof. Shigeru Sakaguchi from
his warm hospitality. Exchanges of ideas with him and his research group have been
beneficial. I have appreciated the possibility of establishing contact with some young
researchers from the other side of the world.
\par
I would like to thank Prof. Mikko Parviainen for the hint that improved Lemma
1.23. I met him and his research group during a refreshing and stimulant week of
 the Summer School in Jyväskylä.
\par
I wish to thank the doctorate program in Firenze and the Dipartimento di Matematica
e Informatica “Ulisse Dini” in Firenze. In particular, I mention the director of
my Ph.D. program Prof. Graziano Gentili, for his attention to the needs of doctoral
students. I also thank my Ph.D. program for its essential financial support.
\par
For their priceless understanding, comfort and motivation, I would like to remark
the importance of my companions in the doctoral students room T/8, especially who
 I shared most of the last three years.
\par
Finally, I wish to end these acknowledgments dedicating special thanks to the support
I received during my life by my family and during the last two years by Masoomeh.
Thank you.

\chapter*{Introduction}
\addcontentsline{toc}{chapter}{Introduction}
%\markboth{INTRODUCTION}{INTRODUCTION}
\markboth{}{}
\label{ch:intro}
%%%%%%%%%%%%
%%%%%%%%%%%%
%%Definizione
The {\it game-theoretic} (or {\it normalized} or {\it one-homogeneous}) {\it $p$-laplacian} is defined as
\begin{equation}
\label{p-Laplace}
\Delta_p^G u = \frac1{p}|\nabla u|^{2-p}\De_p u,
\end{equation}
for $p \in (1,\infty)$, where $\De_p u$ is the usual $p$-laplacian 
$$
\De_p u = \mathrm{div}\!\left(|\nabla u|^{p-2}\nabla u\right),
$$
and as
\begin{equation*}
\label{infinity}
\Delta_\infty^G u = \frac{\langle \nabla^2 u \nabla u, \nabla u \rangle}{|\nabla u|^2},
\end{equation*}
in the extremal case $p=\infty$. With these notations, by computing formally the divergence in \eqref{p-Laplace}, $\De_p^G$ can be written for any $p\in [1,\infty]$ as
\begin{equation}
\label{p-decomposition}
\De_p ^G u = \frac1{p} \De u +\left(1 -\frac{2}{p}\right) \De_\infty ^G u,
\end{equation}
where $\De$ is the classical Laplacian.
\par
%%%%%%%%
%%%%%%%%%%%
%%Motivazioni per lo studio dell'operatore
Works by Peres, Schramm, Scheffield and Wilson (see \cite{PSSW, PS}) have emphasized the role of $\De_p^G$ in stochastic differential equations in the context of {\it game theory}. Indeed, equations for $\De_p^G$ appear when one considers the limiting value for vanishing length of steps of certain two-players games called {\it tug-of-war (or TOW) games with noise} in the case $p\in (1,\infty)$ and TOW games in the case $p = \infty$. In this context, it is possible to consider a stochastic game between two opponents, one who wants to maximize and the other who wants to minimize the payoff. Heuristically, the game consists of a combination of random moves (which correspond to noise and are dictated by $\De$) and moves that are orthogonal to the gradient (which correspond to TOW and depend on the operator $\De_\infty^G$). In this context, we also mention works by Manfredi, Parviainen and Rossi, in \cite{MPR-SJMA, MPR-ESAIM, MPR-PAMS, MPR-ASNS}.
\par 
The relation between TOW games and differential equations similar to those involving the game-theoretic $p$-laplacian are considered by Nystr\"om and Parviainen in the context of {\it market manipulation} and {\it option pricing} (see \cite{NP}).
\par
There is also a growing interest for equations involving the game-theoretic $p$-laplacian in connection to numerical methods for {\it image enhancement or restoration} (see \cite{Do} and \cite{BSA}). Typically, for a possibly corrupted image represented by a function $u_0$, it is considered an evolutionary process based on $\De_p^G$ with initial data $u_0$ and homogeneous Neumann boundary conditions. As explained in \cite{Do}, the different choice of $p$ affects the direction in which the brightness evolves; the $1$-homogeneity of $\De_p^G$ ensures that such an evolution does not depend on the brightness of the image. The relation between solutions of parabolic equations and a corresponding parametrized elliptic equation is examined in \cite{BSA} for the classical $p$-laplacian, and can be extended to the case of $\De_p^G$ in hand.

Besides the cited applications, problems for $\De_p^G$ have been recently studied by Attouchi and Parviainen \cite{AP-2018}, Attouchi, Parviainen and Ruosteenoja \cite{APR}, Parviainen and Ruosteenoja \cite{PR}, A. Bj\"orn, J. Bj\"orn and Parviainen \cite{BBP}, Parviainen and V\'azquez \cite{PV}, Banerjee and Garofalo \cite{BG-IUMJ, BG-CPAA}, Does \cite{Do}, Juutinen and Kawohl \cite{JK}, Kawohl, Kr\"omer and Kurtz \cite{KKK} as well as Banerjee and Kawohl \cite{BK} and Jin and Silvestre \cite{JS}.
\\

%%%%%%%%%%%%
%%%%%%%%%%
%%Proprietà dell'operatore
Observe that, having in mind the formal decomposition \eqref{p-decomposition}, when $p=2$ we simply obtain $\De_2^G = \De / 2$ and that, for $p \neq 2 $, $\De_p^G$ can be seen as a proper singular perturbation of $\De / p$. Indeed, one notices that $\De_\infty^G u$ has discontinuous coefficients when $\na u = 0$.
Nevertheless, $\De_p^G$ is uniformly elliptic (in case $p\in (1,\infty)$) and (degenerate) elliptic in the case $p=\infty$.
\par 
It is evident that, for $p\not=2$, $\De^G_p$ is nonlinear. However, $\De_p^G$ is somewhat reminiscent of the lost linearity of the Laplace operator, since it is $1$-homogeneous, that is 
$$
\De_p^G \left( \la u \right) = \la \De_p^G u \ \mbox{ for any } \ \la \in \RR,
$$
 differently from $\De_p$, which instead is $(p-1)$-homogeneous. The nonlinearity of  $\De_p^G$ is indeed due to its non-additivity. Nevertheless, $\De_p^G$ acts additively if one of the relevant summands is constant and, more importantly, on functions of one variable and on radially symmetric functions. We shall see in the sequel that these last properties are decisive for the purposes of this thesis.
\par
Also, when $p \neq 2 $, differently from $\De_p$, $\De_p^G$ is not in divergence form. This fact implies that we cannot apply the standard theory of distributional weak solutions. We need to consider the theory of {\it viscosity solutions}. The main tool of this theory we will use is the \textit{comparison principle}, that we recall and adapt to our purposes in Chapter \ref{ch:explicit}, together with some versions of the \textit{strong maximum principle} and the \textit{Hopf-Oleinik lemma}, which will be used in Chapter \ref{ch:applications}.

%%%%%%%%%
%%%%%%%%%%%
%Presentazione dei problemi
In this thesis, we focus on the connection between asymptotic formulas for solutions of certain game-theoretic $p$-laplacian problems and some geometrical features of the relevant domain. 
\par
We will generally consider a domain $\Om$ in $\RR^N$, with $N\ge 2$, not necessarily bounded, with boundary $\Ga \neq \varnothing$.
We shall consider viscosity solutions $u=u(x,t)$ of the following initial-boundary value problem:
\begin{eqnarray}
&u_t=\De_p^G u \ &\mbox{ in } \ \Om\times(0,\infty), \label{G-heat} \\
&u=0  \ &\mbox{ on } \ \Om\times\{ 0\}, \label{initial} \\
&u=1  \ &\mbox{ on } \ \Ga\times(0,\infty). \label{boundary}
\end{eqnarray}
\par
Also, we shall consider viscosity solutions $u^\ve$ of the one-parameter family of boundary value problems
\begin{eqnarray}
&u^\ve-\ve^2\De_p^G u^\ve=0 \ &\mbox{ in } \ \Om, \label{G-elliptic} \\
&u^\ve=1  \ &\mbox{ on } \ \Ga. \label{elliptic-boundary}
\end{eqnarray}
 Our attention will focus on the asymptotic analysis for small positive values of parameters $t$ and $\ve$.
 
 %%%%%%%
 %%%%%%
 %%Caso lineare:connessione dei due problemi
When $p=2$, the two cases are strongly connected. Indeed, by taking advantage of the linearity of \eqref{G-heat}, one can use a modified {\it Laplace transform}, to obtain that, 
\begin{equation}
\label{eq:intro_modified laplace}
u^\ve(x) = \frac1{\ve^2} \int_{0}^{\infty} u(x,t) e^{-\frac{t}{\ve^2}}\,dt \ \mbox{ for } \ x \in \ol\Om,\ \ve > 0.
\end{equation}
\par
In the case $p\neq 2$ the two problems are no more equivalent. Nevertheless, if $\Om$ has spherical or one-dimensional symmetry, due to the fact that for radial or one-dimensional functions the operator $\De_p^G$ acts linearly, \eqref{eq:intro_modified laplace} still holds true. This observation will be crucial. A proper manipulation of \eqref{eq:intro_modified laplace} will give suitable barriers to estimate the parabolic solution (see Lemma \ref{ball control from above}).
\par
%%%%%%%%%%%%
%%%%%%%
%%Varadhan-type formulas
In what follows, we will describe the main results of this thesis, which are mainly contained in the two papers \cite{BM-JMPA, BM-AA}. The case of problem \eqref{G-heat}-\eqref{boundary} is considered in \cite{BM-JMPA}, whereas \cite{BM-AA} addresses the case of problem \eqref{G-elliptic}-\eqref{elliptic-boundary}.
Under the assumption that $\Om$ merely satisfies the topological assumption $\Ga= \pa\left(\RR^N \setminus \ol\Om\right)$,  in  Theorem \ref{th:parabolic-pointwise} we establish for $p\in (1,\infty]$ the asymptotic profile of the solution of \eqref{G-heat}-\eqref{boundary} for small values of time:
\begin{equation}
\label{eq:intro_BM-JMPA_varadhan}
\lim_{t \to 0^+} 4 t \log u(x,t) = -p'\, d_\Ga(x)^2, \ x\in\ol\Om.
\end{equation}
 Here, by $p'$ we mean the conjugate exponent of $p$, that is $1/p + 1/p' = 1$, for $p\in(1,\infty)$, and $p'=1$, when $p=\infty$. Also, by $d_\Ga(x)$, we mean the distance of the point $x \in \Om$ from the boundary $\Ga$, defined by
$$
d_\Ga(x) = \inf\{|x- y|: y \in \Ga \}, \ x\in\ol\Om.
$$
\par 
Moreover, in Theorem \ref{th:elliptic-pointwise}, we obtain the corresponding formula for the solution of the elliptic problem \eqref{G-elliptic}-\eqref{elliptic-boundary}:
\begin{equation}
\label{eq:intro_BM-AA_varadhan}
\lim_{\ve \to 0^+} \ve \log u^\ve(x) = -\sqrt{p'} \, d_\Ga(x), \ x\in\ol\Om.
\end{equation}
\par
These pointwise asymptotic profiles extend known formulas in the linear case, first obtained by Varadhan by using analytic methods (see \cite{Va-66, Va}). See also \cite{EI}, where Evans and Ishii used arguments pertaining the theory of viscosity solutions, and \cite[Section 10.1]{FW}, for a treatment with probabilistic methods by Freidlin and Wentzell. These formulas are common in the context of {\it large deviations theory}. There, random differential
 equations with small noise intensities are considered. The profiles for small values of $\ve$ and $t$ of the solutions of \eqref{G-elliptic}-\eqref{elliptic-boundary} and \eqref{G-heat}-\eqref{boundary} are respectively related to the behavior of the exit time and to the probability to exit from $\Om$ of a certain stochastic process (see \cite{EI} and \cite{FW}).  
 \par
Formulas \eqref{eq:intro_BM-JMPA_varadhan} and \eqref{eq:intro_BM-AA_varadhan} are obtained by employing barrier arguments based on accurate estimates on radial solutions. In particular, all we need is to control solutions both in a ball or in the complement of a ball. In the elliptic case we are able to compute them in terms of {\it Bessel functions} whereas the parabolic case is more delicate, since we need to properly involve \eqref{eq:intro_modified laplace} and the existence of a global solution of \eqref{G-heat}.
\par
We mention that proper versions of formula \eqref{eq:intro_BM-JMPA_varadhan} are included in a series of works by Magnanini-Sakaguchi in linear cases (see \cite{MS-AM, MS-AN, MS-PRSE, MS-IUMJ, MS-JDE, MS-MMAS}), in certain nonlinear contexts (see \cite{MS-PRSE, MS-AIHP, MS-JDE2}) and concerning initial-value problems  (\cite{MPeS, MPrS}). 
%We mention that proper versions of formula \eqref{eq:intro_BM-JMPA_varadhan} are included in a series of works by Magnanini-Sakaguchi in linear cases (see \cite{MS-AM, MS-AN, MS-PRSE, MS-IUMJ, MS-JDE, MS-MMAS}), in certain nonlinear contexts (see \cite{MS-PRSE, MS-AIHP, MS-JDE2}) and concerning initial-value problems  (\cite{MPeS, MPrS}). Recently, it has been investigated the case of two-phase problems, by Sakaguchi (in \cite{Sa1, Sa2}) and by Cavallina, Magnanini and Sakaguchi in \cite{CaMS}.
\\

%%%%%%%%
%%%%%%%%
%%Formule per le q-medie
A second type of asymptotic formulas that we obtain, which deeply link solutions of  \eqref{G-heat}-\eqref{boundary} or \eqref{G-elliptic}-\eqref{elliptic-boundary} to the geometry of the domain, involve certain statistical quantities called {\it $q$-means}. Given $x\in\Om$, $t>0$ and $q\in (1,\infty]$, the $q$-mean $\mu_q(x,t)$ on $B_R(x)\subset \Om$ of the solution $u$ of \eqref{G-heat}-\eqref{boundary} is the unique real number $\mu$ such that
 $$
 \Vert u (\cdot,t)- \mu\Vert_{L^q\left(B_R(x)\right)} \leq \Vert u(\cdot,t) -\la \Vert_{L^q\left(B_R(x)\right)} \ \mbox{ for any } \ \la \in\RR.
 $$
 Observe that $\mu_q$ generalizes the mean value of $u$ on $B_R(x)$, which is obtained when one chooses $q=2$ in the above definition.
 \par
 Consider a domain of class $C^2$. Suppose  that there exists $y_x \in \Ga$ such that $(\RR^N\setminus\Om)\cap\pa B_R(x)=\{ y_x\}$, where $R= d_\Ga(x)$. Assume that $\ka_1(y_x),\dots,\ka_{N-1}(y_x) < \frac1{R}$, where we have denoted with $\ka_1,\dots,\ka_{N-1}$ the principal curvatures with respect to the inward normal of $\Ga$ at $y_x$, and set
\begin{equation*}
\Pi_\Ga(y_x) =\prod\limits_{j=1}^{N-1}
\Bigl[1-R\,\ka_j(y_x)\Bigr].
\end{equation*}
\par
In Theorem \ref{th:JMPA-asymptotics-qmean} we establish that, for any $q \in(1,\infty)$ and $p\in (1,\infty]$,
%$$
%\lim_{t\to 0^+}
% \left(\frac{R^2}{t}\right)^{\frac{N+1}{4(q-1)}}\mu_{q}(x,t) =
% C_{N,p,q}\,\left\{ \Pi_\Gamma(y_x)\right\}^{-\frac{1}{2(q-1)}},
%$$
%This formula holds for any $p\in(1,\infty]$ and $q\in(1,\infty)$. The positive constant $C_{N,p,q}$ will be specified in Theorem \ref{th:JMPA-asymptotics-qmean}.
%\par
%In the elliptic case, we consider the $q$-mean of the solution of \eqref{G-elliptic}-\eqref{elliptic-boundary} and, for the same values of $p$ and $q$, we compute:
%$$
%\lim_{\ve\to 0^+}
%\left(\frac{R}{\ve}\right)^{\frac{N+1}{2(q-1)}}\mu_{q,\ve}(x)=
%   \tilde{C}_{N,p,q}\,\left\{\,\Pi_\Gamma(y_x)\right\}^{-\frac{1}{2(q-1)}},
%  $$
%The value of $\tilde{C}_{N,p,q}$ can be found in Theorem \ref{th:AA-qmean}.
%\par
%In Theorems \ref{th:JMPA-asymptotics-qmean} and \ref{th:AA-qmean} also the extremal case in which $q=\infty$ will be treated obtaining that
%$$
%\lim_{t\to 0^+} \mu_\infty(x,t) = \lim_{\ve\to 0^+} \mu_{\infty,\ve}(x) = \frac1{2}.
%$$
 \begin{equation}
 \label{eq:introJMPA-qmean}
\lim_{t\to 0^+} \left(\frac{R^2}{t}\right)^{\frac{N+1}{4(q-1)}}\!\!\!\!\mu_q(x,t)= C_{N,p,q}\left\{\Pi_\Gamma(y_x)\right\}^{-\frac{1}{2(q-1)}},
\end{equation}
where $C_{N,q,p}$ is a positive constant only depending on the labelled parameters.
% \begin{multline}
% \label{eq:introJMPA-qmean}
%\lim_{t\to 0^+} \left(\frac{R^2}{t}\right)^\frac{N+1}{4(q-1)}\!\!\!\!\mu_q(x,t)=
%\\
%\left\{\frac{N!\,\int_0^\infty \Erfc(\si)^{q-1} \si^\frac{N-1}{2} d\si}{\,\Ga \left( \frac{N+1}{2} \right)^2}\right\}^\frac1{q-1} \left\{ p'^{\frac{N+1}{2}}\,\Pi_\Gamma(y_x)\right\}^{-\frac{1}{2(q-1)}},
%\end{multline}
In the extremal case $q=\infty$ and for any $p\in (1,\infty]$, we obtain that
$$
\mu_\infty(x,t) \to \frac1{2} \ \mbox{ as } t \to 0^+.
$$
\par
Analogously, from an accurate improvement of barriers in the case of smooth domains we compute the asymptotic profile of $\mu_{q,\ve}$, the $q$-mean of $u^\ve$ on $B_R(x)$. In fact, in Theorem \ref{th:AA-qmean} we show that
\begin{equation}
  \label{eq:introAA-qmean}
\lim_{\ve\to 0^+}
  \left(\frac{R}{\ve}\right)^{\frac{N+1}{2(q-1)}}\mu_{q,\ve}(x)=
  \widetilde{C}_{N,p,q}\left\{ \Pi_\Ga(y_x)\right\}^{-\frac{1}{2(q-1)}}.
\end{equation}
%\begin{equation}
%  \label{eq:introAA-qmean}
%\lim_{\ve\to 0^+}
%  \left(\frac{\ve}{R}\right)^{-\frac{N+1}{2(q-1)}}\mu_{q,\ve}(x)=
%  \left\{\frac{2^{-\frac{N+1}{2}}N!}{(q-1)^{\frac{N+1}{2}}\Ga\left(\frac{N+1}{2}\right)}\right\}^{\frac1{q-1}}\left\{p'^{\frac{N+1}{2}}\,
%\Pi_\Ga(y_x)\right\}^{-\frac{1}{2(q-1)}},
%\end{equation}
Also in this case, we obtain that
$$
\mu_{\infty,\ve}(x) \to \frac{1}{2} \ \mbox{ as } \ \ve \to 0^+.
$$
\par
%Both formulas \eqref{eq:introJMPA-qmean} and \eqref{eq:introAA-qmean} extend formulas obtained in the linear case by Magnanini and Sakaguchi substantially when $p=2=q$. In particular, in \cite{MS-IUMJ}, the authors consider the case of the {\it heat content}. They have shown that
%\begin{equation}
%\label{product-curvatures}
%\lim_{t\to 0^+} t^{-\frac{N+1}{4}} \int_{B_R(x)} u(z,t)\,dz= 
%\frac{c_N\,R^{\frac{N-1}{2}}}{\sqrt{\Pi_\Ga(y_x)}},
%\end{equation}
%where $\Pi_\Ga(y_x)$ is given by \eqref{eq:Pi-gamma}. In \cite{MS-AM}, the integral in \eqref{product-curvatures} is referred to as the {\it heat content} of the ball $B_R(x)$.
%\par
% in \cite[Theorem 2.3]{MS-AM}, in the case $p=2$ for problem \eqref{G-elliptic}-\eqref{elliptic-boundary}, the following formula involving the mean value of $u^\ve$ was established: 
%\begin{equation}
% \label{eq:MS}
%\lim_{\ve\to 0^+}
%\left(\frac{R}{\ve}\right)^{\frac{N-1}{2}}\dashint_{\pa B_R(x)}u^\ve(y)\,dS_y=\frac{c_N}{\sqrt{\Pi_\Ga(y_x)}};
%\end{equation}
%$c_N$ is a constant that can be derived from \cite[Theorem 2.3]{MS-AM}.
We emphasize that \eqref{eq:introJMPA-qmean} and \eqref{eq:introAA-qmean} generalize, to each $p\in(1,\infty]$, and extend, to any $q \in(1,\infty)$, known formulas in the linear case, for $p = q = 2$. In fact, we recall that in \cite[Theorem 4.2]{MS-PRSE}, it has been given the following asymptotic formula for the so-called {\it heat content} on $B_R(x)$:
\begin{equation}
\label{MS-PRSE}
\lim_{t\to 0^+}
t^{-\frac{N+1}{4}}\int_{B_R(x)}u(z,t)\,dz=
\frac{C_N R^{\frac{N-1}{2}}}{ \sqrt{\Pi_\Gamma(y_x)}},
\end{equation}
where $u$ is the solution of the heat equation satisfying \eqref{initial} and \eqref{boundary}. A normalization makes apparent the connection of the heat content to $\mu_2(x,t)$. See also  \cite[Theorem 2.3]{MS-AM} for a similar formula in the elliptic case. Similar formulas for the case $q=2$ in nonlinear settings can be found in \cite{MS-PRSE} for the evolutionary $p$-Laplace equation, and in \cite{MS-JDE2}, for a class of non-degenerate fast diffusion equations. 
\par
It is worth noting that \eqref{eq:introJMPA-qmean} and \eqref{eq:introAA-qmean} are essentially based on two ingredients. The first is the geometrical lemma \cite[Lemma 2.1]{MS-PRSE} which determines the behavior of $\cH_{N-1}\left(\Ga_s \cap B_R(x)\right)$ for vanishing $s > 0$ in terms of the function $\Pi_\Ga$. Here, $\Ga_s = \{ x\in\Om: d_\Ga(x) =s \}$, for $s > 0$ and $\cH_{N-1}$ is the $(N-1)$-Hausdorff measure. The second ingredient is the construction of 
 %%%%%%
 %%%%%
 %%Stime uniformi delle Varadhan-type formulas
 sharp uniform estimates of Varadhan-type formulas \eqref{eq:intro_BM-JMPA_varadhan} and \eqref{eq:intro_BM-AA_varadhan}, which are novelty even in the case $p=2$.
\par
In the case $\Om$ is of class $C^{0,\om}$, that is $\Ga$ is locally a graph of a continuous function with modulus of continuity controlled by $\om$ (see Section \ref{sec:uniform}), we provide in Theorem \ref{th:uniform-JMPA} the following estimate:
 \begin{equation}
 \label{eq:introBM-JMPA-uniform}
 4 \,t \log u(x,t) + p'\, d_\Ga(x)^2 = O\left(t \,\log \psi_\om(t)\right),
 \end{equation}
 for $t\to 0^+$, uniformly on every compact subset of $\ol\Om$. Here, $\psi_\om(t)$ is a function that depends on $\om$ and is positive and vanishes as $t \to 0^+$. In particular, if $\Om$ is smooth enough, \eqref{eq:introBM-JMPA-uniform} gives a sharp estimate of the rate of convergence in \eqref{eq:intro_BM-JMPA_varadhan}. In fact, if for example $\Om$ is a $\al$-H\"older domain, for some $0 < \al < 1$, we obtain that the right-hand side of \eqref{eq:introBM-JMPA-uniform} is $O\left(t \log t\right)$, as $t\to 0^+$.
 \par
 In the elliptic case, explicit barriers are available, 
we obtain more accurate uniform estimates. In fact, in Theorem \ref{th:uniform-elliptic}, we prove that 
 \begin{equation}
 \label{eq:introBM-AA-uniform1}
 \ve\log u^\ve(x) + \sqrt{p'}\,d_\Ga(x) = 
 \begin{cases}
 \displaystyle 
  O\left(\ve\right) \ \mbox{ if } \ p=\infty,\\
 \displaystyle 
  O\left( \ve \log \ve \right) \ \mbox{ if } \ p \in (N, \infty),
  \end{cases}
  \end{equation}
  as $\ve \to 0^+$, on every compact subset of $\ol\Om$.  In the case of a domain $\Om$ of class $C^{0,\om}$, it holds instead that
  \begin{equation}
  \label{eq:introBM-AA-uniform2}
  \ve\log u^\ve(x) +\sqrt{p'}\, d_\Ga(x)=
  \begin{cases}
  \displaystyle 
  O\left( \ve \log |\log \psi_\om(\ve)|\right)\ \mbox{ if } \ p=N,
  \\
  \displaystyle
  O\left(\ve \log \psi_\om(\ve)\right) \ \mbox{ if } \ p\in(1,N),
  \end{cases}
  \end{equation}
  for $\ve \to 0^+$, uniformly on every compact subset of $\ol\Om$.
  \par
  We observe that the presence of the threshold for the exponent $p$ in this last formula seems to be connected to the integrability of the global solution of \eqref{G-heat} with respect to the variable $t \in (0,\infty)$. This suggests that even in the parabolic case we may expect this kind of behavior. We were not able to prove it so far.
  \par
 Finally, notice that, by using comparison results, formulas \eqref{eq:introBM-JMPA-uniform} and \eqref{eq:introBM-AA-uniform1}-\eqref{eq:introBM-AA-uniform2} can be easily extended to the case of a prescribed non-constant data on the boundary. See Corollaries \ref{positive boundary data} and \ref{cor:elliptic positive boundary data}.
\\
\par
%%%%%
%%%%%%%%%
%%Applicazione alla simmetria
In Chapter \ref{ch:applications}, the obtained Varadhan-type formulas and formulas for $q$-means will find applications to geometric and symmetry results. The linearity of $\De$  was used in \cite{MS-AM} to derive radial symmetry of compact {\it stationary isothermic surfaces}, that is those level surfaces of the temperature which are invariant in time. In Chapter \ref{ch:applications}, we will extend this type of result to the case $p\ne 2$. In the case $p=2$, it was shown that the mean values $\mu_2(x,t)$ or   $\mu_{2,\ve}(x)$ do not depend on $x$ if this lies on a stationary isothermic surface, and hence, for instance, \eqref{MS-PRSE} gives that
\begin{equation*}
\Ga \ni y \mapsto \Pi_\Ga(y) \mbox{ is constant. }
\end{equation*} 
The radial symmetry then ensues from Alexandrov's Soap Bubble Theorem for Weingarten surfaces (see \cite{Al}). 
\par
 For $p\not=2$, this approach is no longer possible. However, when $\Ga$ is compact, an approach based on the {\it method of moving planes} (see \cite{Ser}, as in \cite{MS-AIHP} and \cite{CMS-JEMS}) is still feasible. We also treat a case in which $\Ga$ is unbounded, by using the {\it sliding method} (see \cite{BCN}), as in \cite{MS-IUMJ}, \cite{MS-AIHP}, \cite{MS-JDE} and \cite{Sa-INDAM} to obtain that $\Ga$ must be a hyperplane. We stress that a crucial step to apply the cited methods in our cases is the application of the (classical) strong comparison principle in a suitable subset (as done in \cite{AG} or in \cite{BK}), which is determined by an application of the strong maximum principle and Hopf lemma for viscosity solutions.
%For $p\not=2$, this approach is no longer possible. However, when $\Ga$ is compact, an approach based ofn the {\it method of moving planes} (see \cite{Ser}) when $\Ga$ is compact, as in \cite{MS-AIHP} and \cite{CMS-JEMS} is still feasible. We also treat a case in which $\Ga$ is unbounded, by using the {\it sliding method} (see \cite{BCN}), as in \cite{MS-IUMJ} to obtain that $\Ga$ must be a hyperplane. We stress that a crucial step to apply the cited methods in our cases is the application of the (classical) strong comparison principle in a suitable subset (as done in \cite{CMS}, \cite{AG} or in \cite{BK}), which is determined by an application of the strong maximum principle and Hopf lemma for viscosity solutions.
\\

%%%%%%%
%%%%%5
%%%%Plan of the work
We conclude this introduction by a summary of this thesis.
Chapter \ref{ch:viscosity solutions} recalls those technical tools of the theory of viscosity solutions which we will use in the remaining chapters. 
\par
In Chapter \ref{ch:explicit}, we consider the cases of symmetric domains, in which $\De_p^G$ acts as a linear operator. This allows us to deal with explicit solutions of \eqref{G-heat}-\eqref{boundary} and \eqref{G-elliptic}-\eqref{elliptic-boundary} and to compute their asymptotic profiles. The explicit formulas give sharp estimates that will be crucial to control the case of generic domains.
\par
Formulas of Varadhan-type, that is \eqref{eq:intro_BM-JMPA_varadhan} and \eqref{eq:intro_BM-AA_varadhan}, are collected in Chapter \ref{ch:asymptotics i}. There, also their uniform sharp versions  \eqref{eq:introBM-JMPA-uniform} and \eqref{eq:introBM-AA-uniform1}-\eqref{eq:introBM-AA-uniform2} shall be given.
\par
In Chapter \ref{ch:asymptotics ii}, we provide our asymptotic formulas for $q$-means  \eqref{eq:introJMPA-qmean} and \eqref{eq:introAA-qmean}.
\par
Finally, Chapter \ref{ch:applications} contains a few applications of the formulas derived in Chapters \ref{ch:asymptotics i} and \ref{ch:asymptotics ii}. In particular, we provide generalization of symmetry results present in the literature.

%\afterpage{
%\thispagestyle{empty}
%}
%\afterpage{\blankpage}

%%%%%%
%%%%%%
%Viscosity solutions
\chapter[Preliminaries]
{Preliminaries on the theory of	 viscosity solutions}
\label{ch:viscosity solutions}
 
 Several results of this thesis are based on some important properties of viscosity solutions. In the next chapters  we shall use comparison principles and we shall apply strong maximum principles and the Hopf-Oleinik Lemma. Since $\De_p^G$ has discontinuous coefficients, standard results for viscosity solutions cannot be directly applied, but they must be adapted. In this chapter we describe how, by pointing out the significant references.
 
 A recent summary on aspects of viscosity solutions of our interest, which includes a quite complete list of references, is a  dedicated chapter in \cite{Gi} (where the theory is instrumental to the study of {\it surface evolution equations}). We adopt that approach. Useful references on definitions and relevant properties of viscosity solutions are also the classical surveys \cite{CGG, CIL}, besides the beginner's guide \cite{Ko}. For more recent and specific works on $\De_p^G$, where viscosity solutions are adopted, we give the following (not complete) list of publications: \cite{AP-2018, APR, BG-IUMJ, BG-CPAA,  BM-AA, BM-JMPA, Do, JK, KKK, MPR-PAMS}.

In Section \ref{sec:definitions}, we shall begin with definitions and relevant properties of viscosity solutions of general singular differential equations. In this more general context, we shall also state those theorems (from \cite{BA-DaL, DaL, GGIS, Sat}) which we shall apply to our cases, in Section \ref{sec:game viscosity}. 
\par 
In Section \ref{sec:game viscosity}, in the specific case of the game-theoretic $p$-laplacian, we shall collect those results which we will use in the next chapters of this thesis. In particular, in Subsection \ref{ssec:game comparisons}, for both equations \eqref{G-heat} and \eqref{G-elliptic}, we will give comparison principles (see Corollaries \ref{cor:parabolic comparison} and \ref{cor:elliptic comparison}) as well as corresponding strong maximum principles (see Corollaries \ref{cor:parabolic maximum} and \ref{cor:elliptic maximum}).

Finally, in Subsection \ref{ssec:hopf lemma} we shall extend a sharp version of Hopf-Oleinik lemma (obtained by Mazya et al., in \cite{ABMMZ}), to viscosity solutions of \eqref{G-heat}-\eqref{boundary} and \eqref{G-elliptic}-\eqref{elliptic-boundary} (see Corollaries \ref{cor:parabolic hopf-oleinik} and \ref{cor:elliptic hopf-oleinik}). 

%{\color{red}
% In Section \ref{sec:definitions}, we shall begin with definitions and relevant properties of viscosity solutions, especially applied to our cases.
%In Section \ref{sec:comparison principles}, we shall collect maximum and comparison principles and a suitable Hopf-Oleinink Lemma.  The contributions we report are contained in the papers \cite{BA-DaL}, \cite{DaL}, \cite{GGIS},\cite{Sat}, concerning suitable nonlinear parabolic and elliptic equations, and \cite{ABMMZ}, for the sharp version of Hopf-Oleinik Lemma.
%}

\section[Viscosity solutions]{Viscosity solutions of singular differential equations}
\label{sec:definitions}

To start with, we introduce some definitions on arbitrary functions, which we use in this chapter.  Let $f$ be a function on a metric space $\cX$, with values in $\RR$. We recall that a function $f$ is called {\it lower semicontinuous} if
\begin{equation*}
\label{lower semicontinuity}
\liminf_{y\to x} f(y) \geq f(x)
\end{equation*}
and that a function $f$ is called {\it upper semicontinuous} if $-f$ is lower semicontinuous. 
\par 
For a function $f$, the {\it lower semicontinuous envelope} $f_*$ is defined by 
$$
f_*(z)= \lim_{\de\to 0^+} \inf\{f(z'): z'\in B_\de(z)\subset \cX\},
$$
and the {\it upper semicontinuous envelope} $f^*$ is defined by $f^*=-(-f)_*$. Note that the previous definition differs from that of $\liminf$, since the infimum is taken upon the whole ball $B_\de(z)$. Of course, if $f$ is continuous, $f$ and its envelopes coincide. 
%{
%\color{blue}
%\par
%For a function $f$, the {\it lower semicontinuous envelope} $f_*$ is defined by 
%$$
%f_*(z)= \lim_{\de\to 0^+} \inf\{f(z'): z'\in B_\de(z)\setminus\{z\}\subset \cX\},
%$$
%and the {\it upper semicontinuous envelope} $f^*$ is defined by $f^*=-(-f)_*$. 
%%{
%%\color{magenta}
%%Note that the previous defintion differs from that of $\liminf$, since the infimum is taken upon the whole ball $B_\de(z)$. Moreover, of course, if $f$ is continuous, $f$ and its envelopes coincide. 
%%}
%We can show that $f_*$ is lower semicontinuous and $f^*$ is upper semicontinuous. Moreover, of course, if $f$ is continuous, $f$ and its envelopes coincide.
%}

 For $N\geq 1$, let $\cS^N$ be the linear space of $N\times N$ symmetric matrices. An operator $F:\Om\times(0,\infty)\times\RR\times\left(\RR^N\setminus\{0\}\right)\times \cS^N \to \RR$, is called {\it (degenerate) elliptic} if it satisfies
\begin{equation*}
\label{degenerate ellipticity}
F(x,t,u,\xi,X) \geq F(x,t,u,\xi, Y) \ \mbox{ if } \ X\le Y,
\end{equation*}
for any $(x,t,u,\xi)\in \Om\times(0,\infty)\times \RR\times\left(\RR^N\setminus\{0\}\right)$ and $X,Y \in \cS^N$.
Here with $X\leq Y$ we mean that $\lan (X-Y) \eta, \eta\ran \leq 0$, for any $\eta\in\RR^N$.
\par
We give the following definitions (see \cite[Chapter 2]{Gi}). A lower semicontinuous function, $u:\Om\times(0,\infty)\to\RR$,  is a {\it viscosity subsolution} of 
\begin{equation}
\label{generic parabolic equation}
u_t+ F\left(x, t, u, \na u, \na^2 u\right) = 0 \ \mbox{ in }\ \Om,
\end{equation}
if for any $(x,t,\phi)\in \Om\times(0,\infty)\times C^2\left(\Om\times (0,\infty)\right)$, such that $u-\phi$ attains its maximum at $(x,t)$, it holds that
\begin{equation}
\label{generic parabolic subsolution}
\phi_t(x,t)+
F_* \left(x, t, u(x,t), \na \phi(x,t), \na^2 \phi (x,t)\right) \leq 0.
\end{equation}
\par
Analogously, an upper semicontinuous function, $v:\Om\times(0,\infty)\to\RR$, is called a {\it viscosity supersolution} of \eqref{generic parabolic equation}, if for any $(x,t,\psi)\in\Om\times(0,\infty)\times C^2\left(\Om\times (0,\infty)\right)$ such that $v-\psi$ attains its minimum at $(x,t)$, it holds that
\begin{equation}
\label{generic parabolic supersolution}
\psi_t(x,t)+
F^* \left(x, t, v(x,t), \na \psi(x,t), \na^2 \psi (x,t)\right) \geq 0.
\end{equation}
\par
Finally, we say that a function $u$ is a {\it viscosity solution} of \eqref{generic parabolic equation} if it is both a viscosity subsolution and a viscosity supersolution. 
%\par 
%\begin{rem}[Localization property]
%\label{localization property}
%{\rm
%The notation of viscosity solution can be localized. If $u$ is a viscosity subsolution in a neighborhood of each point of $\Om\times(0,\infty)$, then it is a subsolution (of the same equation) in $\Om\times(0,\infty)$. 
%\par
%If $u$ is a viscosity subsolution in $\Om\times (0,\infty)$, then it is a subsolution (of the same equation) in every open set of $\Om\times (0,\infty)$. 
%\par
%The same is obviously valid also if we replace subsolutions with supersolutions.
%}
%\end{rem}
\par 
The next lemma affirms that the theory is consistent with the classical definition of solutions. The proof is straightforward.

\begin{lem}[Consistency]
\label{consistency}
Assume that $F$ is (degenerate) elliptic.
\par
 Let $u\in C^2\left(\Om\times(0,\infty)\right)$ be such that
$$
F_*\left(x,t, u, \na u, \na^2 u\right) \leq u_t \leq F^*\left(x,t, u, \na u, \na^2 u\right),
$$
for every $(x,t)\in \Om\times (0,\infty)$, then $u$ is a viscosity solution of \eqref{generic parabolic equation}.
\end{lem}

For our purposes, we will need the following extension lemma, that claims  that, for smooth functions, it suffices to check the definitions away from isolated critical points.

\begin{lem} [Extension]
\label{extension lemma}
Assume that $F$ is (degenerate) elliptic. 
\par
Let $u\in C^2\left(\Om\times(0,\infty)\right)$ such that
\begin{enumerate}[(i)]
\item $(x_0 , t) \in \Om \times (0,\infty)$ is the unique point in $\Om\times (0,\infty)$ such that
$$
\na u (x_0,t) = 0.
$$
\item In $\left( \Om \setminus \{x_0\} \right) \times (0,\infty)$, it holds that
 $$
F_*\left(x,t, u, \na u, \na^2 u\right) \leq u_t \leq F^*\left(x,t, u, \na u, \na^2 u\right).
$$
\end{enumerate}
\par
Then $u$ is a solution of \eqref{generic parabolic equation} in $\Om \times (0,\infty)$.
\end{lem}

\begin{proof}
Take a sequence of points $y_n\in\Om\setminus\{x_0\}$ such that $y_n\to x_0$ as $n\to\infty$.  From our assumption on $u$, we have that both of the following inequalities hold at $(y_n,t)$:
$$
u_t + F_*\left(x, t, u, \na u, \na^2 u \right) \leq 0,
$$
and
$$
u_t + F^*\left(x, t, u, \na u, \na^2 u \right)\geq 0.
$$
Now, since $F_*$ is lower semicontinuous and $F^*$ is upper semicontinuous we have that
\begin{multline*}
u_t ( x_0 , t ) +F_*\left(x_0, t, u(x_0), \na u(x_0), \na^2 u(x_0) \right) \le
 \\
 \liminf_{n\to \infty} \left\{u_t ( y_n , t ) +F_*\left(y_n, t, u(y_n), \na u(y_n), \na^2 u(y_n) \right)\right\} \leq 0
\end{multline*}
and 
\begin{multline*}
u_t ( x_0 , t ) +F^*\left(x_0, t, u(x_0), \na u(x_0), \na^2 u(x_0) \right) \ge
  \\
\limsup_{n\to \infty} \left\{ u_t ( y_n , t ) +F^*\left(y_n, t, u(y_n), \na u(y_n), \na^2 u(y_n) \right)\right\} \geq 0.
\end{multline*}
The claim follows, thanks to Lemma \ref{consistency}.
\end{proof}

\begin{rem}
{
\rm
Obvious adjustments of definitions and Lemmas \ref{consistency} and \ref{extension lemma} are given in the case of singular elliptic differential equations (see for example \cite{CIL}).
 }
\end{rem}

%For the elliptic case we have the following definitions (for example, from \cite{CIL}).  We say that a lower semicontinuous function $u:\Om\to \RR$ is a viscosity subsolution of
%\begin{equation}
%\label{generic elliptic equation}
%F\left(x, u, \na u, \na^2 u\right)= 0
%\end{equation}
%in $\Om$, if for every $(x,\varphi) \in \Om \times C^2\left(\Om\right)$ such that $u - \varphi$ attains its maximum at $x$, it holds
%\begin{equation*}
%F_*\left(x, u(x), \na\varphi(x), \na^2 \varphi (x) \right) \leq 0;
%\end{equation*}
%An upper semicontinuous function $v:\Om\to\RR$ is a viscosity supersolution of \eqref{generic elliptic equation} if for every $(x,\psi) \in \Om \times C^2\left(\Om \right)$ such that $v - \psi$ attains its minimum at $x$, then
%\begin{equation*}
%F^*\left(x, v(x), \na \psi (x), \na^2 \psi(x) \right) \geq 0.
%\end{equation*}
%Finally, a viscosity solution of \eqref{generic elliptic equation} is a function that is both viscosity subsolution and viscosity supersolution.
%\par 
%For solutions of \eqref{generic elliptic equation} are valid same conclusions of Lemmas \ref{consistency} and \ref{extension lemma}, with obvious adjustments.

\subsection{Comparison principles}
\label{ssec:generic principles}

For our purposes, we propose quite general comparison results, in a not necessarily bounded domain $\Om$. The relevant assumptions apply to the differential equations \eqref{G-heat} and \eqref{G-elliptic}.
\par 
As a \textit{modulus of continuity} we mean a continuous function $\om:[0,\infty)\to [0,\infty)$, such that $\om(0)=0$. If $X \in \cS^N$, by $|X|$ we intend the operator norm of $X$ on $\RR^N$.

\begin{thm}[{\cite[Theorem 2.1]{GGIS}}]
\label{GGIS comparison}
Assume that $F:\RR^N\setminus\{0\}\times \cS^N\to\RR$ is continuous and (degenerate) elliptic. In addition, let $F$ satisfy the following properties:
 \begin{enumerate}[(a)]
\item it holds that
$$
-\infty < F_*\left(0,0\right) = F^*\left(0,0\right) < +\infty;
$$

\item
 for every $R>0$,
$$
\sup\{|F\left(\xi, X\right)|: 0< |\xi|\leq R,\  |X|\leq R\} < \infty.
$$
\end{enumerate}
\par 
Let $u$ and $v$ be, respectively, a subsolution and a supersolution of 
$$
u_t + F\left(\na u, \na^2 u\right)=0
$$
 in $\Om\times (0,\infty)$. Assume that
\begin{enumerate}[(i)]
\item $u(x,t) \leq K(|x|+1)$, $v(x,t) \geq -K (|x|+1)$, for some $K>0 $  independent of $(x,t)\in\Om\times (0,\infty)$;
\item there is a modulus $\om$ such that 
$$
u^*(x,t) - v_*(y,t) \leq \om(|x-y|),
$$
for all  $(x,y,t)\in \pa\left(\Om\times\Om\right)\times(0,\infty)\cup \left(\Om\times\Om\right)\times\{0\}$;
\item $u^*(x,t) - v_*(y,t) \leq K(|x-y| + 1)$ on $\pa\left(\Om\times\Om\right)\times(0,\infty)\cup \left(\Om\times\Om\right)\times\{0\}$, for some $K>0$ independent of $(x,y,t)$.
\end{enumerate}
Then, it holds that 
$$
u^*\leq v_* \ \mbox{ on } \ \ol\Om\times (0,\infty).
$$
\end{thm} 

\begin{rem}
\label{assumptions C}
{\rm
Note that (ii) is equivalent to the condition $u^* \leq v_*$ on $\pa\Om\times(0,\infty)$ and that (i) and (iii) are unnecessary, if $\Om$ is bounded. 
\par
It is also evident that, if $u^*$ and $v_*$ are bounded, then (i) and (iii) are satisfied.
}
\end{rem}

We now state a general comparison principle for elliptic equations, which is a corollary of \cite[Theorem 2.2]{Sat}, that treats general equations.

\begin{thm}
\label{sato comparison}
Let $F:\RR^N\setminus\{0\}\times \cS^N\to \RR$ be continuous, (degenerate) elliptic and such that:
\begin{enumerate}[(a)]
\item
$$
-\infty < F_*\left(0,0\right) = F^*\left(0,0\right) < +\infty;
$$

\item
 for every $R>0$,
$$
\sup\{|F\left(\xi, X\right)|: 0< |\xi|\leq R,\,|X|\leq R \} < \infty;
$$
\item for every $R>\rho>0$, there exists a modulus $\si$ such that
$$
|F\left(x,r,\xi,X\right)-F\left(x,r, \eta, X\right)| \leq \si(|\xi-\eta|),
$$
for all $x\in\Om$, $r\in\RR$, $\rho \leq |\xi|, |\eta|\leq R$, $|X|\leq R$.
\label{f.5}
\end{enumerate}
\par 
Let $u$ and $v$ be, respectively, a subsolution and a supersolution of 
$$
u + F\left( \na u, \na^2 u\right) =0 \ \mbox{ in } \ \Om.
$$
\par 
Moreover, suppose that
\begin{enumerate}[(i)]
\item $u(x) \leq K(|x|+1)$, $v(x) \geq -K (|x|+1)$, for some $K>0 $  independent of $x\in\Om$;
\item there is a modulus $\om$ such that 
$$
u^*(x) - v_*(y) \leq \om(|x-y|), \ \mbox{ for all } \ (x,y)\in \pa\left(\Om\times\Om\right);
$$
\item $u^*(x) - v_*(y) \leq K(|x-y| + 1)$ on $\pa\left(\Om\times\Om\right)$, for some $K>0$ independent of $(x,y)\in\pa\left(\Om\times \Om\right)$.
\end{enumerate}
\par
Then, it holds that
$$
u^*\leq v_*\ \mbox{ on } \ \Om.
$$
\end{thm}

\begin{rem}
\label{assumptions on F}
{\rm
Observe that assumptions $(a)$ and $(b)$ of Theorem \ref{sato comparison} are fulfilled, if both $F_*$ and $F^*$ are continuous in their variables. In particular, the supremum in $(b)$ is less than or equal to
$$
\sup\{|F_*(\xi,X)| + |F^*(\xi,X)|: |\xi|,\,|X| \leq R\}.
$$
}
\end{rem}

\subsection{Strong maximum principles}
\label{ssec:maximum}

In Chapter \ref{ch:applications} we shall use the strong maximum principle for \eqref{G-heat} and \eqref{G-elliptic}.  Here, we state results for a quite large class of differential operators. 
\par
We start with the next theorem, which can be seen as a corollary of \cite[Corollary 2.3]{DaL}.

 \begin{thm}
 \label{dalio maximum}
 Assume that $F:\RR^N\setminus\{0\} \times \cS^N\to\RR$ is lower semicontinuous and (degenerate) elliptic. Moreover, assume that $F$ satisfies the following requirements:
 \begin{enumerate}[(a)]
% \item for any  for any, $\xi\in\RR^N\setminus\{0\}$ and $X,Y\in\cS^N$, it holds that
%  $$
%  F \left(x,t,r,\xi,X\right)\leq F\left(x,t, s,\xi, Y\right)  \mbox{ if } \ s\geq r \ \mbox{ and } \ Y\leq X.
% $$ 
 \item there exists $\rho_0 > 0 $ such that , for any choice of $0 < |s|, |\xi| <\rho_0$, there exists $\ga_0\geq 0$ such that
 $$
 s + F\left(\xi, I-\ga \xi\otimes \xi\right) >0
 $$
 holds for every $\ga\geq \ga_0>0$;
 \item for all $\eta >0 $ there exist a function $\varphi :(0,1) \to (0,\infty)$, $\ve_\eta>0$ and $\ga_0 \geq 0$ such that for all $\la \in(0, \ve_\eta]$ and $\ga > \ga_0$,
 $$
 \la s + F_*\left(\la \xi, \la \left(I - \ga \xi\otimes \xi \right)\right) \geq \varphi (\la) \left[s+F_*\left(\xi,\left( I -\ga \xi\otimes \xi\right)\right)\right]
 $$
 holds for all  $0< |\xi|\leq \eta$, $|s|\leq \eta$.
 \item there exists $\de_0>0$ such that	
 $$
 1+F_*\left(0, \de I\right) >0 
 $$
 for all $0< \de <\de_0$.
 \item for all $\eta >0$ there exist $\varphi:(0,1) \to (0,\infty)$, $\ve_\eta>0$ such that for each $K>0$ and for each $\la\in (0,\ve_\eta]$, then
 \begin{equation*}
 \la s + F_*\left(2 K \la (y-x), 2K \la I\right) \geq\\ \varphi (\la) \left[s + F_*\left(2K (y-x), 2K I\right)\right],
 \end{equation*}
 holds for all $-\eta \leq s\leq 0$.
 \end{enumerate}
 \par
 Let $u$ be a viscosity subsolution of \eqref{generic parabolic equation}
% $$
% u_t + F\left(x, t, u, \na u, \na^2 u\right) = 0
% $$
 in $\Om\times (0,\infty)$.  Suppose that $u$ achieves a  maximum at $(x_0,t_0)\in \Om\times (0,\infty)$. 
 \par
 Then $u$ is constant on the set of all points, which can be
connected to $(x_0,t_0)$ by a simple continuous curve in $\Om\times(0,t_0)$ along which the $t$-coordinate is nondecreasing.
 \end{thm}
 
 \begin{rem}
  \label{strong minimum principle}
  {
\rm
Note that a {\it strong minimum principle} also holds for \eqref{generic parabolic equation}, once the obvious necessary changes in Theorem \ref{dalio maximum} are made.
 }
 \end{rem}
%Indeed, if $v$ is a supersolution of \eqref{generic parabolic equation} and $\min_{\Om\times(0,\infty)} v = v(x,t)= \te$, for some $\te \in\RR$, then $w = \te - v $ satisfies the assumptions of Theorem \ref{dalio maximum}.

We now give the elliptic version of the strong maximum principle, which will be applied to solutions of equation \eqref{G-elliptic}. We report a corollary of	\cite[Corollary 1]{BA-DaL}, for a general operator $G=G\left(u, \na u, \na^2 u \right)$. Then, in next section, we shall apply to the case of the game-theoretic $p$-laplacian.
{
\thm
\label{bardi dalio maximum}
Let $\Om$ be a connected set, $G:\RR\times \RR^N\setminus\{0\}\times \cS^N \to \RR$.  Assume that $G$ is lower semicontinuous and satisfies
\enumerate[(a)]{
 \item for any  $s,r \in\RR$, $\xi\in\RR^N\setminus\{0\}$ and $X,Y\in\cS^N$, it holds that
  $$
  G \left(r,\xi,X\right)\leq G\left(s,\xi, Y\right)  \mbox{ if } \ s\geq r \ \mbox{ and } \ Y\leq X;
 $$
 \item for every $\eta > 0$, there exists a function $\varphi: (0,1) \to (0,1)$ such that 
 $$
 G\left( \la r, \la \xi, \la X \right) \geq \varphi(\la) G\left( r, \xi, X \right)
 $$
 holds for all $r\in[-1,0]$, $0 < |\xi| \leq \eta$, $|X|\leq \eta$;
 \item
there exists $\rho_0 > 0$ such that for any choice of $0 < |\xi| \leq \rho_0$, 
 $$
 G\left( 0, \xi, I - \ga \xi \otimes \xi\right) > 0 \ \mbox{ for } \ \ga > \ga_0
 $$
 holds for some $\ga_0\geq 0$.
}
\par
 Suppose that $u$ be a viscosity subsolution of 
$$
G\left(u, \na u, \na^2 u \right) = 0
$$
%$$
%F\left(x, u, \na u, \na^2 u\right) = 0
%$$
in $\Om$ that achieves a nonnegative maximum in $\Om$. Then $u$ is constant in $\Om$.
}

\section{The case of the game-theoretic $p$-laplacian}
\label{sec:game viscosity}

Now, we consider the case of the game-theoretic $p$-laplacian, that is the differential equations \eqref{G-heat} and \eqref{G-elliptic}. To begin with, we observe that $\De_p^G$ can be formally seen as a singular, quasi-linear operator $F: \left(\RR^N\setminus\{0\}\right)\times \cS^N \to \RR$, where
\begin{equation}
\label{p-laplace-coefficients}
F(\xi,X)=-\tr[A(\xi) X],
\end{equation}
with
\begin{equation}
\label{G-coefficients}
 A(\xi)=\frac1{p}\,I+ \left(1-\frac{2}{p} \right)\,\frac{\xi\otimes \xi}{|\xi|^2}
\end{equation}
for $\xi\in\RR^N\setminus\{0\}$.
Here, $I$ denotes the $N\times N$ identity matrix.
\par
Observe that, if $\xi\neq 0$, $F=F_*=F^*$ while, if $\xi=0$, we can explicitly calculate the semicontinuous envelopes of $F$ (see \cite{AP-2018} or \cite{Do}). It holds that
\begin{eqnarray}
\label{p-laplacian-inf-env}
p\,F_*(0,X)=-\tr(X)-\min(p-2,0)\,\la(X)-\max(p-2,0)\,\La(X),\\
\label{p-laplacian-sup-env}
p\,F^*(0,X)=-\tr(X)-\max(p-2,0)\,\la(X)-\min(p-2,0)\,\La(X),
\end{eqnarray}
where $\la(X)$ and $\La(X)$ are the maximum and minimum eigenvalue of $X$.
Since $F$ has a bounded discontinuity at $\xi=0$, we have that 
\begin{equation}
\label{bounded discontinuity}
-\infty < F_*\left(\xi,X\right) \leq F^*\left(\xi,X\right) < \infty,
\end{equation}
for any $\left(\xi, X\right)\in \RR^N\times \cS^N$.
\par
Note that $F$ is \emph{uniformly elliptic}, in the case $p\in (1,\infty)$, since
$$
\min(1/p',1/p)\,I\leq A(\xi)\leq\max(1/p',1/p)\,I,
$$
and merely (degenerate) elliptic in the case $p=1, \infty$. Moreover, for $\xi\neq 0$, $F$ is a linear operator in the variable $X$.

In the case of \eqref{G-heat}, \eqref{generic parabolic subsolution} and \eqref{generic parabolic supersolution} are replaced by the following. We say that an upper semicontinuous function in $\Om\times (0,\infty)$, $u:\Om\times(0,\infty)\to\RR$, is a {\it viscosity subsolution} of \eqref{G-heat} if, for every $(x,t,\phi)\in\Om\times\left(0,\infty\right)\times C^2\left(\Om\times (0,\infty)\right)$ such that $u-\phi$ attains its maximum at $(x,t)$, then
\begin{equation}
\label{parabolic subsolution}
\begin{cases}
\phi_t(x,t)-\De_p^G \phi(x,t) \leq 0\ \mbox{ if } \ \na\phi(x,t)\neq 0\\
\phi_t(x,t) + F_*\left(0,\na^2\phi(x,t)\right) \leq 0\ \mbox{ if } \na \phi(x,t) =0,
\end{cases}
\end{equation}
where $F_*$ is given by \eqref{p-laplacian-inf-env}.

We say that a lower semicontinuous function in $\Om\times(0,\infty)$, $v:\Om\times(0,\infty)\to\RR$, is a {\it viscosity supersolution} of \eqref{G-heat} if, for every $(x,t,\psi)\in\Om \times (0,\infty)\times C^2\left(\Om\times (0,\infty)\right)$, such that $v-\psi$ attains its minimum at $(x,t)$, then
\begin{equation}
\label{parabolic supersolution}
\begin{cases}
\psi_t(x,t) - \De_p^G \psi(x,t) \geq 0 \ \mbox{ if } \na\psi(x,t)\neq 0\\
\psi_t(x,t) + F^*\left(0,\na^2\psi(x,t)\right) \geq 0 \ \mbox{ if } \ \na\psi(x,t)=0,
\end{cases}
\end{equation}
where $F^*$ is given by \eqref{p-laplacian-sup-env}.

A function $u$ that is both a viscosity subsolution and viscosity supersolution is called a {\it viscosity solution} of \eqref{G-heat}.

%{
%\color{red}
%\begin{rem}
%{\rm
%Actually, as it has been observed in \cite{Gi} (but also in \cite{Do} and \cite{MPR-SJMA}), we can relax the class of test functions in the above definitions. Indeed, in \eqref{parabolic subsolution} and in \eqref{parabolic supersolution}, it suffices to test with smooth functions $\phi \in C^2(\Om\times (0,\infty))$ such that $\na \phi =0$ implies $\na ^2 \phi =0$.
%}
%\end{rem}
%}

Here is the case of the {\it resolvent equation} \eqref{G-elliptic}. We say that an upper semicontinuous function in $\Om$, $u:\Om\to\RR$, is a {\it viscosity subsolution} of \eqref{G-elliptic} if, for every $(x,\phi)\in\Om\times C^2\left(\Om\right)$ such that $u-\phi$ attains its maximum at $x$, then
\begin{equation*}
\label{elliptic subsolution}
\begin{cases}
u(x) - \ve^2 \De_p^G\phi(x)\leq 0 \ \mbox{ if } \ \na\phi(x) \neq 0;\\
u(x) +\ve^2 F_*\left(0, \na^2 \phi(x)\right) \leq 0 \ \mbox{ if } \ \na\phi(x)=0,
\end{cases}
\end{equation*}
where $F_*$ is given by \eqref{p-laplacian-inf-env}.

%{
%\color{red}
%\rem{
%\rm
%Observe that if $u \geq 0$ then, from \eqref{elliptic subsolution}, we have that for every $(x,\phi)\in\Om\times C^2\left(\Om\right)$ such that $u-\phi$ attains its maximum at $x$ and $\na \phi(x) \neq 0$, then
%$$
%- \ve^2 \De_p^G\phi(x)\leq 0.
%$$
%Then $u$ is a viscosity subsolution of $-\De_p^G = 0$.
%}
%}

We say that a lower semicontinuous function $v:\Om\to\RR$, is a {\it viscosity supersolution} of \eqref{G-elliptic} if, for every $(x,\psi)\in\Om\times C^2\left(\Om\right)$, such that $v-\psi$ attains its minimum at $x$, then
\begin{equation*}
\label{elliptic supersolution}
\begin{cases}
v(x) - \ve^2 \De_p^G \psi(x)\geq 0 \ \mbox{ if } \ \na \psi(x) \neq 0,\\
v(x) + \ve^2 F^*\left(0, \na^2 \psi(x)\right) \geq 0 \ \mbox{ if } \ \na \psi(x) =0,
\end{cases}
\end{equation*}
where $F^*$ is given by \eqref{p-laplacian-sup-env}.
A function $u$ is a viscosity solution of \eqref{G-elliptic} if it is both a viscosity subsolution and a viscosity supersolution.

\begin{rem}
{\rm
Observe that, since $\De_p^G$ is degenerate elliptic, then both Lemmas \ref{consistency} and \ref{extension lemma} are valid.
}
\end{rem}

\subsection{Comparison and strong maximum principles}
\label{ssec:game comparisons}

In this section, we provide the ad hoc comparison results that will be applied in the rest of the thesis. 
%{
%\color{red}
%We only give the statements for the parabolic case and refer to the literature for the elliptic case.
%}

\begin{cor}[Comparison principle for \eqref{G-heat}]
\label{cor:parabolic comparison}
Let $\Om$ be a domain in $\RR^N$, with non-empty boundary $\Ga$. Let $u$ and $v$ be two bounded viscosity solutions of \eqref{G-heat} in $\Om\times(0,\infty)$. Assume that $u$ and $v$ are continuous on $\Ga\times (0,\infty)$ and on $\Om\times\{0\}$.
\par
Then, if $u\leq v$ on $\Ga\times (0,\infty)\cup \Om\times\{0\}$, it holds that
$$
u\leq v \ \mbox{ on } \ \ol\Om\times (0,\infty).
$$
\end{cor}

\begin{proof}
It is enough to observe that we can apply Theorem \ref{GGIS comparison} to \eqref{G-heat}. Indeed, (i), (ii) and (iii) of Theorem \ref{GGIS comparison} are satisfied, thanks to Remark \ref{assumptions C} and the assumption for $u-v$ on the parabolic boundary. On the other hand,  (a) and (b) of Theorem \ref{GGIS comparison}, follow from \eqref{bounded discontinuity} and Remark \ref{assumptions on F}, since, in the case of $F_*$ and $F^*$ given by \eqref{p-laplacian-inf-env} and \eqref{p-laplacian-sup-env}, both $F_*$ and $F^*$ are continuous in their variables.
\end{proof}

A standard obvious corollary to comparison principle is the uniqueness of solutions of initial-boundary problems. We state the one of our interest.

\begin{cor}[Uniqueness of solutions of \eqref{G-heat}-\eqref{boundary}]
\label{cor:parabolic uniqueness}
Let $\Om$ be as in Corollary \ref{cor:parabolic comparison}. Then, the bounded viscosity solution of \eqref{G-heat}-\eqref{boundary} is unique.
\end{cor}

The comparison principle for equation \eqref{G-elliptic} follows from Theorem \ref{sato comparison}, as follows.

\begin{cor}[Comparison principle for \eqref{G-elliptic}]
\label{cor:elliptic comparison}
Assume that $\Om$ is a domain, with non-empty boundary $\Ga$. Let $u$ and $v$ be two bounded viscosity solutions of \eqref{G-elliptic} in $\Om$. Assume that $u$ and $v$ are continuous up to the boundary $\Ga$.
\par
Then, if $u\leq v$ on $\Ga$, it holds that
$$
u\leq v \ \mbox{ on } \ \ol\Om.
$$
\end{cor}

\begin{proof}
We need only to ensure that we can apply Theorem \ref{sato comparison} to the case in which $F$ is given by \eqref{p-laplace-coefficients}. In virtue of Remark \ref{assumptions on F}, since $F_*$ and $F^*$ given by \eqref{p-laplacian-inf-env} and \eqref{p-laplacian-sup-env} are continuous, it suffices to verify $(c)$ of Theorem \ref{sato comparison}.  The condition $(c)$ of Theorem \ref{sato comparison} is fulfilled since, away from $\xi = 0$, $F$  is differentiable with respect to $\xi$ and then, in particular, $F$ is Lipschitz continuous (with respect to $\xi$) in the compact set $\{(\xi,X):\rho \leq |\xi| \leq R, |X|\leq R\}$. 
\end{proof}

\begin{cor}[Uniqueness of solutions of \eqref{G-elliptic}-\eqref{elliptic-boundary}]
\label{cor:elliptic uniqueness}
Let $\Om$ be as in Corollary \ref{cor:elliptic comparison}. Then, the bounded viscosity solution of \eqref{G-elliptic}-\eqref{elliptic-boundary} is unique.
\end{cor}

We give, as corollary of Theorem \ref{dalio maximum}, the following result. We observe that in \cite{BBP} the {\it strong minimum principle} for \eqref{G-heat} is proved, by means of a weak Harnack inequality.

\begin{cor}[Strong maximum principle for \eqref{G-heat}]
\label{cor:parabolic maximum}
Let $\Om$ be connected. Let $u$ be a viscosity subsolution of \eqref{G-heat} in $\Om\times (0,\infty)$. If $u$ attains its maximum at a point $(x_0,t_0)\in \Om\times(0,\infty)$, then $u$ must be constant in $\ol\Om\times [0,t_0]$.
\end{cor}

\begin{proof}
We need to check that $F$ in \eqref{p-laplace-coefficients} verifies assumptions $(a)$-$(d)$ of Theorem \ref{dalio maximum}. Conditions $(b)$ and $(d)$ are fulfilled by choosing $\varphi (\la) = \la$, since $-\De_p^G$ is one-homogeneous.
\par
Given $s, \xi \neq 0 $, condition $(a)$ can be read, in the case of game-theoretic $p$-laplacian, as
\begin{equation*}
p s -\tr\left\{
\left( I + (p-2) \frac{\xi\otimes \xi}{|\xi|^2}\right)
\left( I - \ga \xi\otimes \xi\right)
\right\}
=
ps - ( N + p - 2 ) + \ga (p-1) |\xi|^2 > 0,
\end{equation*}
%\begin{multline*}
%p s -\tr\left\{
%\left( I + (p-2) \frac{\xi\otimes \xi}{|\xi|^2}\right)
%\left( I - \ga \xi\otimes \xi\right)
%\right\}
%=
%\\
%ps - ( N + p - 2 ) + \ga (p-1) |\xi|^2 > 0,
%\end{multline*}
which is true for any $\ga > \ga_0$, where $\ga_0 = 
\left[
\frac
{ N+p-2 - ps}
{ (p-1)|\xi|^2 }
\right]_+
$ and for $t \in \RR$, $\left[ t \right]_+ = \max\{ 0 , t \}$.
\par
Analogously, condition $(c)$ it is satisfied, since here we have that
$$
1 + F_*\left(0,\de I\right)=
1 - \de N / p' >0,
$$
for any $\de < p'/N$.
\end{proof}

\begin{rem}
\label{minimum principle}
{\rm
From Remark \ref{strong minimum principle} we have that also a strong minimum principle holds for \eqref{G-heat}.
}
\end{rem}

We shall also need an analogous result for equation \eqref{G-elliptic}, as application of Theorem \ref{bardi dalio maximum}. The proof is straightforward and follows the same line as that of Corollary \ref{cor:parabolic maximum}.

\begin{cor}[Strong maximum principle for \eqref{G-elliptic}]
\label{cor:elliptic maximum}
Let $\Om$ be connected and $u$ be a nonnegative (viscosity) solution of \eqref{G-elliptic} in $\Om$. If $u$ attains a maximum at an interior point, then $u$ must be constant in $\Om$.
\end{cor}

\begin{proof}
We apply Theorem \ref{bardi dalio maximum} with $G\left(u, \na u, \na^2 u \right) = u + \ve^2 F \left( \na u, \na^2 u \right)$, where $F$ is given by \eqref{p-laplace-coefficients}. Indeed, conditions $(a)$ and $(c)$ of Theorem \ref{bardi dalio maximum} follow from the uniform ellipticity of $-\De_p^G u$ and the condition $(b)$ is due to the one-homogeneity of $-\De_p^G u$.
\end{proof}

\begin{rem}
\label{maximum principle}
{\rm 
Observe that we can apply Theorem \ref{bardi dalio maximum} also to $G\left(u,\na u,\na^2 u\right) = F\left(\na u,\na^2 u\right)$ where $F$ is given by \eqref{p-laplace-coefficients}. This gives the strong maximum principle for subsolutions of $-\De_p^G u=0$.
}
\end{rem}

\subsection{Hopf-Oleinik lemma}
\label{ssec:hopf lemma}

We conclude this chapter with a version of Hopf-Oleinik boundary point lemma. This lemma will be used in Chapter \ref{ch:applications}. In the recent work \cite[Theorem 4.4]{ABMMZ} it has been given a sharp version of it, involving the so-called {\it pseudo-balls}, defined as follows. Given $a, b , R>0$ and $\om:[0,R] \to [0,\infty)$, the pseudo-ball, $\mathscr{G}_{a,b}^\om$, with apex $0$, direction $e_N=(0,\dots,1)\in\SS^{N-1}$ and shape function $\om$, is defined by
\begin{equation}
\label{pseudo-ball}
\mathscr{G}_{a,b}^\om = \{x\in B_R: a\om(|x|)|x| \leq x_N \leq b\}.
\end{equation}

In \cite[Theorem 4.4]{ABMMZ} it is considered the case of classical solutions of certain second order linear differential equations in a domain satisfying a pseudo-ball interior condition. Here, we give its adaptation only to the case of our interest.

\par 
 We say that $\tilde{\om}:[0,R] \to [0,\infty)$ is a {\it Dini continuous} function, if it is continuous, $\tilde{\om}(t) > 0$, for $t \in (0,R]$, and satisfies
\begin{equation*}
\int_{0}^{R}\frac{\tilde{\om}(t)}{t}\,dt<\infty.
\end{equation*}

\begin{lem}[Sharp version of Hopf-Oleinik Lemma for $-\De_p^G$]
\label{hopf-oleinik}
Let $a, b, R >0$ and $\Om = \mathscr{G}_{a,b}^\om$, as in \eqref{pseudo-ball}.
\par 
Assume that $\om$ is continuous on $[0,R]$, $\om(t)>0$, for $t\in (0,R]$, it satisfies 
\begin{equation}
\label{condition on omega}
\sup_{0<t\leq R}\left(\frac{\om(t/2)}{\om(t)}\right)<\infty,
\end{equation}
and there exists $\eta > 0$, such that
$$
\frac{\om(s)}{s} \leq \eta \frac{\om(t)}{t}\ \mbox{ for } \ 0 < t \leq s \leq R.
$$
\par
Suppose that there exists a Dini continuous function $\tilde{\om}:[0,R)\to[0,\infty)$, such that
\begin{equation*}
\label{assumption hopf}
\limsup_{\Om\ni x\to 0}
\left[
\frac{\om(|x|)}{|x|}\frac{x_N}{\tilde{\om}(|x|)}
\right]
<\infty.
\end{equation*}
Let $p\in (1,\infty)$ and $u$ be a viscosity subsolution of
\begin{equation}
\label{p-harmonic}
-\De_p^G u = 0,
\end{equation}
such that  $u(0) > u$ in $\Om$.
\par 
If it is assumed that $u$ is differentiable at $0$, then $\na u(0)\neq 0$.
\end{lem}

The proof follows that of \cite[Theorem 4.4]{ABMMZ}. We report its main steps.

\begin{proof}
Let $K > 0$ be the finite number that realizes the supremum in \eqref{condition on omega}. As in \cite[Theorem 4.4]{ABMMZ}, given $\ga > 1+\max\{0, \log_2 K\}$, $C_0, C_1 > 0$, we define the function $v:\mathscr{G}_{a,b}^\om \to \RR$, by
\begin{equation}
\label{pseudo-ball barrier}
v(x)  = - x_N - C_0\int_{0}^{x_N}\int_{0}^{\si} \frac{\tilde{\om}(t)}{t}\, dt \, d\si + C_1 \int_{0}^{|x|}\int_{0}^{\si} \frac{\om (t)}{t}\left(\frac{t}{\xi}\right)^{\ga - 1}  dt \, d\si.
\end{equation}
\par 
We first list the relevant proprieties of $v$, which follow by the assumptions on both $\om$ and $\tilde{\om}$. The function $v$ is of class $C^0\left(\ol\Om\right)\cap C^2\left(\Om\right)$, it satisfies $v(0)=0$ and $\na v(0) = -e_N$.
\par
By combining \eqref{pseudo-ball barrier}, assumptions on $\om$ and elementary integral inequalities, as shown in \cite[Theorem 4.4]{ABMMZ}, it holds that
$$
v(x) \geq \left( \frac{C_1}{2\eta \ga} -a - a C_0 \int_{0}^{b} \frac{\tilde{\om}(t)}{t}\,dt \right) |x|\om(|x|) \ \mbox{ on } \ \pa \mathscr{G}_{a,b}^\om\setminus\{x_N = b \}.
$$
For a fixed $C_1 > 2 a \eta \ga$, it is sufficient to choose $b_* > 0$ such that
$$
\int_{0}^{b_*} \frac{\tilde{\om}(t)}{t}\,dt \leq \frac{C_1 - 2 a \eta \ga}{2 a \eta \ga C_0}
$$ 
to obtain that $v \geq 0$ on $\pa \mathscr{G}_{a,b_*}^\om \setminus\{x_n = b_*\}$. Moreover, since $u$ and $v$ are continuous functions and the fact that $u(0) > u$ in $\mathscr{G}_{a,b}^\om$, on $\cK = \{ x\in \pa \mathscr{G}_{a,b}^\om: x_N = b_* \}$, there exists $\la > 0$ such that
$$
\la < \frac{\min_{\cK}|u-u(x_0)|}{\max_{\cK}|v|}.
$$
In particular, this last condition implies that $\la v \geq u - u(x_0)$ on $\cK$.
Hence, there exist $b_* >0 $ and $\la > 0$ such that 
$$
\la v \geq u - u(0)\ \mbox{ in } \ \pa \mathscr{G}_{a,b_*}^\om.
$$ 
Now, in order to apply comparison, we only need to show, for some $C_0$, that $v$ is a supersolution of \eqref{p-harmonic}.
%Since $\om, \tilde{\om}$ are continuous and satisfy the Dini integrability condition, it follows that $v$ is well defined and 
%$$
%v\in C^2\left(\Om\right)\cap C^0\left(\overline{\Om}\right).
%$$
%We want to prove that $v$ is a viscosity supersolution of \eqref{p-harmonic} in $\Om$ and that $v \geq u$ on $\pa \Om$. Then, we apply comparison. 
%\par 
%In \cite{ABMMZ} it has been shown these that there exist $b_* >0 $ and $\la > 0$ such that $\la v \geq u - u(0)$ in $\pa \mathscr{G}_{a,b_*}^\om$ (see the proof of \cite[Theorem 4.4]{ABMMZ}). Moreover, it has been shown also that $\na v (0) = -e_N \neq 0$. Indeed, by differentiating, we have that
%$$
%\na v (x) =- \left[ 1 + C_0 \int_{0}^{x_N} \frac{\tilde{\om}(t)}{t}\,dt \right] e_N+ \left[C_1 \int_{0}^{|x|}\frac{\om(t)}{t}\left(\frac{t}{|x|}\right)^{\ga - 1}\, dt\right] \frac{x}{|x|}.
%$$
%\par 
\par
By differentiating twice, we have that
\begin{equation*}
\na^2 v(x) = -C_0 \frac{\tilde{\om}(x_N)}{x_N} e_N\otimes e_N  +
 C_1 \int_{0}^{|x|} \om(t) t^{\ga - 2}\, dt\left[\frac{I}{|x|^\ga} - \ga \frac{x\otimes x}{|x|^{\ga + 2}}\right]+C_1 \frac{\om(|x|)}{|x|}\frac{x \otimes x}{|x|^2}.
\end{equation*}
%\begin{multline*}
%\na^2 v(x) = -C_0 \frac{\tilde{\om}(x_N)}{x_N} e_N\otimes e_N  +
%\\ C_1 \int_{0}^{|x|} \om(t) t^{\ga - 2}\, dt\left[\frac{I}{|x|^\ga} - \ga \frac{x\otimes x}{|x|^{\ga + 2}}\right]+C_1 \frac{\om(|x|)}{|x|}\frac{x \otimes x}{|x|^2}.
%\end{multline*}
\par
Since $\na v(0)= -e_N$, by possibly taking $R$ sufficiently small, we can assume that $\na v(x) \neq 0$, for any $x\in\Om$. After setting $\xi(x) = \frac{\na v (x)}{|\na v(x)|}$, we can calculate the following:
\begin{multline*}
p\tr\left\{A(\xi)\na^2 v\right\}(x)=-C_0 \frac{\tilde{\om}(x_N)}{x_N}\left[1+(p-2)\xi_N(x)^2\right]+
\\
\frac{C_1}{|x|^\ga}\int_{0}^{|x|}\om(t)t^{\ga-2}\,dt\left\{N + (p-2) -\ga - \ga(p-2) \frac{\lan x,\xi(x)\ran^2}{|x|^2}\right\}+
\\
C_1\frac{\om(|x|)}{|x|}\left[1+(p-2)\frac{\lan x,\xi(x)\ran ^2}{|x|^2}\right],
\end{multline*}
where $A(\xi)$, defined in \eqref{G-coefficients}, is the matrix of coefficients of $\De_p^G$.
\par
Then, by using also that $|\xi (x)| = 1$, we infer that
\begin{multline*}
p\De_p^G v(x)\leq \\
\frac{\tilde{\om}(x_N)}{x_N}\left[-C_0 c_{1,p} +\frac{C_1}{|x|^\ga}\int_{0}^{|x|}\om(t) t^{\ga-2} \,dt \frac{x_N}{\tilde{\om}(x_N)}\left[N+p -2 -\ga c_{1,p}\right]_+ +\right. \\
\left. C_1\frac{\om(|x|)}{|x|}\frac{x_N}{\tilde{\om}(x_N)}c_{2,p}\right],
\end{multline*}
where $c_{1,p}=\min\{1, p-1\}$, $c_ {2,p} = \max\{ 1, p-1 \}$ and $[t]_+ = \max\{0 , t \}$. 

The previous inequality gives that
\begin{equation*}
\De_p^G v (x)\leq \\
(1/p)\frac{\tilde{\om}(x_N)}{x_N}\left\{ -C_0 c_{1,p}+C_1 \La \left(M{\left[N+p-2-\ga c_{1,p}\right]_+} + c_{2,p}\right) \right\},
\end{equation*}
where
$$
\Lambda=\sup_{x\in \Om}\left(\frac{\om(|x|)}{|x|}\frac{x_N}{\tilde{\om}(x_N)}\right),
$$
and
$$
M=\sup_{x\in\Om} \left(\frac{|x|^{1-\ga}}{\om(|x|)}\int_{0}^{|x|}\om(t)t^{\ga - 2}\,dt\right).
$$
Observe that $\La$ and $M$ are finite as shown in the proof of \cite[Theorem 4.4]{ABMMZ}, from the assumptions on $\om$ and $\tilde{\om}$.
\par
Thus, by choosing $C_0$ such that 
$$
C_0 > {C_1 \frac{M \La}{c_{1,p}}\left[N+p-2 -\ga c_{1,p}\right]_+} + \frac{c_{2,p}}{c_{1,p}},
$$
then $\De_p^G v (x) \leq 0$. Hence, by the arbitrariness of $x$, we have that $-\De_p^G v \geq 0$ in $\mathscr{G}_{a,b_*}^\om$.
%\begin{equation*}
%\label{v supersolution}
%-\De_p^G v \geq 0\ \mbox{ in } \ \mathscr{G}_{a,b_*}^\om.
%\end{equation*}
\par
Now, from
$$
\la v \leq u - u(0) \ \mbox{ on }\ \ol{\mathscr{G}_{a,b_*}^\om}
$$
and $\na v(0) = -e_N$, we can readily obtain the conclusion by a standard argument.
%On the other hand, by an inspection of the definition of $v$, we have that
%$$
%\int_{0}^{x_N} \int_{0}^{\si} \frac{\tilde{\om}(t)}{t}\,dt\,d\si\leq x_N \int_{0}^{x_N} \frac{\tilde{\om}(t)}{t}\,dt
%$$
%and
%\begin{multline*}
%\int_{0}^{|x|}\int_{0}^{\si} \frac{\om(t)}{t}\left(\frac{t}{\si}\right)^{\ga - 1} dt\,d\si \\
%\geq \eta^{-1} \frac{\om(|x|)}{|x|}\int_{0}^{|x|}\int_{0}^{\si}\left(\frac{t}{\si}\right)^{\ga - 1}dt\,d\si = \frac{\om(|x|)|x|}{2\eta\ga}.
%\end{multline*}
%Thus, since on $\pa\Om\setminus \{x_N=b\}$, $x_N=\om(|x|)|x|$, then
%$$
%v(x)\geq \left\{-a+\frac{C_1}{2\eta \ga} - aC_0\int_{0}^{b_*}\frac{\tilde{\om}(t)}{t}\,dt\right\}|x|\om(|x|),
%$$
%so that we can find $C_1$ such that $\frac{C_1}{2\eta\gamma}-a > 0$ and $b_*\in (0, b]$ such that
%\begin{equation}
%\label{v-boundary}
%v \geq 0 \ \mbox{ on }\  \pa\Om \setminus\{x_N= b_*\}.
%\end{equation}
%and $v(x_0)=0$. 
%\par 
%On the other hand, on $\Om\cap \{x_N=b_*\}$, it holds that
%$$
%\frac{v - v(x_0)}{u-u(x_0)} \geq \la^{-1},
%$$
%for some $\la>0$. 
%Hence, using comparison for equation \eqref{p-harmonic} (see \cite[Theorem 2.7]{JLM}), \eqref{v supersolution} and \eqref{v-boundary}, it follows that $\la (v - v(x_0)) - (u-u(x_0))\geq 0$ on $\ol\Om\cap\{ x_N \leq b_*\}$. Then, $0$ is a minimum for $v - v(0) -( u - u(0))$, thus we conclude in the standard manner.
\end{proof}

\begin{rem}
{\rm
In \cite{ABMMZ}, it has been observed that the conditions in Lemma \ref{hopf-oleinik} on $\om$ are sharp (see \cite{ABMMZ} for details).
}
\end{rem}

\begin{rem}
\label{pseudo-balls domains}
{\rm
The class of domains for which Lemma \ref{hopf-oleinik} is valid is quite large and it is given implicitly by the assumptions that must be satisfied by $\om$.
\par
In particular, if $\Om$ is a domain of class $C^{1, \al}$, for some $\al \in (0,1)$, by choosing $\tilde{\om}(t) = \om(t) = t^\al$, the conclusion of Lemma \ref{hopf-oleinik} holds true. 
\par 
In \cite[Theorem 4.7]{ABMMZ}, it has been observed that, if $\Om$ is a domain of class $C^{1,\om}$, where $\om$ is Dini continuous and quasi-increasing, then the conclusion of Lemma \ref{hopf-oleinik} holds true. We say that $\om$ is quasi-increasing if there exists $\tilde{\eta} > 0$ such that $\om(s) \geq \tilde{\eta}\, \om(t)$, for any $0 < s \leq t$.
}
\end{rem}

In order to apply Lemma \ref{hopf-oleinik} to our case, we need to show that the solution of \eqref{G-heat}-\eqref{boundary} is a (viscosity) subsolution of the equation \eqref{p-harmonic} on every slice $\Om\times\{t\}$, for $t > 0$.
%(The definitions of viscosity subsolution ,and supersolution, of \eqref{p-harmonic} are obvious adaptations of those for \eqref{elliptic subsolution} and \eqref{generic elliptic equation}.)

{
\lem
\label{parabolic monotonicity}
Let $u$ be the viscosity solution of \eqref{G-heat} satisfying \eqref{initial}-\eqref{boundary}. 
\par 
Then, for every $t > 0$, $u(\cdot, t) : \Om \to \RR$ is a viscosity subsolution of \eqref{p-harmonic} in $\Om$. 
}

\begin{proof}
Given $\tau > 0$, set
$$
v(x,t) = u(x,t + \tau)\ \mbox{ for } \ (x,t) \in \Om \times (0,\infty).
$$
Up to translate the test functions in \eqref{parabolic subsolution} and \eqref{parabolic supersolution}, we verify that $v$ is a viscosity solution of \eqref{G-heat}. After that, the boundary condition \eqref{boundary} is obviously satisfied by $v$, and we have that $v(x, 0 ) = u(x,\tau) > 0$, by applying the strong minimum principle (Remark \ref{minimum principle}). Corollary \ref{cor:parabolic comparison} gives that $v \geq u$ on $\ol\Om\times(0,\infty)$, which yields the following inequality
\begin{equation}
\label{eq:monotonicity}
u (x,t+\tau) \geq u(x,t), \ x\in\ol\Om, \ t,\tau >0.
\end{equation}
\par
Now, for a fixed $t>0$, let $(x,\varphi) \in \Om\times C^2\left(\Om\right)$ such that $u(\cdot,t) - \varphi$ attains its maximum at $x$. We show that 
\begin{equation}
\label{slice lemma}
F_*\left(\na  \varphi (x), \na^2 \varphi(x) \right) \leq 0,
\end{equation}
where $F$ is that in \eqref{p-laplace-coefficients}.
\par
Define the function $\phi\in C^2\left(\Om\times(0,\infty)\right)$ by $\phi(y,s) = \varphi(y)$, for $(y,s)\in\Om\times (0,\infty)$. From \eqref{eq:monotonicity} and the assumption on $\varphi$, it follows that $u(x,t) - \phi(x,t) \geq u(y,s) - \phi(y,s)$, for any $y\in\Om$ and $0 < s \leq t$. By proceeding as in \cite[Theorem 1]{Ju2001}, this inequality is sufficient to infer that
$$
\phi_t(x,t) + F_*\left(\na \phi(x,t),\na^2 \phi(x,t)\right) \leq 0,
$$ 
since $F$ is a quasi-linear (degenerate) elliptic operator. Therefore, \eqref{slice lemma} follows and this concludes the proof.
\end{proof}

We say that $\Om$ satisfies an interior $\om$-pseudo-ball condition in $x_0\in\Ga$, if there exist $a, b, R >0$ and a modulus $\om$, such that, up to translate and rotate $\Om$, $\mathscr{G}_{a,b}^\om \subset \Om$ and $\pa \mathscr{G}_{a,b}^\om \cap \Ga =\{x_0\}$.

{
\cor
[{Hopf-Oleinik lemma for \eqref{G-heat}}]
\label{cor:parabolic hopf-oleinik}
Let $\Om$ be a domain satisfying the interior $\om$-pseudo-ball condition, with $\om$ that satisfies assumptions of Lemma \ref{hopf-oleinik}. 
\par 
Let $p\in (1,\infty)$ and $u$ be the viscosity solution of \eqref{G-heat}-\eqref{boundary}. Assume that there exist $x_0\in\Ga$ and $\ol{t} > 0$, such that  
$$
u(x_0,\ol{t}) > u\ \mbox{ in } \ \mathscr{G}_{a,b}^\om\times\{\ol{t}\}.
$$
\par 
If it is assumed that $u$ is differentiable at $x_0$, then $\na u(x_0,\ol{t})\neq 0$.
}

\begin{proof}
We apply first Lemma \ref{parabolic monotonicity} and then Lemma \ref{hopf-oleinik} to $y \mapsto u(y,\ol{t})$.
\end{proof}

{
\cor
[Hopf-Oleinik lemma for \eqref{G-elliptic}]
\label{cor:elliptic hopf-oleinik}
Let $\Om$ be a domain satisfying the interior $\om$-pseudo-ball condition, with $\om$ that satisfies assumptions of Lemma \ref{hopf-oleinik}. 
\par
Let $p\in (1,\infty)$ and $u^\ve$ be the viscosity solution of \eqref{G-elliptic}-\eqref{elliptic-boundary}. Assume that there exists $x_0\in\Ga$, such that  
$$
u^\ve(x_0) > u\ \mbox{ in } \ \mathscr{G}_{a,b}^\om.
$$
\par 
If it is assumed that $u^\ve$ is differentiable at $x_0$, then $\na u^\ve(x_0)\neq 0$.
}

\begin{proof}
Observe that, by setting $w \equiv 0$ on $\ol\Om$, then $w$ is a solution of \eqref{G-elliptic} and $w < u^\ve$ on $\Ga$. Hence, by applying Corollary \ref{cor:elliptic comparison} we have that $u^\ve \geq 0$ on $\ol \Om$. This implies that $u^\ve$ is a viscosity subsolution of \eqref{p-harmonic}. We then conclude by applying Lemma \ref{hopf-oleinik}.
\end{proof}
%%%%%
%%%
% simplest domains

%\afterpage{
%\thispagestyle{empty}
%}
%\afterpage{\blankpage}

\chapter{Asymptotics for explicit solutions}
\label{ch:explicit}

In this chapter, we consider equations \eqref{G-heat} and \eqref{G-elliptic} and we deduce asymptotic formulas for global solutions and for solutions of certain boundary-value problems in symmetric domains, such as the half-space, the ball and the exterior of a ball. 
%In this chapter, we shall summarize results, contained in \cite{BM-AA, BM-JMPA}, on asymptotic formulas for radial solutions of the problems of our interest.  In the next chapters these results, in particular precise estimates on asymptotic formulas for radial functions, are of most importance to obtain main results of this thesis.
\par
The obtained solutions will later be used as barriers to extend the relevant asymptotic formulas to more general domains.
\par
In the cases examined in this chapter, the (viscosity) solutions can be explicitly computed by taking advantage of the fact that they are smooth and that the relevant equations become linear. The corresponding boundary-value problems become considerably simpler, since they concern functions that depend on only one space variable.
\par
Most of the explicit representations are based on Bessel functions, whose relevant properties are recalled in Section \ref{sec:bessel}. We present the corresponding theorems in Sections \ref{sec:global solutions}, \ref{sec:radial elliptic solutions} and \ref{sec:radial parabolic solutions}. Section \ref{sec:asymptotics} is then devoted to the asymptotic analysis of the obtained solutions as the relevant parameters tend to zero.

\section{Formulas for Bessel functions}
\label{sec:bessel}
For an overview of this subject, we refer to \cite[Chapter 9]{AS}. Here, we briefly collect the properties of our interest. 
\par
Given $\al \in \CC$, the {\it Bessel's equation of order $\al$}, is the following ordinary differential equation:
\begin{equation}
\label{bessel equation}
\si^2 y'' + \si y' +(\si^2-\al^2) y =0 \ \mbox{ for } \ \si > 0.
\end{equation}
 Every solution of \eqref{bessel equation} can be written as
$$
A \, J_\al(\si) + B\, Y_\al(\si),
$$
where $J_\al$ is called a {\it Bessel function of first kind}, $Y_\al$, is called a {\it Bessel function of second kind}, and $A, B$ are constants. We know that $J_\al$ is finite at $\si=0$, while $Y_\al$ is singular at $\si=0$.
%{
%\color{magenta}
%As in  \cite[Formula 9.1.10]{AS}, $J_\al$ can be defined by its series expansion near $\si=0$, 
%\begin{equation*}
%\label{bessel series}
%J_\al(\si) = \left(\frac{\si}{2}\right)^\al\sum_{m=0}^{\infty} \frac{(-\si^2/4)^m}{m!\Ga\left(\al+m-1\right)}.
%\end{equation*}
%}
%As observed in \cite{AS}, as  a function on $\CC$, $J_\al$ is a holomorphic function throughout $\CC$ cut along the real axis. $J_\al$ is an oscillating function with a sequence of zeroes $\{ \ka_n:n \in \NN \}$.
\par 
The following result can be found in \cite{KKK}.
\begin{lem}[{\cite[Theorem 2.2]{KKK}}]
\label{lem:radial-eigenfunctions}
Let $p \in (1,\infty)$. The eigenfunctions of the problem
\begin{align*}
\label{eq:equation for radial eigenfunctions}
&v''(\si) + \frac{N-1}{p-1}\frac{v'(\si)}{\si} + \frac{p}{p-1} \la_n v(\si) = 0 \ \mbox{ in } \ (0,1),\\
&v'(0)= v(1)  = 0,
\end{align*}  
are given by
\begin{equation}
\label{eq:eigenfunctions}
v_n(\si) = \si ^{-\frac{N-p}{2(p-1)}} J_\frac{N-p}{2(p-1)}(\ka_n\si),
\end{equation}
where, $\ka_n$ is the $n$-th zero of $J_\frac{N - p}{2(p - 1)}$ and $\la_n = \frac{\ka_n^2}{p'}$, $n= 1, 2, \cdots$.
\par 
Moreover, after normalization, the set $\{v_n: n\in \NN\}$ form a complete orthonormal system in the weighted space $L^2\left((0,1); \si^\frac{ N - 1 }{ p - 1 }d\si\right)$.
\end{lem}
\begin{rem}
\label{rem:infinity-eigenfunctions}
{
\rm
Observe that Lemma \ref{lem:radial-eigenfunctions} can be extended to the case $p=\infty$. In this case, for $n = 1,2,\cdots$, the eigenfunctions are given by
$$
v_n (\si) = \sqrt{2} \cos \left (\la_n \si\right), \quad \la_n = \frac{(2n - 1) \pi}{2}.
$$
}
\end{rem}
\par
We will also use the \textit{modified Bessel functions}, that can be defined by the formulas:
\begin{align*}
\label{modified}
&I_\al(\si) = i^{-\al} J_\al(i \si), \ \mbox{ for } \ \si\in\RR,\\
&K_\al(\si)= \frac{\pi}{2 \sin(\al\pi)}\left(I_\al(\si)-I_{-\al}(\si)\right), \ \mbox{ for } \ \si\in\RR,
\end{align*}
where here by $J_\al$ we mean the analytic extension to the complex plane. (Notice that the above definitions hold when $\al$ is not an integer; at integer points $I_\al$ and $K_\al$ are obtained as limits in the parameter $\al$.)
\par
We say that $I_\al$ is a {\it modified Bessel function of first kind} and $K_\al$ is a {\it modified Bessel function of second kind}. They are two linearly independent solutions of the {\it modified Bessel's equation} of order $\al$:
\begin{equation}
\label{modified bessel eqt}
\si^2 y'' + \si y' -(\si^2+\al^2) y =0.
\end{equation}

We will use the following integral representations of $I_\al$ and $K_\al$ (see \cite[formulas 9.6.18, 9.6.23]{AS}), 
\begin{equation}
\label{integral modified bessel}
I_\al(\si) = \frac{(\si/2)^\al}{\sqrt{\pi}\Ga\left(\frac{\al+1}{2}\right)}\int_{0}^{\pi} e^{\si\cos \te}(\sin\te)^{2\al}\,d\te, \ \mbox{ for } \ \re(\al)> - \frac1{2},
\end{equation}
and
\begin{multline}
\label{modified second kind}
K_\al(\si)= \frac{\sqrt{\pi}(\si/2)^\al}{\Ga\left(\frac{\al+1}{2}\right)}\int_{0}^{\infty} e^{-\si\cosh \te}(\sinh\te)^{2\al}\,d\te, \\\
\ \mbox{ for } \ \re(\al) >-\frac1{2}, \ |\arg(\si)| <\frac{\pi}{2}.
\end{multline}

The next lemma will be used to obtain solutions related to the game-theoretic $p$-laplacian in radially symmetric domains and it can be seen as corollary of \cite[Formula 9.1.52]{AS}.

\begin{lem}
\label{lem:scaled differential equations}
Let $p \in (1,\infty)$ and $\la > 0$. The functions
\begin{equation}
\label{eq:radial elliptic solutions}
\si^\frac{p - N}{2(p - 1)}\,I_\frac{N - p}{2(p - 1)}\left(\sqrt{\frac{p}{p-1}}\la \si\right)\ \mbox{ and } \ \si^\frac{p - N}{2(p - 1)}\,K_\frac{N - p}{2(p - 1)}\left(\sqrt{\frac{p}{p-1}}\la \si\right),
\end{equation}
are linearly independent solutions of
\begin{equation}
\label{radial functions elliptic equation}
y''(\si) + \frac{N-1}{p-1}\frac{y'(\si)}{\si} - \frac{p}{p-1} \la^2 y(\si) = 0 \ \mbox{ for } \ \si>0.
\end{equation}
\end{lem}

\begin{proof}
The proof follows at once, by direct inspection.
\end{proof}

%\begin{rem}
%\label{infinity radial elliptic solutions}
%{
%\rm
%Observe that, in the extremal case $p=\infty$, equation \eqref{radial functions elliptic equation} is of the form
%$$
%y''(\si) - \la^2 y(\si) = 0
%$$ 
%and it is a straightforward computation see that the solutions are of the form
%$$
%c_1 e^{\la \si} + c_2 e^{-\la \si} \ \mbox{ for } \ \si >0,
%$$
%for any $c_1,c_2\in\RR$.
%}
%\end{rem}

We conclude this section by deducing asymptotic formulas for the integrals involved in \eqref{integral modified bessel} and \eqref{modified second kind}. The next lemma is essentially contained in \cite{BM-AA}.

\begin{lem}[Asymptotics for the modified Bessel functions]
\label{lem:asymptotics for modified bessel}
For $\al > -1$ and $\si >0$, let
$$
g(\si) = \int_{0}^{\pi} e^{-\si (1- \cos\te)} (\sin\te)^\al\,d\te,
$$
$$
f(\si)=\int_0^\infty e^{-\si (\cosh\te-1)} (\sinh\te)^\al \,d\te.
$$ 
\par
Then, we have that
\begin{equation}
\label{eq:asymptotics-to-infinity-I}
g(\si) = 2^{\frac{\al -1}{2}}\Ga\left(\frac{\al+1}{2}\right)  \si^{-\frac{\al+1}{2}}  \left\{1+ O(1/\si)\right\} 
\end{equation}
\begin{equation}
\label{eq:asymptotics-to-infinity}
f(\si)=
2^\frac{\al-1}{2} \Ga\left(\frac{\al+1}{2}\right)\si^{-\frac{\al+1}{2}} \bigl\{ 1+O(1/\si)\bigr\} 
\end{equation}
as $\si \to \infty$ and
\begin{equation}
\label{eq:asymptotic to zero}
f(\si)=\begin{cases}
\si^{-\al}\, \Ga\left(\al\right)\bigl\{ 1+o(1)\bigr\}  \ &\mbox{ if } \ \al>0, 
\vspace{6pt}
\\
\log(1/\si)+O(1) \ &\mbox{ if } \ \al=0, 
\vspace{2pt}
\\
\frac{\sqrt{\pi}}{2\,\sin(\al\pi/2)}\,\frac{\Ga\left(\frac{\al+1}{2}\right)}{\Ga\left(\frac{\al}{2}+1\right)}+o(1)  \ &\mbox{ if } \ -1<\al<0,
\end{cases}
\end{equation}
as $\si \to 0^+$.
\end{lem}

\begin{proof}
To start with, we establish formula \eqref{eq:asymptotics-to-infinity-I}. By the change of variable $\tau=\si\,(1-\cos\te)$ we get:
$$
g(\si)=2^\frac{\al-1}{2} \si^{-\frac{\al+1}{2}} \int_0^\infty e^{-\tau}\left(\tau-\frac{\tau^2}{2\si}\right)^\frac{\al-1}{2}\!\!d\tau.
$$
This formula gives \eqref{eq:asymptotics-to-infinity-I}, once we observe that the integral $\int_{0}^{\infty} e^{-\tau} \tau^{\frac{\al+1}{2}}$ converges.
\par
To deduce \eqref{eq:asymptotics-to-infinity}, we use the change of variable $\tau=\si\, (\cosh\te-1)$, obtaining
$$
f(\si)=\frac1{\si}\int_0^\infty e^{-\tau} \left(\frac{2\tau}{\si}+\frac{\tau^2}{\si^2}\right)^\frac{\al-1}{2} \!\!d\tau.
$$
When $\si\to\infty$, our claim follows, as above, by writing
$$
f(\si)=2^\frac{\al-1}{2} \si^{-\frac{\al+1}{2}} \int_0^\infty e^{-\tau}\left(\tau+\frac{\tau^2}{2\si}\right)^\frac{\al-1}{2}\!\!d\tau.
$$
\par
When $\si\to 0^+$ and $\al>0$, our claim follows by writing
$$
f(\si)=\si^{-\al} \int_0^\infty e^{-\tau}\left(\tau^2+2\si\,\tau\right)^\frac{\al-1}{2}\!\!d\tau,
$$
since the integral $\int_{0}^{\infty} e^{-\tau} \tau ^{\al -1} \, d\tau$ converges, from $\al > 0$.
\par
For $-1<\al\le 0$, we use \cite[Formula 9.6.23]{AS}, to infer that
$$
f(\si)=\frac1{\sqrt{\pi}}\,\Ga\left(\frac{\al+1}{2}\right)\left(\frac{\si}{2}\right)^{-\frac{\al}{2}} e^\si\, K_{\al/2}(\si).
$$
Then, \cite[Formulas 9.6.9 and 9.6.13]{AS} give our claims for $\al=0$ and $-1<\al<0$, respectively.
\end{proof}

%%%%%
%%%%
%%Global solution

\section{The global solutions}
\label{sec:global solutions}

In this section, we present the global solutions of the parabolic equation \eqref{G-heat} and of the resolvent equation \eqref{G-elliptic}. The former can be found in \cite[Proposition 2.1]{BG-IUMJ} . We recall that $p' = p/(p-1)$ and we mean $p'=1$ for $p=\infty$.

\begin{prop}
\label{prop:parabolic global solution}
Let $p \in (1,\infty]$ and $\Phi$ be the function defined as
\begin{equation}
\label{eq:parabolic global solution}
\Phi(x,t) =
\begin{cases}
\displaystyle
 t^{-\frac{N+p-2}{2(p-1)}}e^{-p'\frac{|x|^2}{4 t}} \ \mbox{ if } \ p\in(1,\infty),
 \vspace{8pt}\\
 \displaystyle
t^{-\frac12} e^{-\frac{|x|^2}{4t}} \ \mbox{ if } \ p=\infty,
\end{cases}
\end{equation}
for $x\in\RR^N$ and $ t > 0$.
\par
Then, $\Phi$ is a viscosity solution of \eqref{G-heat} in $\RR^N \times (0,\infty)$ and is bounded on $\left(\RR^N\setminus B_\de \right) \times (0,\infty)$, for $\de > 0$.
\end{prop}
\begin{proof}
Observe that $\Phi \in C^\infty \left( \RR^N \times(0,\infty) \right)$ and that $\Phi(x,t) = \phi(|x|,t)$, where $\phi=\phi(r,t)$ is clearly defined by \eqref{eq:parabolic global solution}. Equation \eqref{G-heat} for a radially symmetric function $u$ reads as
$$
u_t=\frac{p-1}{p} u_{rr} + \frac{N-1}{p} \frac{u_r}{r}.
$$ 
By simple computations, we verify that this equation is satisfied (pointwise) by $\phi(r,t)$ for $r\ne 0$ and $t>0$.
\par
Thus, $\Phi$ is a classical solution in $\left(\RR^N \setminus\{0\} \right) \times (0,\infty)$, where its spatial gradient does not vanish. Then, we apply Lemma \ref{extension lemma} to conclude.
\par 
The stated boundness of $\Phi$ easily follows by its definition.
\end{proof}

\begin{prop}
\label{prop:elliptic global solution}
Let $p \in (1,\infty]$ and $\Phi_\ve$ be the function defined, for $x\in \RR^N$, by
\begin{equation*}
\label{elliptic global solution}
\Phi_\ve(x) = 
\begin{cases}
\displaystyle
\int_{0}^{\infty}e^{-\sqrt{p'}\frac{|x|}{\ve}\cosh\te}(\sinh\te)^\frac{ N - p }{ p - 1 }\,d\te\ &\mbox{ if }\ p\in(1,\infty), 
\vspace{8pt}\\
\displaystyle
e^{-\frac{|x|}{\ve}} \ &\mbox{ if } \ p=\infty.
\end{cases}
\end{equation*}
\par
Then $\Phi_\ve$ is a viscosity solution of \eqref{G-elliptic} in $\RR^N\setminus\{ 0\}$. Moreover, in the case $p\in(N,\infty]$, $\Phi_\ve$ is bounded in $\RR^N$ while, in the case $p\in (1,N]$, it is bounded on the complement of any ball centered at $0$.
\end{prop}

\begin{proof}
The case $p=\infty$ is immediate. In the case $p\in(1,\infty)$, from \eqref{modified second kind}, we have that $\Phi_\ve = \phi(|x|)$, where $\phi(r)$ is a solution of \eqref{radial functions elliptic equation}, for $r > 0$. A direct check shows that \eqref{radial functions elliptic equation} is simply the equation \eqref{G-elliptic} for radial functions. Hence, $\Phi_\ve$ is a solution of \eqref{G-elliptic}, outside the origin. The stated boundedness follows from the properties of the function $K_\frac{N-p}{2(p-1)}$.
\end{proof}

%%%%%%
%%%%
%Elliptic radial solutions

\section{Elliptic solutions in symmetric domains}
\label{sec:radial elliptic solutions}

%In this section, we collect some of the main ingredients to obtain results of this thesis. These results regard radial solutions of both problems \eqref{G-elliptic}-\eqref{elliptic-boundary} and \eqref{G-heat}-\eqref{boundary}. 
%\par
%In this context, a crucial observation is the following: if a function $u(x)=\varphi(|x|)$, for some $\varphi:(0,\infty) \to \RR$, then
%\begin{eqnarray*}
%\De_p^G u (x)= \De u (x) / p + (1-2/p) \De_\infty^G u(x)	\\
%= \frac{\varphi''(|x|)}{p'} + \frac{N-1}{p} \frac{\varphi'(|x|)}{|x|},
%\end{eqnarray*}
%since, in this case, $\De_\infty ^G u(x) = \varphi''(|x|)$.
%\par
%Starting to this, we shall emphasize the connection with one-variable special functions called {\it Bessel functions} and {\it modified Bessel functions} and their asymptotics.
%\par
%Preliminarily, in the next subsection, we expose the main definitions and properties of the special functions, which we shall use later in this section, and then, in Subsections \ref{ssec:radial solutions} and \ref{ssec:radial parabolic}, we study, respectively, radial solutions of \eqref{G-elliptic}-\eqref{elliptic-boundary} and the solution of \eqref{G-heat}-\eqref{boundary} in a ball.
%
%To fix the notations, $B_R$ shall be the ball centered at the origin with radius $R>0$. 
%\par
{
%\color{magenta}
%In this section, by appropriately applying Lemmas \ref{lem:scaled differential equations}, we obtain quite explicit radial solutions of the elliptic problem \eqref{G-elliptic}-\eqref{elliptic-boundary}.
%}

\begin{lem}[{Elliptic solution in the ball, \cite[Lemma 2.1]{BM-JMPA}}]
\label{th:ball elliptic solution}
Set $p\in (1,\infty]$.
\par
Then, for $x\in\ol B_R$, the following function,
\begin{equation} 
\label{ball solution formula}
u^\ve(x)=\begin{cases}
\displaystyle 
\frac{\int_{0}^{\pi}e^{\sqrt{p'}\frac{|x|}{\ve}\cos \te}(\sin\te) ^{\frac{ N - p }{ p - 1 }}\,d\te}{\int_{0}^{\pi}e^{\sqrt{p'}\frac{R}{\ve}\cos\te}(\sin\te) ^{\frac{ N - p }{ p - 1 }}\,d\te} \ &\mbox{ if } \ p\in(1,\infty),
\vspace{8pt} \\
\displaystyle
\frac{\cosh(|x|/\ve)}{\cosh(R/\ve)}  \ &\mbox{ if } \ p=\infty,
\end{cases} 
\end{equation}
is the (viscosity) solution of \eqref{G-elliptic}-\eqref{elliptic-boundary}.
\end{lem}

\begin{proof}
We only consider the case $ p \in (1,\infty)$, while the extremal case $p=\infty$ is similar and simpler. We have that $u^\ve(x) = \phi(|x|)$, where $\phi(r) = r^\frac{p-N}{2(p-1)}I_\frac{N-p}{2(p-1)}(r)$. For Lemma \ref{lem:scaled differential equations}, $\phi(r)$ is a solution of \eqref{radial functions elliptic equation}, which is the equation \eqref{G-elliptic} for radial functions. Hence, $u^\ve$ is a classical solution of \eqref{G-elliptic}, away from the origin. Since $u^\ve$ is of class $C^2$ in $B_R$, then Lemma \ref{extension lemma} informs  us that it is also a viscosity solution in the whole ball. Finally, it is clear that $u^\ve = 1$ on the boundary, then it is the solution of \eqref{G-elliptic}-\eqref{elliptic-boundary}, by uniqueness (Corollary \ref{cor:elliptic uniqueness}).
\end{proof}

%{
%\color{magenta}
%\begin{proof}
%For radial symmetric functions, equation \eqref{G-elliptic} can be read as equation \eqref{radial functions elliptic equation}, with $\la = \ve^{-2}$. 
%The case of $p=\infty$ is immediate. For $ p \in (1,\infty)$, we observe that $u^\ve(x)=\varphi(|x|)/\varphi(R)$, where 
%$$
%\varphi(\si) = \int_{0}^{\pi} e^{\sqrt{p'}\frac{\si}{\ve}\cos\te}(\sin\te)^\frac{ N - p }{ p - 1 }\,d\te.
%$$
% \par 
% After observing that $\varphi(\si)$ is a multiple of $\si^{-\frac{N - p}{2(p-1)}} I_\frac{N-p}{2(p-1)}(\si \sqrt{p'}/\ve)$ (because of \eqref{integral modified bessel}), we apply Lemma \ref{lem:scaled differential equations} to infer that $\varphi$ satisfies the following equation:
%  \begin{equation}
%    \label{eq:ode1}
%    \frac{\varphi''(\si)}{p'}+\frac{N-1}{p}\frac{\varphi'(\si)}{\si}-\frac{\varphi(\si)}{\ve^2}=0\ \mbox{ in }\ (0,R)
% \end{equation}
% which is \eqref{G-elliptic} written for radial symmetric functions. Hence, $u^\ve$ is a classical solution of \eqref{G-elliptic} in $B_R \setminus\{0\}$.
% \par 
%Moreover, by direct inspection, $u^\ve$ is of class $C^2(B_R)$. Then, we can apply Lemma \ref{extension lemma} at the origin, the unique critical point of $u^\ve$, where the differential equation \eqref{G-elliptic} is not classically defined.
%\par
%Finally, the statement follows by the uniqueness of viscosity solution, given by Corollary \ref{cor:elliptic comparison}.
%\end{proof}	
%}

\begin{lem}[{Elliptic solution in the exterior of the ball, \cite[Lemma 2.2]{BM-AA}}]
\label{th:exterior elliptic solution}
Set $ p \in (1, \infty]$.
\par
Then, the following function, defined for $x\in \RR^N \setminus B_R$, by
\begin{equation}
\label{solution-exterior}
u^\ve(x)=
\begin{cases}
\displaystyle
\frac{\int_{0}^{\infty}e^{-\sqrt{p'}\frac{|x|}{\ve}\cosh\te}(\sinh\te)^\frac{ N - p }{ p - 1 }\,d\te}{\int_{0}^{\infty}e^{-\sqrt{p'} \frac{R}{\ve}\cosh\te}(\sinh\te)^\frac{ N - p }{ p - 1 }\,d\te}\ &\mbox{ if }\ p\in(1,\infty),  
\vspace{8pt}\\
\displaystyle
e^{-\frac{|x|-R}{\ve}} \ &\mbox{ if } \ p=\infty,
\end{cases}
\end{equation}
is the bounded (viscosity) solution of \eqref{G-elliptic}-\eqref{elliptic-boundary}.
\end{lem}

\begin{proof}
We just need to observe that $u^\ve$ is just $\Phi_\ve$ of Proposition \ref{prop:elliptic global solution}, once it is normalized to the value $1$ on the boundary of $B_R$. Hence, $u^\ve$ is a bounded solution of \eqref{G-elliptic}-\eqref{elliptic-boundary}. By uniqueness (Corollary \ref{cor:elliptic uniqueness}), we conclude.
\end{proof}

\section{Parabolic solutions in symmetric domains}
\label{sec:radial parabolic solutions}

%{
%\color{magenta}
% This section is divided into two subsections. In Subsection \ref{ssec:some lemmas} we establish two technical lemmas on the solution of \eqref{G-heat}-\eqref{boundary} in $B \times (0,\infty)$. In Subsection \ref{ssec:parabolic ball asymptotic} we obtain the short-time asymptotics (Theorem \ref{th:parabolic ball asymptotic}).
% }
% In this section, we establish the series representation of the solution $u$ of the parabolic problem \eqref{G-heat}-\eqref{boundary}, in a ball. Indeed, we observe that, since for radial functions equation \eqref{G-heat} is linear, the solution of \eqref{G-heat}-\eqref{boundary} it can be established as series. The following lemma is contained in \cite{BM-JMPA}, and it is also obtained as an application of \cite[Corollary 2.5]{KKK}.

%\section[One-dimensional solution of $(1.3)$ - $(1.5)$]{{One-dimensional solution of \eqref{G-heat}-\eqref{boundary}}}
%\label{sec:one-dimensional}
Before going on, we point out that, from now on, with $\Ga$ we indicate the boundary of the relevant set, which it will be evident by the context.
\par
First, we focus our attention on the case of the half-space of $\RR^N$. We will use the {\it complementary error function} defined by 
$$
\Erfc(\si)=\frac{2}{\sqrt{\pi}}\,\int_\si^\infty e^{-\tau^2} d\tau, \quad \si\in\RR.
$$

\begin{prop}[{Parabolic solution in the half-space, \cite[Proposition 2.3]{BM-JMPA}}]
\label{Asymp hspace}
Let $p\in (1,\infty]$ and $H$ be the half-space in which $x_1>0$. The function $\Psi$, defined by 
$$
\Psi(x,t)=\sqrt{\frac{p'}{4\pi}}\int_{\frac{x_1}{\sqrt{t}}}^{\infty} e^{-\frac14 p' \si^2}\,d\si=
\Erfc\left(\frac{\sqrt{p'} x_1}{2\sqrt{t}}\right) \ \mbox{ for } \ (x,t)\in\ol{H}\times (0,\infty).
$$
is the bounded solution of \eqref{G-heat}-\eqref{boundary}.
\end{prop}

\begin{proof}
In virtue of Corollary \ref{cor:parabolic uniqueness}, the bounded solution of \eqref{G-heat}-\eqref{boundary} is unique. The function $\Psi$ is smooth with no critical points. To conclude, it is enough to note that, after an inspection, it satisfies pointwise \eqref{G-heat} in $H \times (0,\infty)$ and both \eqref{initial} and \eqref{boundary}.
\end{proof}

%\subsection{Two technical lemmas}
%\label{ssec:some lemmas}

Now, we establish a series representation of parabolic solutions in the ball.

\begin{lem}[{Parabolic solution in the ball}]
\label{lem:series solution}
Let $ p \in (1,\infty]$.
\par
Then, the solution of \eqref{G-heat}-\eqref{boundary} has the following representation:
\begin{enumerate}[(i)]
\item 
if $p \in(1, \infty)$, we have that, for $x\in B_R$ and $t > 0$,
\begin{equation}
\label{solution by series}
u(x,t) = 2 \left( \frac{ R }{ |x| } \right)^\frac{ N- p }{ 2(p - 1) }
\sum_{ n = 1 }^{ \infty }
\frac
{ J_\frac{ N - p }{ 2(p - 1) }\!\left( \frac{ \ka_n }{ R }|x| \right) }
{ \ka_n \, J_\frac{ N + p - 2 }{ 2(p - 1) }\left( \ka_n \right) }
\left( 1 - e^{-\frac{\ka_n^2}{p'\, R^2} t} \right).
\end{equation}
where $\ka_n$ is the $n$-th positive zeros of $J_\frac{N - p}{2(p - 1)}$. 
%(Here, the convergence must be intended in $L^2\left((B_R; |x|^{(N-1)(2-p)/(p-1)} \, dx)\right)$.)
\item
If $p=\infty$, we get, for $x\in B_R$ and $ t > 0$:
\begin{equation}
\label{infinity solution by series}
u(x,t)=
\frac4{\pi}\sum_{n=1}^{\infty}
\frac
{ (-1)^{n-1} \cos\left( (2n-1)\frac{\pi |x|}{2R}\right) }
{ 2n-1 }
 \left( 1 - e^{-\frac{(2n-1)^2 \pi^2}{4R^2}\,t} \right).
\end{equation}
\end{enumerate}
\end{lem}

\begin{proof}
Preliminarily, we observe that, since \eqref{G-heat} is invariant if we add a constant to $u$, $ 1 - u $ is also a solution of \eqref{G-heat}. Moreover, $ 1 - u $ vanishes on $\Ga \times (0,\infty)$ and is equal to $1$ on $B_R\times \{0\}$. We want to find the series representation of $1 - u$.
\par
We use the same scheme, in all cases $p \in (1,\infty]$, clarifying that, in the extremal case $p = \infty$, we use Remark \ref{rem:infinity-eigenfunctions} instead of Lemma \ref{lem:radial-eigenfunctions}.
Now, set $\al = \frac{N-p}{2(p-1)}$.
By applying Lemma \ref{lem:radial-eigenfunctions}, we infer that $\{\si^{-\al} J_\al \left( \frac{\ka_n}{R} \si \right):n=1,2,\cdots\}$ is a complete system in  $L^2\left((0,R); \si^{2\al +1 }\,d\si \right)$. This fact implies that 
\begin{equation}
\label{coefficients}
\sum_{n=1}^{\infty} c_n \left\{\si^{-\al} J_\al\left(\frac{\ka_n}{R}\si\right)\right\} = 1 \ \mbox{ for } \ \si \in (0,R),
\end{equation}
for
\begin{equation*}
c_n = 
\frac
{ \int_{0}^{R}\si^{\al + 1} J_\al\left(\frac{\ka_n}{R} \si\right) \, d\si }
{ \int_{0}^{R}\si J_\al\left( \frac{\ka_n}{R} \si \right)^2\, d\si }
=
\frac
{ \int_{0}^{R}\si^{\al + 1} J_\al\left(\frac{\ka_n}{R} \si\right) \, d\si }
{ R^2 J_{\al + 1}(\ka_n)^2 /2},
\end{equation*}
where, in the last equality, we have used \cite[formulas 11.4.5 and 9.5.4]{AS}. 
Now, since, by formula \cite[formula 11.3.20]{AS} we know that
$$
\int_{ 0 }^{ 1 } \si^{ \al + 1 } J_\al\left( \ka_n \si \right) \, d\si = \frac{ J_{ \al + 1 }(\ka_n) }{ \ka_n },
$$
we obtain
$$
c_n = \frac{2 R^\al}{\ka_n J_{\al +1 }(\ka_n)}.
$$
\par
We see that, for any $n=1,2,\cdots$, by reason of \eqref{eq:eigenfunctions}, the function
$$
u_n (x,t) = |x|^{-\al} J_\al\left(\frac{\ka_n}{R} |x|\right) e^{-\frac{ \ka_n^2 }{ p'\,R^2 } t},
$$
 is a solution of \eqref{G-heat} in $B_R\times (0,\infty)$ and it vanishes on $\Ga \times (0,\infty)$. 
 \par
Thus, by uniqueness (Corollary \ref{cor:parabolic uniqueness}, we have that
$$
1 - u(x,t) = \sum_{n=1}^{\infty} \frac{2 R^\al}{\ka_n J_{\al + 1}\left(\ka_n\right)} |x|^{-\al} J_\al\left(\frac{\ka_n}{R} |x|\right) e^{-\frac{ \ka_n^2 }{ p'\,R^2 } t},
$$
which, by using \eqref{coefficients}, gives \eqref{solution by series}. 
%{
%\color{magenta}
%
%From \cite[Formula 11.3.20]{AS},
%$$
%\int_{ 0 }^{ 1 } \te^{ \al + 1 } J_\al\left( \ka_n \te \right) \, d\te = \frac{ J_{ \al + 1 }(\ka_n) }{ \ka_n }
%$$
%and, from \cite[Formulas 11.4.5 and 9.5.4]{AS},
%$$
%\int_{0}^{1} \te J_\al^2\left(\ka_n \te\right)\,d\te = \frac{ J_{\al + 1}^2\left(\ka_n\right)}{ 2 },
%$$
%}
%{
%\color{magenta}
%\par
%The case $p = \infty$ follows by considering the eigenfunctions of the \eqref{eq:equation for radial eigenfunctions} in $(0,R)$, in the extremal case $p = \infty$, which, for $n = 1,2,\cdots$, are given by $v_n (\si) = \sqrt{2} \cos \left ( (2n - 1) \frac{\pi}{2} \si\right)$. In this case, we repeat the same argument used for $p \in (1,\infty)$, with $\ka_n = \frac{(2n - 1) \pi}{2R}$ and the fact that the constant function $1$ can be written as series
%$$
%1 = \sum_{n=0}^{\infty} \tilde{c}_n \cos \left(  (2n - 1) \frac{\pi}{2R} \si \right) \ \mbox{ for } \ \si \in(0, R),
%$$
%where
%$$
%\tilde{c}_n = 
%\frac
%{ \int_{0}^{R} \cos \left( (2n-1) \frac{\pi}{2R} \si\right) \,d\si }
%{ \int_{0}^{R} \cos^2\left( (2n-1) \frac{\pi}{2R} \si \right) \, d\si }
%=
%\frac4{\pi}
%\frac
%{ (-1)^{n-1} }
%{ (2n-1) }.
%$$
%}
\par
Finally, we observe that, thanks to Remark \ref{rem:infinity-eigenfunctions}, in the case $p = \infty$, \eqref{solution by series} is just \eqref{infinity solution by series}.
\end{proof}

Last, we report the following connection between elliptic and parabolic radial (viscosity) solutions, as stated in \cite[Lemma 2.5]{BM-JMPA}.

\begin{lem}[Laplace transform and symmetric solutions]
\label{radial laplace transform}
Let $ p \in (1,\infty]$. Let $u(x,t)$ be given by Lemma \ref{lem:series solution}. Then, it holds
\begin{equation}
\label{eq:modified laplace transform}
u^\ve(x) = \ve^{-2} \int_{0}^{\infty} u(x,\tau) e^{-\tau/\ve^2}\,d\tau \ \mbox{ for } \ x\in\ol{B_R},
\end{equation}
where $u^\ve$ is the solution of \eqref{G-elliptic}-\eqref{elliptic-boundary}, in $B_R$.
\end{lem}

\begin{proof}
We prove that both sides of \eqref{eq:modified laplace transform} have the same eigenfunction expansion. Here, we treat the case $p\in (1,\infty)$. In the extremal case $p=\infty$ we have only to utilize the expressions of eigenfunctions in Remark \ref{rem:infinity-eigenfunctions}.
\par
Let $\ka_n$ be given by Lemma \ref{lem:radial-eigenfunctions}, then, for any $\ve>0$, it holds that
\begin{equation*}
\label{eq:laplace transform partial}
\ve^{-2}
\int_{0}^{\infty}\left(1 - e^\frac{\ka_n^2 \tau}{p'R^2}\right) e^{-\tau/\ve^2}\,d\tau
 = 
\frac{\ve^2\ka_n^2/(p'R^2)}{\ve^2\ka_n^2/(p' R^2) +1} 
=
 \frac{\ve^2 \ka_n^2}{\ve^2 \ka_n^2+ p' R^2}.
\end{equation*}
\par
Hence, from \eqref{solution by series}, we obtain that the right-hand side of \eqref{eq:modified laplace transform} equals
\begin{equation}
\label{eq:series expression of u^ve}
2 \left(R/|x|\right)^\frac{N-p}{2(p-1)} R^{- 2} 
\sum_{ n=1 }^{ \infty } \frac{ \ka_n }{ J_\frac{ N + p - 2 }{ 2(p - 1) }(\ka_n) } \left\{\frac{p'}{\ve^2} + \frac{\ka_n ^2}{R^2}\right\}^{-1}\, J_\frac{N-p}{2(p-1)}\left(\frac{\ka_n}{R}|x|\right),
\end{equation}
for any $x\in B_R$ and $\ve > 0$. 
\par 
Now, we observe that the sum of the last series is the function $u^\ve$. Indeed, by comparing \eqref{ball solution formula} and \eqref{integral modified bessel}, we have that $u^\ve$ is given by 
\begin{equation*}
\label{eq:expression for u^ve}
u^\ve(x) = 
(R/|x|)^\frac{ N - p }{ 2(p - 1) } \, I_\frac{ N - p }{ 2(p - 1) }\left(\frac{\sqrt{p'}}{\ve}|x|\right) / I_\frac{N-p}{2(p-1)}\left(\frac{\sqrt{p'}}{\ve}R\right)
\end{equation*}
and then the relevant coefficients can be calculated by
%\begin{multline*}
%u^\ve(x) =
%\\
% \frac{R^\frac{N-p}{2(p-1)}}{I_\frac{N-p}{2(p-1)}\left(\frac{\sqrt{p'}}{\ve}R\right)}
%\sum_{n=1}^{\infty}
% \frac
% {\int_{0}^{R} \si I_\frac{N-p}{2(p-1)}\left(\frac{\sqrt{p'}}{\ve}\si\right)\, J_\frac{N-p}{2(p-1)}\left(\frac{\ka_n}{R}\si\right)\,d\si}
% {R^{2 }J_{\frac{N+p-2}{2(p-1)} + 1}(\ka_n)^2/2 }
%   \frac{J_\frac{N-p}{2(p-1)}\left(\frac{\ka_n}{R}|x|\right)}{|x|^{\frac{N-p}{2(p-1)}}},
%\end{multline*}
 applying \cite[Formula 11.3.29]{AS}, that is
\begin{multline*}
\int_{0}^{R} \si I_\frac{N-p}{2(p-1)}\left(\frac{\sqrt{p'}\si}{\ve}\right)\, J_\frac{N-p}{2(p-1)}\left(\frac{\ka_n}{R}\si\right)\,d\si =\\
 \ka_n I_\frac{N-p}{2(p-1)}\left(\frac{\sqrt{p'} R}{ \ve}\right) J_\frac{N+p-2}{2(p-1)}\left(\ka_n\right)\left\{\frac{p'}{\ve^2} + \frac{\ka_n ^2}{R^2}\right\}^{-1},
\end{multline*}
which gives \eqref{eq:series expression of u^ve}.
\end{proof}

\section{Asymptotics}
\label{sec:asymptotics}

In this section, we collect asymptotic formulas for the functions presented in Sections \ref{sec:radial elliptic solutions} and \ref{sec:radial parabolic solutions}.

\subsection{The elliptic case}
\label{ssec:elliptic asymptotics}

\begin{thm}[{Asymptotics in the ball, \cite[Lemma 2.1]{BM-AA}}] 
\label{th:asymptotics in the ball}
Let $p\in (1,\infty]$. Assume that $u^\ve$ be the (viscosity) solution of \eqref{G-elliptic}-\eqref{elliptic-boundary} in $B_R$.
\par
Then, it holds that
\begin{equation}
\label{uniform-ball}
\ve\log u^\ve+\sqrt{p'} d_\Ga=
\begin{cases}
\displaystyle 
O(\ve\,\log\ve) \ &\mbox{ if } \ 1<p<\infty,
\\
\displaystyle
O(\ve)  \ &\mbox{ if } \ p=\infty,
\end{cases} 
\end{equation}
uniformly on $\ol{B_R}$ as $\ve\to 0^+$.
\end{thm}

\begin{proof}
\par
First, we observe that $d_\Ga ( x ) = R - |x| $.
The case $p=\infty$ follows at once, since \eqref{ball solution formula} gives:
$$
\ve\log\left\{u^\ve(x)\right\}+d_\Ga(x)=\ve\log\left[\frac{1+e^{-2\frac{|x|}{\ve}}}{1+e^{-2\frac{R}{\ve}}}\right].
$$
\par 
If $1<p<\infty$, by \eqref{ball solution formula} we have that
\begin{equation*}
\ve\log\left\{u^\ve(x)\right\}+\sqrt{p'}\,d_\Ga(x)=
\ve\,\log\left[\frac{\int_{0 }^{\pi}e^{-\sqrt{p'}(1-\cos\te)\,\frac{|x|}{\ve}}(\sin\te)^\frac{N-p}{p-1}\,d\te}{\int_{0}^{\pi}e^{-\sqrt{p'}(1-\cos\te)\,\frac{R}{\ve}}(\sin\te)^\frac{N-p}{p-1}\,d\te}\right]
\end{equation*}
and the right-hand side is decreasing in $|x|$, so that 
\begin{equation*}
0
\le
\ve\,\log\left\{u^\ve(x)\right\}+\sqrt{p'}\,d_\Ga(x)\le 
\ve\,\log\left[
\frac{\int_{0}^{\pi}(\sin\te)^\frac{N-p}{p-1}\,d\te}
{\int_{0}^{\pi}e^{-\sqrt{p'}\,(1-\cos\te)\,\frac{R}{\ve}}(\sin\te)^\frac{N-p}{p-1}\,d\te}
\right].
\end{equation*}
This formula gives \eqref{uniform-ball}, since we have that 
\begin{equation*}
\int_{0}^{\pi}e^{-\sqrt{p'}\,(1-\cos\te)\frac{R}{\ve}}(\sin\te)^\frac{N-p}{p-1}\,d\te=2^{\frac{N - 2p  +1}{2(p-1)}}\Ga\left(\frac{N - 1}{2p-2}\right)\left(\frac{R\sqrt{p'}}{\ve}\right)^{-\frac{N - 1}{2p - 2}}\,[1+O(\ve)]
\end{equation*}
as $\ve\to 0^+$, by using \eqref{eq:asymptotics-to-infinity-I}, with $\si = \sqrt{p'} R/\ve$.
\end{proof}

The next theorem is contained in \cite[Lemma 2.2]{BM-AA}.

\begin{thm}[Asymptotics in the exterior of the ball]
Let $p \in (1,\infty]$. Assume that $u^\ve$ be the (viscosity) solution of \eqref{G-elliptic}-\eqref{elliptic-boundary} in $\RR^N\setminus \ol{B_R}$.
\par 
Then, it holds that
\begin{equation}
\label{eq:uniform-exterior-estimate}
  \ve\log u^\ve+\sqrt{p'}\,d_\Ga=O(\ve)  \ \mbox{ as } \ \ve\to 0^+,
\end{equation}
uniformly on every compact subset of $\RR^N\setminus B_R$.
\end{thm}

\begin{proof}
\par
Notice that $d_\Ga(x)=|x|-R$. If $p=\infty$, \eqref{eq:uniform-exterior-estimate} holds exactly as
$$
\ve\log\left\{u^\ve(x)\right\}+d_\Ga(x)\equiv 0.
$$
\par
If $1<p<\infty$, we write that
$$
\ve\log\left\{u^\ve(x)\right\}+\sqrt{p'}\,d_\Ga(x)=
\ve \log \left[
\frac{\int_{0}^{\infty}e^{-\sqrt{p'}\frac{|x|}{\ve}(\cosh\te-1)}(\sinh\te)^\frac{ N - p }{ p - 1 }\,d\te}{\int_{0}^{\infty}e^{-\sqrt{p'}\frac{R}{\ve}(\cosh\te-1)}(\sinh\te)^\frac{ N - p}{ p - 1 }\,d\te}
\right]
$$
and hence, by monotonicity,  we have that
$$
\ve\,\log\left\{\frac{\int_{0}^{\infty}e^{-\sqrt{p'}\frac{R'}{\ve}(\cosh\te-1)}(\sinh\te)^\frac{N-p}{p-1}\,d\te}{\int_{0}^{\infty}e^{-\sqrt{p'}\frac{R}{\ve}(\cosh\te-1)}(\sinh\te)^\frac{N-p}{p-1}\,d\te}\right\}\le\ve\log\left\{u^\ve(x)\right\}+\sqrt{p'}\,d_\Ga(x)\le 0,
$$
for every $x$ such that $R\le|x|\le R'$, with $R'>R$.
Our claim then follows by an inspection on the left-hand side, after applying \eqref{eq:asymptotics-to-infinity}, with $\si=\sqrt{p'} R'/\ve$ and $\si=\sqrt{p'} R/\ve$.
\end{proof}

\subsection{The parabolic case}
\label{ssec:parabolic asymptotics}

We state the following theorem, contained essentially in \cite[Proposition 2.3]{BM-JMPA}.

\begin{thm}[Asymptotics in the half-space]
\label{th:asymptotics in half-space}
Let $p \in (1, \infty]$ and let $H$ be the half-space in which $x_1 > 0$. Let $\Psi$ be given by 
$$
\Psi(x,t)=\sqrt{\frac{p'}{4\pi}}\int_{\frac{x_1}{\sqrt{t}}}^{\infty} e^{-\frac14 p' \si^2}\,d\si,
$$
for $x\in H $ and $t >0$.
\par
Then, it holds that 
$$
4t\log \left\{\Psi(x,t)\right\} = -p'\,d_\Ga(x)^2 + O(t \log t),
$$
uniformly for $x$ in every strip $\{ x\in\RR^N: 0\le x_1\le \de\}$ with $\de>0$.
\end{thm}

\begin{proof}
Observe that $d_\Ga(x) = x_1$. By employing a change of variables in the expression of $\Psi$, we get that
$$
\Psi(x,t)=e^{-\frac{p' x_1^2}{4t}}\,\sqrt{\frac{p'}{4\pi}}\int_0^\infty e^{-\frac12 p' \frac{x_1}{\sqrt{t}}\,\si-\frac14 p' \si^2}\,d\si,
$$
and hence
$$
4t \log\Psi(x,t)+p' x_1^2= 4t\,\log\left( \sqrt{\frac{p'}{4\pi}}\int_0^\infty e^{-\frac12 p' \frac{x_1}{\sqrt{t}}\,\si-\frac14 p' \te^2}\,d\si \right).
$$
\par
Thus, for $0\le x_1\le\de$, we have that
\begin{equation}
\label{eq:one-dimensional rate}
4t\,\log\left( \sqrt{\frac{p'}{4\pi}}\int_0^\infty e^{-\frac{p' \de}{2\sqrt{t}}\si-\frac{p' \si^2}{4}}\,d\si  \right)
\le
 4t \log\Psi(x,t)+p' x_1^2\le 0.
\end{equation}
\par
Integrating by parts on the two functions
$$
e^{-\frac{p'\de}{2\sqrt{t}}\si}\ \mbox{ and } \ e^{-\frac{p' \si^2}{4}}
$$
and a change of variable give that
$$
\int_{0}^{\infty} e^{-\frac{p' \de}{2\sqrt{t}}\si-\frac{p' \si^2}{4}}\,d\si 
=
\frac{2}{ p'\de}\sqrt{t} \left[ 1 + O(t) \right]
$$
%Integrating by parts on the two functions
%$$
%-\frac{2\sqrt{t}}{p'\de} e^{-\frac{p'\de}{2\sqrt{t}}\si}\ \mbox{ and } \ e^{-\frac{p'\si^2}4}
%$$ 
%and a change of variable give that
%$$
%\int_{0}^{\infty} e^{-\frac{p' \de}{2\sqrt{t}}\si-\frac{p' \si^2}{4}}\,d\si 
%=
%\frac{2}{ p'\de}\sqrt{t} \left[ 1 + O(t) \right]
%$$
as $t \to 0^+$  (see also \cite[Formula 7.1.23]{AS}). Then \eqref{eq:one-dimensional rate} gives the desired uniform convergence.

\end{proof}

In the case of the ball, the series representation of $u(x,t)$, established in Lemma \ref{lem:series solution},
is not convenient to obtain an asymptotic formula for $t\to 0$. We then proceed differently. To start with, we state the next lemma. 

\begin{lem}
\label{ball control from above}
Let $p \in (1,\infty]$. Assume that $u$ is the solution of \eqref{G-heat}-\eqref{boundary}.
\par
If $p \in (1,\infty)$, then, for every $x \in \ol{B_R}$, $t > 0$ and $\la > 0$, it holds that
\begin{equation}
\label{eq:ball control from above}
4t \log u(x,t) 
\leq \frac{4}{\la^2} -
\\
 \frac{4}{\la}\sqrt{p'} d_\Ga(x)
+
4 t \log
\left[
\frac
{\int_{0}^{\pi} (\sin \te)^\frac{N-p}{p-1}\, d\te}
{\int_{0}^{\pi} e^{-\frac{\sqrt{p'}}{\la t} (1-\cos\te) d_\Ga(x)}\,(\sin\te)^\frac{N-p}{p-1}\,d\te}
 \right].
\end{equation}
%\begin{multline}
%\label{eq:ball control from above}
%4t \log u(x,t) 
%\leq \frac{4}{\la^2} -
%\\
% \frac{4}{\la}\sqrt{p'} d_\Ga(x)
%+
%4 t \log
%\left[
%\frac
%{\int_{0}^{\pi} (\sin \te)^\frac{N-p}{p-1}\, d\te}
%{\int_{0}^{\pi} e^{-\frac{\sqrt{p'}}{\la t} (1-\cos\te) d_\Ga(x)}\,(\sin\te)^\frac{N-p}{p-1}\,d\te}
% \right].
%\end{multline}
\par
If $p=\infty$, for every $x\in\ol{B_R}$, $t > 0$ and $\la > 0$, we have that
\begin{equation}
\label{eq:infinity ball control from above}
4t\log u(x,t) \leq \frac{4}{\la^2} - \frac{4}{\la}d_\Ga(x) + 4t \log \left[ \frac{2}{1 + e^{-2 \frac{d_\Ga(x)}{\la t}}} \right].
\end{equation}
\end{lem}

\begin{proof}
Preliminarly, observe that if $x\in\Ga$, \eqref{eq:ball control from above} and \eqref{eq:infinity ball control from above} are obviously satisfied. 
\par
For a fixed $x\in B_R$, we argue as follows. By  Lemma \ref{radial laplace transform}, for every $\ve>0$ the function
$u^\ve$ defined in \eqref{eq:modified laplace transform} is the solution of \eqref{G-elliptic}-\eqref{elliptic-boundary} in $B_R$. Moreover, from \eqref{eq:monotonicity} in the proof of Lemma \ref{parabolic monotonicity}, we have that $t\mapsto u(x,t)$ is increasing.
\par
Thus, it holds that
$$
\ve^2 u(x,t)\,e^{-t/\ve^2}\le \int_t^\infty u(x,\tau)\,e^{-\tau/\ve^2}\,d\tau \leq \int_{0}^{\infty}u(x,\tau)e^{-\tau/\ve^2}\,d\tau = \ve^2 u^\ve(x),
$$
and hence
$$
u(x,t)\le u^\ve(x)\,e^{t/\ve^2};
$$
the last inequality holds for any $t, \ve>0$. Next, we choose
$\ve=\la\, t$ and obtain that
$$
u(x,t)\le u^{\la t}(x)\,e^{1/(\la^{2} t)} \ \mbox{ for any } \ t>0.
$$
Therefore, 
\begin{equation*}
\label{first lambda-inequality}
4t\,\log u(x,t)\le 4t\,\log u^{\la t}(x)+\frac4{\la^2}.
\end{equation*}
\par
Now, let $B^x=B_{d_\Ga(x)}(x)$, be the ball centered at $x$, with radius $d_\Ga(x)$. Let $u_{B^x}^\ve$ be the solution of \eqref{G-elliptic}-\eqref{elliptic-boundary} in $B^x$. By the comparison principle (Corollary \ref{cor:elliptic comparison}), we have that, for any $t,\la > 0$ and $x\in B_R$,
$$
u^{\la t} \leq u_{B^x}^{\la t}\ \mbox{ on } \ \ol{B^x}
$$
and, in particular,  
$$
u^{\la t}(x) \leq u_{B^x}^{\la t}(x).
$$
By using this fact and the previous inequality, we obtain that, for any $x\in B_R$, $t>0$ and $\la>0$,
%By the comparison principle (Corollary \ref{cor:elliptic comparison}), we can replace in \eqref{first lambda-inequality} at $x$, $u^{\la t}$ with $u_{B'}^{\la t}$, the solution of \eqref{G-elliptic}-\eqref{elliptic-boundary} in $B'$:
%%Now, for every $x\in \Om$, consider $B'$, the ball centered in $x$, with radius $d_\Ga(x)$ and $u^{\la t}_{B'}$, the solution of \eqref{G-elliptic}-\eqref{elliptic-boundary} in $B'$. Then, by applying the comparison principle (Corollary \ref{cor:elliptic comparison}), we have that $u^{\la t}(x) \leq u_{B'}^{\la t} (0)$. This fact and \eqref{first lambda-inequality} imply that
\begin{equation*}
\label{lambda-inequality}
4t\,\log u(x,t)\le 4t\,\log u_{B^x}^{\la t}(x)+\frac4{\la^2},
\end{equation*}
which, together with  \eqref{ball solution formula}, gives \eqref{eq:ball control from above} and \eqref{eq:infinity ball control from above}. Indeed, we observe that, by translating, $u^{\la t}_{B^x}(x)$ is simply the solution of \eqref{G-elliptic}-\eqref{elliptic-boundary} in $B_{d_\Ga(x)}$, with $\ve=\la t$, evaluated at the origin. 
\end{proof}

We are ready to obtain the uniform asymptotics for $u(x,t)$ in the ball. The pointwise formula is given in \cite[Theorem 2.6]{BM-JMPA}.

\begin{thm}[Asymptotics in the ball]
\label{th:parabolic ball asymptotic}
Set $p \in (1,\infty]$ and let $u(x,t)$ be the viscosity solution of \eqref{G-heat}-\eqref{boundary}. 
\par
Then, it holds that
\begin{equation}
\label{limit}
4t\log u(x,t) + p'\,d_\Ga(x)^2 = O \left( t \log t\right),
\end{equation}
as $t \to 0^+$, uniformly on $\ol{B_R}$.
\end{thm}

\begin{proof}
Given $x\in B_R$ there exists $y\in\Ga$ such that $|x-y|=d_\Ga(x)$ ($y$ is unique unless $x=0$). Let $H$ be the half-space containing $B_R$ and such that $\pa H\cap\Ga=\{ y\}$; notice that $d_\Ga(x)=d_{\pa H}(x)$.
\par 
Let $\Psi^y$ be the solution of \eqref{G-heat}-\eqref{boundary} in $H\times(0,\infty)$; since $B_R$ is contained in $H$, $\Psi^y$ obviously satisfies \eqref{G-heat} and \eqref{initial} for $B_R$ and, also, $\Psi^y\le 1$ on $\Ga\times(0,\infty)$.
By the comparison principle (Corollary \ref{cor:parabolic comparison}), we get that $u\ge\Psi^y$ and hence
\begin{equation*}
\label{half-space inequality}
4t\log u(x,t)\geq 4t\log\Psi^y(x,t) \ \mbox{ for } \ (x,t)\in \ol{B_R}\times(0,\infty).
\end{equation*}
Thus,  Theorem \ref{th:asymptotics in half-space} with $\de = 2R$, implies that
\begin{equation}
\label{liminf ball}
 4t\log u(x,t)\ge -p'\,d_{\pa H}(x)^2 + O(t \log t),
\end{equation}
uniformly on $\ol{B_R}$, as $t \to 0^+$, which is \eqref{limit} by one side, since $d_{\pa H}(x) = d_\Ga(x)$.
\par
On the other hand, by applying Lemma \ref{ball control from above} with  $ \la= \la^*>0$ such that 
$$
-\frac{4 \sqrt{p'} d_\Ga(x)}{\la^*}+\frac4{(\la^*)^2}=-p'\,d_\Ga(x)^2, \ \mbox{ that is } \ \la^*=\frac2{\sqrt{p'}\, d_\Ga(x)},
$$
we obtain, if $p\in(1,\infty)$,
\begin{equation*}
4\,t\log u(x,t)\le -p'\,d_\Ga(x)^2 + 4 \,t \log
\left[
\frac
{\int_{0}^{\pi} (\sin \te)^\frac{N-p}{p-1}\, d\te}
{\int_{0}^{\pi} e^{-\frac{p'}{2 t} (1-\cos\te) d_\Ga(x)^2}\,(\sin\te)^\frac{N-p}{p-1}\,d\te}
 \right],
\end{equation*}
and the obvious corresponding inequality, in the case $p=\infty$.
\par
Hence, by applying \eqref{eq:asymptotics-to-infinity-I} with $\si = p' d_\Ga(x)^2 / (2\, t)$, we conclude that
\begin{equation}
\label{limsup ball}
4\,t\log u(x,t)\le -p'\,d_\Ga(x)^2 + O\left ( t\log t\right),
\end{equation}
uniformly on $\ol{B_R}$, as $t \to 0^+$. Putting together \eqref{liminf ball} and \eqref{limsup ball}, we conclude the proof.
\end{proof}

%\afterpage{\blankpage}

%\afterpage{
%\thispagestyle{empty}
%}
 
\chapter{Varadhan-type formulas}
\label{ch:asymptotics i}

%\section{Introduction}
%\label{sec:intro asymptotic analysis}

In his paper \cite{Va}, S.R.S. Varadhan considered the following problems for the heat and resolvent equations: 
\begin{eqnarray*}
&u_t - \frac1{2}\tr\left[A \,\na^2 u\right] = 0\ &\mbox{ in } \ \Om\times(0,\infty),\\
&u=1\ &\mbox{ on } \ \Ga\times (0,\infty),\\
&u=0 \ &\mbox{ on } \ \Om\times\{0\},
\end{eqnarray*}
%$$
%u_t - \frac1{2}\tr\left[A \,\na^2 u\right] = 0\ \mbox{ in } \ \Om\times(0,\infty)
%$$
%such that \eqref{initial} and \eqref{boundary}, 
and
\begin{eqnarray*}
&\ve^{-2} u^\ve - \frac1{2}\tr\left[ A \, \na^2 u^\ve \right]=0\ &\mbox{ in } \ \Om,\\
&u^\ve=1 \ &\mbox{ on } \Ga,
\end{eqnarray*}
%$$
%\ve^{-2} u^\ve - \frac12\tr\left[ A \, \na^2 u^\ve \right]=0\ \mbox{ in } \ \Om,
%$$
%satisfying \eqref{G-elliptic}.
where $A=A(x)$, for $x\in\Om$, is a symmetric and positive definite $N\times N$ matrix.
\par 
When $A$ is uniformly elliptic and uniformly H\"older continuous, Varadhan proved the following asymptotic formulas: 
\begin{equation*}
-\lim_{t\to 0^+}2t\log u(x,t) = d_\Ga^A(x)^2
\end{equation*}
and
\begin{equation*}
-\lim_{\ve \to 0^+}\ve\log u^\ve(x) = \sqrt{2}\,d_\Ga^A(x),
\end{equation*}
where $d_\Ga^A$ is the distance from the boundary of $\Om$, induced by a Riemannian metric related to the matrix $A$. In particular, when $A$ is the identity matrix, $d_\Ga^A=d_\Ga$ coincides with the usual euclidean distance, defined by
$$
d_\Ga(x) = \inf\{|x-y|:y\in \Ga \}.
$$
\par
In this chapter, we establish asymptotic Varadhan-type formulas for the solutions of problems \eqref{G-heat}-\eqref{boundary} and \eqref{G-elliptic}-\eqref{elliptic-boundary}, in quite general domains. See Theorems \ref{th:parabolic-pointwise} and \ref{th:elliptic-pointwise}.
\par
In particular, we prove the pointwise formulas \eqref{eq:intro_BM-JMPA_varadhan} and \eqref{eq:intro_BM-AA_varadhan}, that is
\begin{equation*}
-\lim_{t\to 0^+} 4 t\log u(x,t) = p'\,d_\Ga(x)^2
\end{equation*}
and 
\begin{equation*}
-\lim_{\ve\to 0^+} \ve \log u^\ve(x) = \sqrt{p'} \, d_\Ga(x).
\end{equation*}
These formulas hold for a (not necessarily bounded) domain $\Om$ which merely satisfies the topological assumption $\Ga = \pa \left(\RR^N \setminus \ol \Om\right)$. 
\par
We also compute the uniform rate of convergence in \eqref{eq:intro_BM-JMPA_varadhan} and \eqref{eq:intro_BM-AA_varadhan}. In fact, we show that
\begin{equation*}
 4 t\log u(x,t) + p'\,d_\Ga(x)^2=O(t\log \psi(t)) \ \mbox{ as } \ t\to 0^+,
 \end{equation*}
 and, for $\ve \to 0$, that
 \begin{equation*}
 \ve\log u^\ve(x) + \sqrt{p'}\,d_\Ga(x) = 
 \begin{cases}
 \displaystyle 
  O\left(\ve\right) \ \mbox{ if } \ p=\infty,\\
 \displaystyle 
  O\left( \ve \log \ve \right) \ \mbox{ if } \ p \in (N, \infty),
  \end{cases}
  \end{equation*}
and
  \begin{equation*}
  \ve\log u^\ve(x) +\sqrt{p'}\, d_\Ga(x)=
  \begin{cases}
  \displaystyle 
  O\left( \ve \log |\log \psi(\ve)|\right)\ \mbox{ if } \ p=N,
  \\
  \displaystyle
  O\left(\ve \log \psi(\ve)\right) \ \mbox{ if } \ p\in(1,N),
  \end{cases}
  \end{equation*}
uniformly on every compact subset of $\ol\Om$. The function $\psi$ depends on appropriate regularity assumptions on $\Ga$. See Theorems  \ref{th:uniform-JMPA} and \ref{th:uniform-elliptic}.
\par
The results presented in this chapter are based on the construction of suitable barriers, which essentially employ the radial solutions, deduced in Chapter \ref{ch:explicit}. In Sections \ref{sec:parabolic barriers} and \ref{sec:elliptic barriers}, we present such barriers.  We stress the fact that no regularity assumption on $\Om$ is needed. 
\par
Sections \ref{sec:first-order} and \ref{sec:uniform} are then dedicated to deduce the already mentioned pointwise and uniform asymptotic formulas.

\section{Barriers in the parabolic case}
\label{sec:parabolic barriers}
%
%Preliminarily, we see a formula for the global solution.
%
%\begin{lem}[Asymptotics for the parabolic global solution]
%\label{lem:asymptotic global solution}
% Set, for $z \in \RR^N\setminus \ol\Om$,
%\begin{equation}
%\label{parabolic barrier from below}
%U(x,t) = \left\{ A_{N,p} d_\Ga(z)^\frac{N+ p -2}{p-1} \right\} \Phi(x-z,t) \ \mbox{ for } \ (x,t) \in \ol\Om \times (0,\infty),
%\end{equation}
%where
%$$
%A_{N,p} = \left\{ \frac{pe}{2(N + p -2)} \right\}^\frac{N + p -2}{2(p - 1)}.
%$$
%\par 
%Then, 
%\begin{equation}
%\label{U on boundary} 
%U \leq 1\ \mbox{ on }\  \Ga \times (0,\infty) \ \mbox{ and } \ U = 0 \ \mbox{ on } \ \Om \times \{0\}.
%\end{equation}
%\par 
%Moreover, we have that
%\begin{multline}
%\label{asymptotics for Phi}
%- 4 t \log U(x,t) = \\
%4t\,\log\left[A_{N,p}\,t^{-\frac{N+p-2}{2(p-1)}}\right]+4t\,\frac{N+p-2}{p-1}\,\log d_\Ga(z)-p'\,|z - x|^2.
%\end{multline}
%\end{lem}
%
%\begin{proof}
%Both \eqref{U on boundary} and \eqref{asymptotics for Phi} follows easily by definitions \eqref{parabolic barrier from below} and \eqref{eq:fundamental solution}. In particular, we observe that, for any $x\in\ol\Om$ and $t > 0$, $\Phi (x-z, t) \leq t^{-\frac{N+p-2}{2(p-1)}}e^{-p'\frac{d_\Ga(z)^2}{4 t}}$  and that 
%$$
%\max\{ t^{-\frac{N+p-2}{2(p-1)}}e^{-p'\frac{d_\Ga(z)^2}{4 t}}: t>0 \} = \left[A_{N,p} d_\Ga(z)^{\frac{N+p-2}{p-1}}\right]^{-1}.
%$$
%\end{proof}

The next two lemmas give global barriers from below and above for the solution of \eqref{G-heat}-\eqref{boundary}. 

\begin{lem}[A pointwise barrier from below]
\label{lem:parabolic control from below}
Set $1<p\leq\infty$. Let $\Om\subset \RR^N$ be a domain and let $z\in\RR^N\setminus\ol\Om$. Assume that $u(x,t)$ is the bounded (viscosity) solution of \eqref{G-heat}-\eqref{boundary}.
\par
Then, for every  $x\in\ol\Om$ and $t>0$, it holds that
\begin{equation}
\label{eq:control from above of -log}
4 t\log \left\{u(x,t)\right\} + p'\,|x - z|^2
 \geq
 4 t \,\log E^-\left(d_\Ga(z),t\right).
\end{equation} 
Here, for $\si > 0$ and $t > 0$, we mean that
\begin{equation}
\label{eq:chi-rate}
E^-\left(\si,t\right) =
A_{N,p} t^{-\frac{ N + p -2 }{ 2(p - 1) }}\si^\frac{ N + p - 2 }{ p - 1 },
\end{equation}
where
$$
A_{N,p} = \left\{ \frac{pe}{2(N + p -2)} \right\}^\frac{N + p -2}{2(p - 1)}.
$$
\end{lem}

\begin{proof}
The function $U^z$, defined, for $x\in\ol\Om$ and $t>0$, by
$$
U^z(x,t) = \left\{ A_{N,p} d_\Ga(z)^\frac{N+ p -2}{p-1} \right\} \Phi(x-z,t) \ \mbox{ for } \ (x,t) \in \ol\Om \times (0,\infty),
$$
 is a solution of \eqref{G-heat} in $\Om \times(0,\infty)$, since $\Phi$ is a global solution (Proposition \ref{prop:parabolic global solution}). Moreover, $U^z(x,0^+)=0$, for any $x\in \Om$ and, in virtue of the fact that
$$
\max\{ \Phi(x-z,t): x\in \Ga,\, t>0 \} = \left(A_{N,p} d_\Ga(z)^{\frac{N+p-2}{p-1}}\right)^{-1},
$$
we also have that $U^z \leq 1$ on $\Ga \times(0,\infty)$. 
\par
Hence, we just apply the comparison principle (Corollary \ref{cor:parabolic comparison}), to conclude that $ U^z \leq u$, on $\ol\Om \times (0,\infty)$, which implies  \eqref{eq:control from above of -log}, by recalling the definition of $\Phi$.
\end{proof}

\begin{lem}[An uniform barrier from above]
\label{lem:parabolic control from above}
Set $1 < p \leq \infty$. Let $\Om\subset\RR^N$ be a domain. Let $u$ be the bounded (viscosity) solution of \eqref{G-heat}-\eqref{boundary}.
\par
Then, for every $x\in\ol\Om$ and $ t > 0$, it holds that
\begin{equation}
\label{eq:control from below of - log}
4t\log \left\{u(x,t)\right\} +p'\,d_\Ga(x)^2
\leq
4 t \log E^+\left(d_\Ga(x), t\right),
\end{equation}
where $E^+\left(\si, t\right)$ is given by
\begin{equation*}
\label{eq:Ep}
E^+(\si, t)=
\begin{cases}
\displaystyle
\frac
{\int_{0}^{\pi} (\sin \te)^\frac{N-p}{p-1}\, d\te}
{\int_{0}^{\pi} e^{-\frac{p'}{2 t} (1-\cos\te) \si^2}\,(\sin\te)^\frac{N-p}{p-1}\,d\te} \ &\mbox{ if } \ p\in(1,\infty),
\vspace{8pt}\\
\displaystyle
\frac{2}{1 + e^{-\frac{p'}{t}\si^2}} \ &\mbox{ if } \ p=\infty,
\end{cases}
\end{equation*}
for $\si \geq 0$ and $t>0$.
\par
In particular, it holds that
\begin{equation*}
t \log E^+\left(d_\Ga(x), t\right)=
\begin{cases}
O(t\,\log t) \ \mbox{ if } p \in (1,\infty),\\
O( t)\ \mbox{ if } p=\infty,
\end{cases}
\end{equation*}
as $ t\to 0^+$, uniformly on every subset of $\ol\Om$ in which $d_\Ga$ is bounded.
\end{lem}

\begin{proof}
Let $u_B$ be the solution of  \eqref{G-heat}-\eqref{boundary} in the unit ball $B$. We prove that 
$
u(x,t)\leq u_B(0,t/d_\Ga(x)^2),
$
for $(x,t)\in\ol\Om\times(0,\infty)$, where we mean that  $u_B(0,t/d_\Ga(x)^2)=1$ when $x\in\Ga$. 
\par
Indeed, for $x\in\Ga$, the inequality is satisfied as an equality. Let $x\in\Om$ and let $v^x=v^x(y,t)$ be the solution of \eqref{G-heat}-\eqref{boundary} in $B^x\times(0,\infty)$, where $B^x$ is the ball centered at $x$ with radius $d_\Ga(x)$. Corollaries \ref{cor:parabolic maximum} and \ref{cor:parabolic comparison} give that
$$
u(y,t)\leq v^x(y,t)\ \mbox{ for every } \ (y,t)\in\ol{B^x}\times(0,\infty),
$$ 
and hence, in particular, $u(x,t)\leq v^x(x,t)$ for every $t>0$. Since $x$ is arbitrary in $\Om$, we infer that
\begin{equation}
\label{inequal}
u(x,t)\le v^x(x,t)\ \mbox{ for }\ (x,t)\in\Om\times(0,\infty).
\end{equation}
\par
Now, for fixed $x\in \Om$, consider the function defined by 
$$
w(y,t)=v^x(x+d_\Ga(x)\,y,d_\Ga(x)^2 t)\ \mbox{ for }\ (y,t)\in \ol{B}\times(0,\infty).
$$
By the translation and scaling invariance of \eqref{G-heat}, we have that $w$ satisfies the problem \eqref{G-heat}-\eqref{boundary} in $B$, and hence equals $u_B$ on $\ol{B}\times(0,\infty)$, by uniqueness.
\par
Therefore, evaluating $u_B$ at $(0, t/d_\Ga(x)^2) $ gives that
$$
u_B(0,t/d_\Ga(x)^2)=w(0,t/d_\Ga(x)^2)=v^x(x,t)\ge u(x,t),
$$
by \eqref{inequal}. 
\par
We conclude and obtain \eqref{eq:control from below of - log} by using \eqref{eq:ball control from above} with $R=1$, $x'=0$, $t'= \frac{t}{d_\Ga(x)^2}$. Indeed, in the case $p \in (1,\infty)$, we have that
\begin{multline*}
4 \left\{ \frac{ t }{ d_\Ga(x)^2 } \right\} \log v^x(x,t) =
4 \left\{ \frac{t}{d_\Ga(x)^2} \right\} \log u_B \left( 0, \frac{t}{d_\Ga(x)^2} \right) \le
\\
 \frac{4}{\la^2} - \frac{4}{\la} \sqrt{p'} + 4 \left\{ \frac{t}{d_\Ga(x)^2} \right\} \log 
 \left[ 
\frac
{\int_{0}^{\pi} (\sin \te)^\frac{N-p}{p-1}\, d\te}
{\int_{0}^{\pi} e^{-\frac{\sqrt{p'}}{\la t} (1-\cos\te)}\,(\sin\te)^\frac{N-p}{p-1}\,d\te}
\right].
\end{multline*}
Hence, by choosing $\la = \frac{2}{\sqrt{p'}}$, we conclude that
\begin{equation*}
4 t \log u_B\left( 0, \frac{t}{d_\Ga(x)^2} \right) \leq
 -p' d_\Ga(x)^2 + 4 t \log \left[ 
\frac
{\int_{0}^{\pi} (\sin \te)^\frac{N-p}{p-1}\, d\te}
{\int_{0}^{\pi} e^{-\frac{p'}{2 t} (1-\cos\te)d_\Ga(x)^2}\,(\sin\te)^\frac{N-p}{p-1}\,d\te}
\right],
\end{equation*}
which implies \eqref{eq:control from below of - log}. In the case $p=\infty$ we just need to replace \eqref{eq:ball control from above} with \eqref{eq:infinity ball control from above} to obtain the conclusion.
\end{proof} 

\section{Barriers in the elliptic case}
\label{sec:elliptic barriers}

The next two lemmas give explicit barriers for the solution $u^\ve$ of \eqref{G-elliptic}-\eqref{elliptic-boundary} in a general domain $\Om$. Compared to the parabolic case, here we obtain a sharper barrier from below, since we have a more convenient formula for the solution in the exterior of a ball.

\begin{lem}[{An elliptic barrier from below, \cite[Lemma 2.4]{BM-AA}}]
\label{elliptic-control-from-below}
Let $\Om\subset\RR^N$ be a domain and $p\in (1,\infty]$. Let $u^\ve$ be the bounded (viscosity) solution of \eqref{G-elliptic}-\eqref{elliptic-boundary}.  Pick $z\in\RR^N\setminus\ol\Om$.
\par
Then, we have that
$$
\ve\log\left\{u^\ve(x)\right\}+\sqrt{p'}\,\{|x-z|-d_\Ga(z)\}\geq \ve\log e_{p,z}^\ve(x)
\ \mbox{ for any } \ x\in\ol\Om,
$$
where
\begin{equation}
\label{error-from-below}
e_{p,z}^\ve(x)=
\begin{cases}
\displaystyle
\frac{\int_{0}^{\infty}e^{-\sqrt{p'}\frac{\cosh\te-1}{\ve}|x-z|}\left(\sinh\te\right)^\frac{N-p}{p-1}\,d\te}
{\int_{0}^{\infty}e^{-\sqrt{p'}\frac{\cosh\te-1}{\ve}d_\Ga(z)}\left(\sinh\te\right)^\frac{N-p}{p-1}\,d\te} \ &\mbox{ if }\ 1<p<\infty,
\vspace{7pt}\\
\displaystyle
1  \ &\mbox{ if } \ p=\infty.
\end{cases}
\end{equation}
\end{lem}

\begin{proof}
 We consider the ball $B=B_R(z)$ with radius $R=d_\Ga(z)$ and let $v^\ve$ be the bounded solution of \eqref{G-elliptic}-\eqref{elliptic-boundary} relative to $\RR^N\setminus\ol{B}\supset \Om$. From the fact that $z\in\RR^N\setminus\ol\Om$, we have that $\Ga\subset \RR^N\setminus B$, which implies that
$$
v^\ve\leq 1 \ \mbox{ on }\ \Ga,
$$
by the explicit expression of $v^\ve$ given in \eqref{solution-exterior}. Thus, by the comparison principle, we infer that $v^\ve\leq u^\ve$ on $\ol\Om$. The desired claim then follows by easy manipulations on \eqref{solution-exterior}.
\end{proof}

\begin{lem}[{An elliptic barrier from above, \cite[Lemma 2.3]{BM-AA}}]
  \label{elliptic-control-from-above}
  Let $\Om\subset \RR^N$ be a domain and $p\in (1,\infty]$. Let $u^\ve$ be the bounded (viscosity) solution of \eqref{G-elliptic}-\eqref{elliptic-boundary}. 
\par 
Then, we have that
$$
\ve\log\left\{u^\ve(x)\right\}+\sqrt{p'}\,d_\Ga(x)\leq \ve\log E_p^\ve\left(d_\Ga(x)\right),
$$
for every $x\in\ol\Om$, where, for $\si \ge 0$,
\begin{equation*}
E_p^\ve(\si)=
\begin{cases}
\displaystyle
\frac{\int_{0}^{\pi}(\sin\te)^\frac{N-p}{p-1}\,d\te}{\int_{0}^{\pi}e^{-\sqrt{p'}\frac{1-\cos\te}{\ve} \si}(\sin\te)^\frac{N-p}{p-1}\,d\te}\ &\mbox{ if }\ 1<p<\infty,
\vspace{8pt}\\
\displaystyle
\frac{2}{1+e^{-\frac{2 \si}{\ve}}}\ &\mbox{ if }\ p=\infty.
\end{cases}
\end{equation*}
\par
In particular, it holds that
\begin{equation*}
\ve \log  E_p^\ve\left(d_\Ga\right)=
\begin{cases}
O(\ve\,\log\ve) \ \mbox{ if } 1<p<\infty,\\
O(\ve)\ \mbox{ if } p=\infty,
\end{cases}
\end{equation*}
as $\ve\to 0^+$, on every subset of $\ol\Om$ in which $d_\Ga$ is bounded.
\end{lem}

\begin{proof}
For a fixed $x\in\Om$, we consider the ball $B^x=B_R(x)$ with $R=d_\Ga(x)$ and denote by $u_{B^x}^\ve$ the solution of \eqref{G-elliptic}-\eqref{elliptic-boundary} with $\Om=B^x$. The comparison principle gives that
$$
u^\ve\leq u_{B^x}^\ve\ \mbox{ on }\ \ol{B^x}
$$  
and, in particular, 
\begin{equation}
\label{eq:ball-inside}
u^\ve(x)\leq u_{B^x}^\ve(x).
\end{equation}
Observe that the uniqueness of  the solution of \eqref{G-elliptic}-\eqref{elliptic-boundary} and the scaling properties of $\De_p^G$ imply that
$$
u_{B^x}^\ve(x)=u_B^{\ve/R}(0),
$$
where $u_B^\ve$ is the solution of \eqref{G-elliptic}-\eqref{elliptic-boundary} with $\Om=B$, the unit ball. The explicit expressions in \eqref{ball solution formula} and \eqref{eq:ball-inside} then yield the pointwise estimate, since $R=d_\Ga(x)$.
\par
The last uniform formula then follows from \eqref{uniform-ball} in Theorem \ref{th:asymptotics in the ball}.
\end{proof}

\section{Pointwise Varadhan-type formulas}
\label{sec:first-order}

\begin{thm}
\label{th:parabolic-pointwise}
Set $p\in (1,\infty]$. Let $\Om$ be a domain in $\RR^N$, with boundary $\Ga$ such that $\Ga=\pa(\RR^N\setminus\ol\Om)$, and let $u$ be the bounded (viscosity) solution of \eqref{G-heat}-\eqref{boundary}.
\par 
Then, we have that
\begin{equation}
\label{pointwise-limit}
\lim_{t\to 0^+}4t\log \left\{u(x,t)\right\}=-p'\,d_\Ga(x)^2 \ \mbox{ for every } x\in\ol{\Om}.
\end{equation}
\end{thm}

\begin{proof}
It is clear that \eqref{pointwise-limit} holds for $x\in\Ga$. Let $x\in\Om$. We only need to apply Lemmas \ref{lem:parabolic control from below}  and \ref{lem:parabolic control from above}. Indeed, for $z \in \RR^N \setminus\ol\Om$, for every and $ t > 0$, Lemmas \ref{lem:parabolic control from below}  and \ref{lem:parabolic control from above} lead to
\begin{multline}
\label{eq:parabolic double control}
 p'\left( d_\Ga(x)^2 - |x - z|^2 \right)
 +4 t \,\log E^-\left(d_\Ga(z), t\right)\le
 \\
4t\log\left\{u(x,t)\right\} + p' \, d_\Ga(x)^2 \leq 
4 t \log E^+\left(d_\Ga(x), t\right).
\end{multline}
%\left[
%\frac
%{\int_{0}^{\pi} (\sin \te)^\frac{N-p}{p-1}\, d\te}
%{\int_{0}^{\pi} e^{-\frac{p'}{2\, t} (1-\cos\te) d^2_\Ga(x)}\,(\sin\te)^\frac{N-p}{p-1}\,d\te}
% \right]
\par
The last chain of inequalities implies at once that
\begin{multline*}
 p'\left(d_\Ga(x)^2 - |x-z|^2\right) \le \liminf_{t\to 0^+} \left[4t\log u(x,t) + p'\, d_\Ga(x)^2\right]\le \\
 \limsup_{t\to 0^+}\left[4t\log u(x,t) + p'\,d_\Ga(x)^2\right] \leq 0,
\end{multline*}
where we have used Lemma \ref{lem:parabolic control from above} and the fact that $t\log E^-\left(d_\Ga(z), t\right) \to 0$, as $t\to 0^+$.
\par
Since $\Ga = \pa \left( \RR^N \setminus\ol\Om \right)$, we can find always a sequence of $z_n$ that converges to a point $y\in\Ga$ such that $d_\Ga(x)=|x-y|$; by taking $z=z_n$ in the last formula and letting $n\to\infty$, we obtain the desired claim. 
%{
%\color{blue}
%\par
%We first compute the limit in \eqref{pointwise-limit}, by replacing $u$ by the barriers constructed in Lemmas \ref{below-barriers} and \ref{above-barrier}.
%In fact, from \eqref{below-barriers formula} and \eqref{global solution}, we easily compute that
%
%$$
%\lim_{t\to 0^+} 4t\,\log U^{z_n}(x,t)=-p' |x-z_n|^2,
%$$
%whereas, by Theorem \ref{Asymptotic in the ball}, we infer that 
%\begin{equation}
%\label{convergence-from-above}
%\lim_{t\to 0^+} 4t\,\log V(x,t)=-p'\,d_\Ga(x)^2.
%\end{equation}
%\par
%Now, for each $n\in\NN$, Lemmas \ref{below-barriers} and \ref{above-barrier} tell us that
%\begin{equation*}
%4t\log U^{z_n}(x,t)\leq 4t\log u(x,t)\leq 4t\log V(x,t),
%\end{equation*}
%for $t>0$. Thus, we get that
%\begin{multline*}
%-p' |x-z_n|^2=\liminf_{t\to 0^+}4t\log U^z(x,t)\le \liminf_{t\to 0^+} 4t\log u(x,t)\le \\
%\le \limsup_{t\to 0^+}4t\log u(x,t)\leq \limsup_{t\to 0^+}4t \log V(x,t)=-p'\,d_\Ga(x)^2.
%\end{multline*}
%Letting $z_n$ tend to $z$ gives the conclusion, since $|x-z|=d_\Ga(x)$.
%}
\end{proof}
%
%{
%\color{blue}
%\begin{rem}
%{\rm
%\label{rem:uniform-convergence-V}
%Notice that the convergence in \eqref{convergence-from-above} is uniform on every subset of $\ol\Om$ where the $d_\Ga$ is bounded; actually, we obtain that
%$$
%4t\log V(x,t)+p'\,d_\Ga(x)^2=O(t\log t),
%$$
%for $t\to 0^+$, on such subsets.  Indeed, by a comparison with the one-dimensional solution as done in Theorem \ref{Asymptotic in the ball}, we can write that
%\begin{equation}
%\label{from above}
%4t\, \log \Erfc\left(\frac{\sqrt{p'} d_\Ga(x)}{2 \sqrt{t}}\right)\le 4t\, \log V(x,t)
%\end{equation}
%while \eqref{lambda-inequality} with $\la=\frac{2}{\sqrt{p'}}$, \eqref{ball solution formula} and some manipulations give for $1<p<\infty$ that
%\begin{equation}
%  \label{eq:uniform-from-above}
%  4t\log V(x,t)\leq
%-p'\,d_\Ga(x)^2+4t\log\left[\frac{\int_{0}^{\pi}(\sin\te)^\al\,d\te}{\int_{0}^{\pi}e^{-p'\frac{1-\cos\te}{2t}d_\Ga(x)^2}(\sin\te)^\al\,d\te}\right].
%\end{equation}
%The explicit expressions in \eqref{from above} and \eqref{eq:uniform-from-above}, imply the desired claim. 
%\par 
%The case $p=\infty$ is analogous, simpler, and even yields the better behavior $O(t)$ as $t\to 0$.
%}
%\end{rem}
%}

\begin{thm}[{\cite[Theorem 2.5]{BM-AA}}]
\label{th:elliptic-pointwise}
Let $p\in (1, \infty]$ and $\Om$ be a domain in $\RR^N$ satisfying $\Ga=\pa\left(\RR^N\setminus\ol{\Om}\right)$; assume that $u^\ve$ is the bounded (viscosity) solution of \eqref{G-elliptic}-\eqref{elliptic-boundary}.
\par 
Then, it holds that
\begin{equation}
  \label{eq:pointwise-elliptic}
  \lim_{\ve\to 0^+}
\ve\log\left\{u^\ve(x)\right\}=
-\sqrt{p'}\,d_\Ga(x) \ \mbox{ for any } \ x\in\ol\Om.
\end{equation}
\end{thm}

\begin{proof}
 Given $z\in\RR^N\setminus\ol\Om$ and $\ve>0$, combining Lemmas \ref{elliptic-control-from-above} and  \ref{elliptic-control-from-below} gives at $x\in\ol{\Om}$ that
\begin{multline}
\label{barrier-chain}
\sqrt{p'}\left\{-|x-z|+d_\Ga(x)+d_\Ga(z)\right\}+\ve\log e_{p,z}^\ve(x)\leq\\
\ve\log\{u^\ve(x)\} +\sqrt{p'}\,d_\Ga(x)\leq
\ve \log E_p^\ve \left(d_\Ga(x)\right).
\end{multline}
Letting $\ve\to 0^+$ then gives that
\begin{multline*}
\sqrt{p'}\left\{ -|x-z|+d_\Ga(x)+d_\Ga(z)\right\}\le\\
 \liminf_{\ve\to 0^+}\left[ \ve\log\{u^\ve(x)\}+\sqrt{p'}\,d_\Ga(x)\right] \leq\\
\limsup_{\ve\to 0^+}\left[\ve\log\{u^\ve(x)\}+\sqrt{p'}\,d_\Ga(x)\right]\le
0, 
\end{multline*}
where we have used Lemma  \ref{elliptic-control-from-above} and the fact that $\ve\log e_{p,z}^\ve(x)$ vanishes, as $\ve\to 0^+$. This follows by applying \eqref{eq:asymptotics-to-infinity} to \eqref{error-from-below}.
\par 
We conclude the proof as in the previous theorem, by letting $z$ tend to $y\in\Ga$ such that  $|x-y|=d_\Ga(x)$.
\end{proof}

\section{Quantitative uniform formulas}
\label{sec:uniform}

For an open set of class $C^0$, we mean that its boundary is locally the graph of a continuous function. For the sequel, it is convenient to specify the modulus of continuity, by the following definition (used in \cite{BM-JMPA, BM-AA}). Let $\om:(0,\infty)\to (0,\infty)$ be a strictly increasing continuous function such that $\om(\tau)\to 0$ as $\tau\to 0^+$. We say that an open set $\Om$ is of class $C^{0,\om}$, if there exists a number $r>0$ such that, for every point $x_0\in\Ga$, there is a coordinate system $(y',y_N)\in\RR^{N-1}\times\RR$, and a function $\zi:\RR^{N-1}\to\RR$ such that
\begin{enumerate}[(i)]
\item
$B_r(x_0)\cap\Om=\{(y',y_N)\in B_r(x_0):y_N<\zi(y')\}$;
\item
$B_r(x_0)\cap\Ga=\{(y',y_N)\in B_r(x_0):y_N=\zi(y')\}$
\item
$|\zi(y')-\zi(z')|\le\om(|y'-z'|)$ for all $(y',\zi(y')), (z',\zi(z'))\in B_r(x_0)\cap\Ga$.
\end{enumerate}
In the sequel, it will be useful the function defined by 
$$
\psi_\om(\si)=\min_{0\leq s \leq r}\sqrt{s^2+\left[\om(s)-\si\right]^2}, \ \mbox{ for } \ \si\geq 0.
$$
This is the distance of the point $z=(0',\si)\in\RR^{N-1}\times\RR$ from the graph of the function $\om$.
Notice that 
\begin{equation}
\label{eq:C2domain psi}
\psi(\si)=\si \ \mbox{ if } \zi\in C^k \ \mbox{ with } \ k\ge 2
\end{equation}
 and, otherwise, $\psi(\si)\sim C\,\om^{-1}(\si)$, for some positive constant $C$, where $\om^{-1}$ is the inverse function of $\om$. For instance, if $\Om$ is of class $C^\al$, with $0<\al<1$ --- that means that $\Ga$ is locally a graph of an $\al$-H\"older continuous function --- then $\psi(\si)\ge a\,\si^{1/\al}$ as $\si\to 0^+$.

Figure \ref{fig:picture} describes the geometric setting considered throughout Section \ref{sec:first-order}.

\begin{figure}
  \centering
  \includegraphics[scale=0.6]{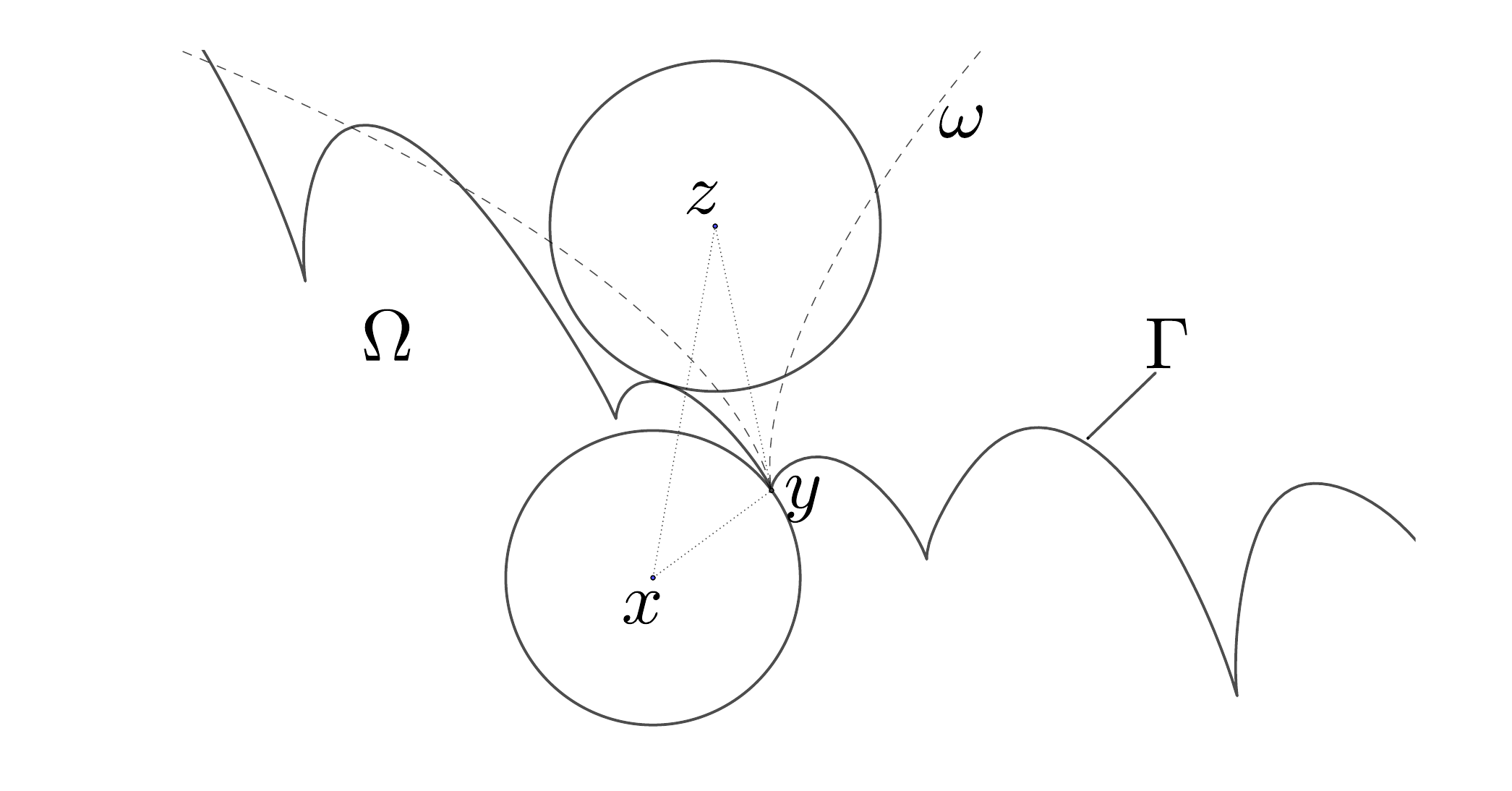} 
  \caption{The geometric description of the argument in the proof of Theorems \ref{th:uniform-JMPA} and \ref{th:uniform-elliptic}.}
  \label{fig:picture}
\end{figure}

\begin{thm}
\label{th:uniform-JMPA}
Let $ p \in (1,\infty]$ and suppose that $\Om$ is a domain of class $C^{0,\om}$. Let $u$ be the bounded (viscosity) solution of \eqref{G-heat}-\eqref{boundary}.
\par
Then, it holds that
\begin{equation}
\label{uniform-estimate}
4t\,\log u(x,t)+p'\,d_\Ga(x)^2=O(t \log \psi_\om(t)) \ \mbox{  as } \ t\to 0^+,
\end{equation}
uniformly on every compact subset of $\ol\Om$. 
In particular, if $t \log\psi_\om(t)\to 0$ as $t\to 0^+$,
then the solution $u$ of \eqref{G-heat}-\eqref{boundary} satisfies \eqref{pointwise-limit} uniformly on every compact subset of $\ol\Om$.
\end{thm}

\begin{proof}
We need to choose $z$, in \eqref{eq:parabolic double control}, uniformly with respect to $x\in\ol\Om$. For every $x\in\ol\Om$, we fix a coordinate system $(y',y_N)\in\RR^{N-1}\times\RR$, with its origin at a point in $\Ga$ at minimal distance $d_\Ga(x)$ from $x$. In this coordinate system, we choose $z(t)=(0',t)$ that, if $t$ is small enough is by construction a point in $\RR^N\setminus\ol{\Om}$, since $t>\zi(0')$. Also, by our assumptions on $\Om$, $d_\Ga(z(t))$ is bounded from below by the distance of $z(t)$ from the graph of the function $y'\mapsto\om(|y'|)$ defined for $y'\in\{y\in B_r(0): y_N=0\}$, that is
$$
d_\Ga(z(t))\ge \min_{0\le s\le r}\sqrt{s^2+[\om(s)-t]^2}.
$$ 
\par
It is clear that this construction does not depend on the particular point $x\in\ol{\Om}$ chosen, but only on the regularity assumptions on $\Om$. 
\par
Then, \eqref{eq:parabolic double control} reads as
\begin{multline*}
 p'\left( d_\Ga(x)^2 - |x - z(t)|^2 \right)
 +4 t \,\log E^-(d_\Ga(z(t)), t) \le
 \\
4t\log\left\{u(x,t)\right\} + p' \, d_\Ga(x)^2 \leq 
4 t \log
E^+\left( d_\Ga(x), t\right),
 \end{multline*}
for every $x\in\ol\Om$ and $t>0$.
 \par
Observe that $d_\Ga(z(t)) \geq \psi_\om(t)$ and $|x-z(t)| \leq d_\Ga(x) + |y - z(t)|$. Hence, if $x$ is such that $d_\Ga(x) \leq \de$, we have that
\begin{multline*}
 -p' \left( |y-z(t)|^2 + 2\de|y-z(t)| \right)
 +4 t \,\log E^-(\psi_\om(t), t) \le
 \\
4t\log\left\{u(x,t)\right\} + p' \, d_\Ga(x)^2
\leq 
4 t \log E^+\left(\de, t\right).
\end{multline*}
\par
From Lemma \ref{lem:parabolic control from above}, the last term is $O(t \log t)$, as $t\to 0^+$.  Whereas, from the choice of $z(t)$, the first term can be read as $-p' (t^2 + 2\de t) + 4 t\log E^-(\psi_\om(t),t)$ and hence its leading term is due to $4 t\log E^-(\psi_\om(t),t)$, which is a $O\left( t \log \psi_\om(t) \right)$, thanks to \eqref{eq:chi-rate}.
\\
\end{proof}

%\begin{rem}
%{\rm
%Under sufficient assumptions on $\om$, we can replace $\psi$ by $a\,\om^{-1}$, for some positive constant $a$, where $\om^{-1}$ is the inverse function of $\om$. For instance, if $\Om$ is of class $C^\al$, with $0<\al<1$ --- that means that $\Ga$ is locally a graph of an $\al$-H\"older continuous function --- then the assumptions of Theorem \ref{th:uniform-JMPA} are fulfilled, since $\psi(t)\ge a\,t^{1/\al}$ as $t\to 0^+$.
%}
%\end{rem}

The same assertion of Theorem \ref{th:uniform-JMPA} holds true even if we replace $1$ in  \eqref{boundary} by a bounded time-dependent non-constant  boundary data, provided that this is bounded away from zero.

\begin{cor}
\label{positive boundary data}
Let $w$ be the bounded solution of \eqref{G-heat}, \eqref{initial} satisfying
$$
w=h\ \mbox{ on }\ \Ga\times(0,\infty),
$$
where the function $h:\Ga\times(0,\infty)\to\RR$ is such that 
$$
\ul{h}\leq h \leq \ol{h} \ \mbox{ on } \ \Ga\times(0,\infty),
$$ 
for some positive numbers $\ul{h}, \ol{h}$.
\par 
Then, we have that
$$
4 t \log w(x,t)= -p'd_\Ga(x)^2+O(t\log \psi_\om(t)) \ \mbox{ as }\ t\to 0^+,
$$
uniformly on every compact subset of $\ol\Om$.
\end{cor}	
\begin{proof}
Since $\ul{h}\, u \leq w \leq \ol{h}\, u$ on $\Ga\times(0,\infty)$, we can apply Corollary \ref{cor:parabolic comparison} to get:
$$
\ul{h}\, u(x,t)\leq w(x,t) \leq \ol{h}\, u(x,t) \ \mbox{ on } \ \ol\Om\times(0,\infty).
$$
This implies that, for every $x\in\ol\Om$ and $t>0$,
\begin{equation*}
  4t \log\ul{h}+4 t \log u(x,t) \leq 4 t \log w(x,t) \leq  4t \log\ol{h}+4 t \log u(x,t).
\end{equation*}
The conclusion then easily follows from Theorem \ref{th:uniform-JMPA}.
\end{proof}

\begin{thm}[{\cite[Theorem 2.6]{BM-AA}}]
\label{th:uniform-elliptic}
Let $ p \in (1,\infty]$ and $\Om$ be a domain of class $C^0$. Suppose that $u^\ve$ is the bounded (viscosity) solution of \eqref{G-elliptic}-\eqref{elliptic-boundary}.
\par
Then, as $\ve\to 0^+$, we have that 
\begin{equation}
\label{eq:uniform-convergence-infinity}
\ve\log\left\{u^\ve(x)\right\}+ \sqrt{p'}\,d_\Ga(x)=
\begin{cases}
O(\ve) \ &\mbox{ if } \ p=\infty, \\
O(\ve\log \ve) \ &\mbox{ if } \ p>N.
\end{cases}
\end{equation}
Moreover,  if $\Om$ is of class $C^{0,\om}$, it holds that
\begin{equation}
  \label{eq:uniform-convergence}
  \ve\,\log \left\{u^\ve(x)\right\}+\sqrt{p'}\,d_\Ga(x)=
  \begin{cases}
  \displaystyle
    O(\ve\log|\log\psi_\om(\ve)|)\ &\mbox{ if } \ N=p, \\  
\displaystyle 
    O(\ve\log\psi_\om(\ve))\ &\mbox{ if }\ 1<p<N.
  \end{cases}
\end{equation}
The formulas \eqref{eq:uniform-convergence-infinity} and \eqref{eq:uniform-convergence} hold uniformly on the compact subsets of $\ol\Om$.
\par
In particular, if $\ve\log\psi_\om(\ve)\to 0$ as $\ve\to 0^+$, then the convergence in \eqref{eq:pointwise-elliptic} is uniform on every compact subset of $\ol\Om$.
\end{thm}

\begin{proof}
For any fixed compact subset $K$ of $\ol\Om$ we let $d$ be the positive number, defined as  
$$
d=
\max_{x'\in K}\{d_\Ga(x'), |x'|\}.
$$ 
To obtain the uniform convergence in \eqref{eq:pointwise-elliptic} we will choose $z=z_\ve$ independently on $x\in K$, as follows.
\par 
If $\Om$ is of class $C^{0,\om}$, fix $x\in K$, take $y\in\Ga$ minimizing the distance to $x$, and 
consider a coordinate system in $\RR^{N-1}\times\RR$ such that $y=(0', 0)$. If we take $z_\ve=(0',\ve)$, then $z_\ve\in\RR^N\setminus\ol\Om$ when $\ve$ is sufficiently small. 
With this choice, \eqref{barrier-chain} reads as
\begin{multline*}
\sqrt{p'}\left\{-|x-z_\ve|+d_\Ga(x)+d_\Ga(z_\ve)\right\}+\ve\log e_{p,z_\ve}^\ve(x)\leq\\
\ve\log u^\ve(x) +\sqrt{p'}\,d_\Ga(x)\leq
\ve \log E_p^\ve\left(d_\Ga(x)\right).
\end{multline*}
Hence, we get:
\begin{equation*}
\label{barrier-chain-uniform}
-\sqrt{p'}\,\ve+\ve\log e_{p,z_\ve}^\ve(x)\leq
\ve\log u^\ve(x) +\sqrt{p'}\,d_\Ga(x)\leq
\ve \log E_p^\ve\left(d_\Ga(x)\right),
\end{equation*}
since $d_\Ga(z_\ve)\ge 0$ and $|x-z_\ve|\le d_\Ga(x)+\ve$. 
\par
Thus, if $p=\infty$, Lemmas \ref{elliptic-control-from-above} and \eqref{elliptic-control-from-below} give that
$$
-\ve\le \ve \log \left\{u^\ve(x)\right\}+d_\Ga(x)\le \ve\log\left\{\frac{2}{1+e^{-\frac{d}{\ve}}}\right\},
$$
being $d_\Ga(x)\le d$, and \eqref{eq:uniform-convergence-infinity} follows at once.

Next, 
if $1<p<\infty$, we recall that $\ve \log E_p^\ve\left(d_\Ga(x)\right)=O(\ve\log \ve)$ on $K$ as $\ve\to 0^+$, by Lemma \ref{elliptic-control-from-above}. On the other hand, by observing that $d_\Ga(z_\ve)\ge \psi(\ve)$, by our assumption on $\Om$, and that also $|x-z_\ve|\le 2 d$ for $\ve\le d$,
\eqref{error-from-below} gives on $K$ that
$$
  e_{p,z_\ve}^\ve\ge
 \frac{\int_{0}^{\infty}e^{-\frac{2 d\,\sqrt{p'}}{\ve}(\cosh\te-1)}\left(\sinh\te\right)^\frac{N-p}{p-1}d\te}
{\int_{0}^{\infty}e^{-\frac{\sqrt{p'}\psi(\ve)}{\ve}(\cosh\te-1)}\left(\sinh\te\right)^\frac{N-p}{p-1}d\te}.
$$
\par
Now, after setting $\al=\frac{N-p}{2(p-1)}$, to this formula we apply \eqref{eq:asymptotics-to-infinity} with $\si=2 d \sqrt{p'}/\ve$ at the numerator and \eqref{eq:asymptotic to zero} $\si=\sqrt{p'} \psi(\ve)/\ve$ at the denominator. Thus, since the sign of $\al$ is that of $N-p$, on $K$ we have as $\ve\to 0$  that
\begin{equation*}
 \ve\log\left(e_{p,z_\ve}^\ve\right)\ge \al\,\ve\log\psi(\ve)-\frac{\al-1}{2}\,\ve\log\ve+O(\ve)=\al\ve\log\psi(\ve)+O(\ve\log\ve),
\end{equation*}
if $p<N$,
$$
 \ve\log\left(e_{p,z_\ve}^\ve\right)\ge -\ve\,\log|\log\psi(\ve)|+O(\ve \log\ve),
$$
if $p=N$, and
$$
 \ve\log\left(e_{p,z_\ve}^\ve\right)\ge \frac{\al+1}{2}\,\ve\log\ve+O(\ve),
$$
if $p>N$.
\end{proof}

\begin{cor}
\label{cor:elliptic positive boundary data}
Let $v^\ve:\Om\to \RR$ be the bounded solution of \eqref{G-elliptic} satisfying
$$
v^\ve=h_\ve\ \mbox{ on }\ \Ga\times(0,\infty),
$$
where, for any $\ve>0$, the function $h_\ve:\Ga\to\RR$ is such that
$$
\ul{h}\leq h_\ve \leq \ol{h} \ \mbox{ on } \ \Ga,
$$ 
for some positive numbers $\ul{h}, \ol{h}$.
\par 
Then, we have that
\begin{equation*}
\ve \log v^\ve (x) + \sqrt{p'}\,d_\Ga(x) = 
\begin{cases}
\displaystyle 
O(\ve) \ \mbox{ if } \ p=\infty,
\\
\displaystyle 
O\left( \ve \log \ve \right) \ \mbox{ if } \ p\in (N,\infty),
\\
\displaystyle 
O\left( \ve \log |\log \psi_\om(\ve)| \right) \ \mbox{ if } \ p=N,
\\
\displaystyle
O \left( \ve \log \psi_\om (\ve) \right) \ \mbox{ if } p\in (1,N).
\end{cases}
\end{equation*}
uniformly on every compact subset of $\ol\Om$.
\end{cor}	

\begin{proof}
The proof runs similarly to that of Corollary \ref{positive boundary data}, having in mind Theorem \ref{th:uniform-elliptic} instead of Theorem \ref{th:uniform-JMPA}.
\end{proof}

%\afterpage{
%\thispagestyle{empty}
%}
%\afterpage{\blankpage}

%\cleardoublepage\null

\chapter{Asymptotics for $q$-means}
\label{ch:asymptotics ii}

In this chapter, we mainly consider domains of class $C^2$. For these domains, we are able to provide further asymptotic formulas involving solutions of \eqref{G-heat}-\eqref{boundary} and \eqref{G-elliptic}-\eqref{elliptic-boundary}. In particular, we consider statistical nonlinear quantities, called {\it $q$-means}, defined, in the case of our interest, as follows. Given $q\in [1,\infty]$, $B$ a ball in $\RR^N$ and a function $u:B\to \RR$, the $q$-mean of $u$ on $B$ is the unique real value $\mu$, such that
\begin{equation}
\label{eq:def-qmean}
\Vert u-\mu\Vert_{L^q\left(B\right)} \leq \Vert u - \la \Vert_{L^q\left(B\right)}, \ \mbox{ for } \ \la \in\RR.
\end{equation}

These quantities generalize the standard mean value, which corresponds to the case $q=2$. These means (there named $p$-means) have also been studied by Ishiwata, Magnanini and Wadade, in \cite{IMW}, in connection with asymptotic mean value properties for $p$-harmonic functions.

Formulas that we give here are proper generalizations of those due to Magnanini and Sakaguchi (see \cite{MS-AM} and \cite{MS-PRSE}) concerning the linear cases. In \cite{MS-PRSE},  the solution of the heat equation subject to conditions \eqref{initial} and \eqref{boundary} is considered and the following formula for the mean value of $u$ on a ball touching the boundary $\Ga$ at only one point $y_x\in\Ga$ is proved:
$$
\lim_{t\to 0^+}
\left(\frac{R^2}{t}\right)^{\frac{N+1}{4}}\dashint_{B_R(x)}u(z,t)\,dz=
\frac{C_N}{ \sqrt{\Pi_\Gamma(y_x)}}.
$$
Here, $C_N$ is a positive constant, $x\in\Om$, $R=d_\Ga(x)$ and $\{ y_x\}=\ol{B_R(x)}\cap\Ga$. Also,
\begin{equation}
\label{eq:Pi-gamma}
\Pi_\Ga(y)= \prod_{j=1}^{N-1} \left[ 1 - R\ka_j(y) \right], \ \mbox{ for } \ y\in \Ga,
\end{equation}
where $\ka_1(y),\dots, \ka_{N-1}(y)$ denote the principal curvatures of $\Ga$ at $y$ with respect to the interior normal direction to $\Ga$. In \cite{MS-AM}, a corresponding elliptic case is considered.
\par
In Section \ref{sec:second-order}, we consider the $q$-mean $\mu_q(x,t)$ on $B_R(x)$ of the solution of \eqref{G-heat}-\eqref{boundary} and obtain the formula:
$$
\lim_{t\to 0^+}
 \left(\frac{R^2}{t}\right)^{\frac{N+1}{4(q-1)}}\mu_{q}(x,t) =
 C_{N,p,q}\,\left\{ \Pi_\Gamma(y_x)\right\}^{-\frac{1}{2(q-1)}},
$$
This formula holds for any $p\in(1,\infty]$ and $q\in(1,\infty)$. The positive constant $C_{N,p,q}$ will be specified in Theorem \ref{th:JMPA-asymptotics-qmean}.
\par
In the elliptic case, we consider the $q$-mean of the solution of \eqref{G-elliptic}-\eqref{elliptic-boundary} and, for the same values of $p$ and $q$, we compute:
$$
\lim_{\ve\to 0^+}
\left(\frac{R}{\ve}\right)^{\frac{N+1}{2(q-1)}}\mu_{q,\ve}(x)=
   \widetilde{C}_{N,p,q}\,\left\{\,\Pi_\Gamma(y_x)\right\}^{-\frac{1}{2(q-1)}},
  $$
The value of $\widetilde{C}_{N,p,q}$ can be found in Theorem \ref{th:AA-qmean}.
\par
In Theorems \ref{th:JMPA-asymptotics-qmean} and \ref{th:AA-qmean} also the extremal case in which $q=\infty$ will be treated obtaining that
$$
\lim_{t\to 0^+} \mu_\infty(x,t) = \lim_{\ve\to 0^+} \mu_{\infty,\ve}(x) = \frac1{2}.
$$
\par
The above limits are obtained by using improved versions of the barriers we have constructed in Chapter \ref{ch:asymptotics i}. These versions, that are valid for $C^2$-regular domains, are presented in Sections \ref{sec:parabolic enhanced barriers} and \ref{sec:elliptic enhanced barriers}. In Section \ref{sec:second-order}, we first prove the asymptotic formulas for the improved barriers (see Lemmas \ref{lem:JMPA-barriers-qmean} and \ref{lem:AA-barriers-qmean}) and hence thanks to appropriate properties of monotonicity of the $q$-mean, we extend the formulas to the relevant solutions.
\par
It is worth noting that the results of this chapter are based on Lemma \ref{lem:geometric-asymptotics}, a geometrical lemma proved in \cite{MS-PRSE}. In Section \ref{sec:second-order}, we recall it, from \cite{MS-PRSE}, with its complete proof.

\section{Improving of barriers in the parabolic case}
\label{sec:parabolic enhanced barriers}

The following lemma is a consequence of Theorem  \ref{th:uniform-JMPA}.

%We can refine the barriers given in Section \ref{sec:parabolic barriers}, as a consequence of Theorem \ref{th:uniform-JMPA}.

\begin{lem}[{\cite[Corollary 2.12]{BM-JMPA}}]
\label{lem:JMPA-enhanced barriers}
Set $p\in(1,\infty]$. 
Let $\Om$ be a domain of class $C^{0,\om}$. Let $v: \ol\Om\times(0,\infty)\to\RR$ be defined by 
$$
\Erfc\left(\frac{\sqrt{p'}\, v(x,t)}{2 \sqrt{t}}\right)= u(x,t)\ \mbox{ for } \ (x,t)\in\ol{\Om}\times(0,\infty),
$$
where $u(x,t)$ is the (bounded) viscosity solution of \eqref{G-heat}-\eqref{boundary}.
\par
Then, 
$$
v(x,t)=d_\Ga(x)+O(t\log \psi_\om(t)) \ \mbox{ as } \ t\to 0^+,
$$
uniformly on every compact subset of $\ol\Om$.
\end{lem}

\begin{proof}
From the definition of $v(x,t)$, operating as in the proof of Theorem \ref{th:asymptotics in half-space} yields that
$$
4t\,\log u(x,t)+p'\,v(x,t)^2=4t\,\log\left( \sqrt{\frac{p'}{4\pi}}\int_0^\infty e^{-\frac12 p' \frac{v(x,t)}{\sqrt{t}}\,\si-\frac14 p' \si^2}d\si \right)\le 0.
$$
\par
By this inequality, since the first summand at the left-hand side converges uniformly on every compact $K\subset\ol{\Om}$ as $t\to 0^+$, we can infer that there exist $\ol{t}>0$ and $\de>0$ such that $0\le v(x,t)\le\de$ for any $x\in K$ and $0<t<\ol{t}$. 
\par
Thus, for $x\in K$ we have that 
\begin{multline*}
-\left[4t\,\log u(x,t)+p'\,d_\Ga(x)^2\right]+4t\,\log\left( \sqrt{\frac{p'}{\pi}}\int_0^\infty e^{-\frac12 p' \frac{\de}{\sqrt{t}}\,\si-\frac14 p' \si^2}d\si \right)\le \\
p'\,\left[v(x,t)^2-d_\Ga(x)^2\right]\le -\left[4t\,\log u(x,t)+p'\,d_\Ga(x)^2\right],
\end{multline*}
which implies the desired uniform estimate, by means of \eqref{uniform-estimate}.
\end{proof}

We can now refine the barriers given in Section \ref{sec:parabolic barriers}. We define a function of $t$ by
\begin{equation}
\label{eq:eta-K}
\eta_{u,K}(t) = \frac1{\sqrt{t}}\max\left\{ |v(x,t) - d_\Ga(x)| : x\in K \right\}\ \mbox{ for } \ t>0.
\end{equation}

\begin{cor}
\label{cor:parabolic-barriers}
Set $p\in(1,\infty]$. Let $\Om$ be a $C^2$ domain. For any compact set $K\subseteq \ol\Om$, we have that
\begin{equation}
\label{eq:parabolic-barriers}
\Erfc\left( \sqrt{\frac{p'}{4t}} d_\Ga(x) + \eta_{u,K}(t) \right) 
\leq 
u(x,t)
\leq
\Erfc\left( \sqrt{\frac{p'}{4t}} d_\Ga(x)  - \eta_{u,K}(t) \right)
\end{equation}
for $(x,t) \in K\times(0,\infty)$. It holds that $\eta_{u,K}(t) = O\left(\sqrt{t} \log t\right)$, as $t\to 0^+$.
\end{cor}

\begin{proof}
We use Lemma \ref{lem:JMPA-enhanced barriers} and \eqref{eq:eta-K} to obtain \eqref{eq:parabolic-barriers}. The asymptotic profile of $\eta_{u,K}$ follows from \eqref{eq:C2domain psi}.
\end{proof}
%
%\begin{rem}
%\label{rem:parabolic-barriers}
%{
%\rm
%In the case in which $\Ga$ is of class $C^2$, for any compact set $K\subseteq \ol\Om$ we set
%\begin{equation}
%\label{eq:eta-K}
%\eta(t) = \frac1{\sqrt{t}}\max\left\{ |v(x,t) - d_\Ga(x)| : x\in K \right\}\ \mbox{ for } \ t>0,
%\end{equation}
%and observe that $\eta(t) = O \left(\sqrt{t} \log t \right)$, as $t\to 0^+$ {\color{magenta} from \eqref{eq:C2domain psi}}. Therefore, we obtain the formula:
%\begin{equation}
%\label{eq:parabolic-barriers}
%\Erfc\left( \sqrt{\frac{p'}{4t}} d_\Ga(x) + \eta(t) \right) 
%\leq 
%u(x,t)
%\leq
%\Erfc\left( \sqrt{\frac{p'}{4t}} d_\Ga(x)  - \eta(t) \right)
%\end{equation}
%for $(x,t) \in K\times(0,\infty)$.
%}
%\end{rem}

\section{Improving of barriers in the elliptic case}
\label{sec:elliptic enhanced barriers}
To start with, we recall that a domain $\Om$ of class $C^2$ satisfies both the uniform exterior and interior ball conditions, i.e. there exist $r_i,r_e>0$ such that every $y\in\Ga$ has the property that there exist $z_i\in\Om$ and $z_e\in\RR^{N}\setminus\ol{\Om}$ for which 
\begin{equation}
\label{uniform-balls-condition}
B_{r_i}(z_i)\subset \Om \subset \RR^{N}\setminus \ol{B}_{r_e}(z_e)\ \mbox{ and } \ \ol{B}_{r_i}(z_i)\cap\ol{B}_{r_e}(z_e)=\{y\}.
\end{equation}

%Throughout this subsection we assume that $\Om$ is a (not necessarily bounded) domain that satisfies  both the uniform exterior and interior ball conditions, i.e. there exist $r_i,r_e>0$ such that every $y\in\Ga$ has the property that there exist $z_i\in\Om$ and $z_e\in\RR^{N}\setminus\ol{\Om}$ for which 
%\begin{equation}
%\label{uniform-balls-condition}
%B_{r_i}(z_i)\subset \Om \subset \RR^{N}\setminus \ol{B}_{r_e}(z_e)\ \mbox{ and } \ \ol{B}_{r_i}(z_i)\cap\ol{B}_{r_e}(z_e)=\{y\}.
%\end{equation}

We will also use two families of probability measures on the intervals $[0, \infty)$ and $[0,\pi]$ with densities defined, respectively, by
\begin{eqnarray*}
&&\displaystyle d\nu^\tau(\te)=\frac{e^{-\tau\,(\cosh\te-1)}(\sinh\te)^\al}{\int_{0}^{\infty}e^{-\tau (\cosh\te-1)}(\sinh\te)^\al\,d\te}\,d\te,\\
&&\displaystyle d\mu^\tau(\te)=\frac{e^{-\tau\,(1-\cos\te)}(\sin\te)^\al}{\int_{0}^{\pi}e^{-\tau (1-\cos\te)}(\sin\te)^\al\,d\te}\,d\te.
\end{eqnarray*}

\begin{lem}[{\cite[Lemma 3.1]{BM-AA}}]
  \label{barriers-for-c2-domain}
  Set $p\in(1,\infty]$.  Let $\Om\subset \RR^N$ be a $C^2$ domain. Assume that $u^\ve$ is the bounded (viscosity) solution of \eqref{G-elliptic}-\eqref{elliptic-boundary}.
\par 
If $p\in (1,\infty)$, we set for $\tau_\ve=\sqrt{p'} r_e/\ve$:
$$
U^\ve(\si)=
\int_{0}^{\infty}e^{-\si\cosh\te} d\nu^{\tau_\ve}(\te), \quad \si\ge 0,
$$
and
\begin{equation*}
V^\ve(\si)=
\begin{cases}
\displaystyle
\int_{0}^{\pi}e^{-\si\cos\te} d\mu^{\tau_\ve}(\te)\ &\mbox { if }\ 0\le \si < \tau_\ve, \\
\vspace{-10pt}\\ 
\displaystyle
\left\{\int_{0}^{\pi}e^{-\si\cos\te} d\mu^0(\te)\right\}^{-1}\ &\mbox{ if }\ \si\ge \tau_\ve.
\end{cases}
\end{equation*}
If $p=\infty$, we set $U^\ve(\si)=e^{-\si}$ and 
\begin{equation*}
V^\ve(\si)=
\begin{cases}
\displaystyle
\frac{\cosh(\tau_\ve-\si)}{\cosh\tau_\ve}\ &\mbox{ if } \ 0\le \si < \tau_\ve, \\
\vspace{-8pt}\\
1/\cosh\si\ &\mbox{ if }\ \si\ge \tau_\ve.
\end{cases}
\end{equation*}
\par
Then, we have that
\begin{equation}
\label{eq:U-V-barriers}
U^\ve\left(\frac{d_\Ga(x)}{\ve/\sqrt{p'}}\right)\le 
u^\ve(x)\le
V^\ve\left(\frac{d_\Ga(x)}{\ve/\sqrt{p'}}\right),
\end{equation}
for any $x\in\ol\Om$.
\end{lem}

\begin{proof}
  Let $p\in(1,\infty)$. For any $x\in\Om$ we can consider $y\in \Ga$ such that $|x-y|=d_\Ga(x)$. From the assumptions on $\Om$ there exists $z_e\in\RR^N\setminus\ol\Om$ such that \eqref{uniform-balls-condition} holds for $y$. As seen in the proof of Lemma \ref{elliptic-control-from-below}, by using the comparison principle and the explicit expression \eqref{solution-exterior}, we obtain
\begin{equation*}
u^\ve(x)\ge\frac{\int_0^\infty e^{-\sqrt{p'} |x-z_e|/\ve\,\cosh\te}(\sinh\te)^\al\,d\te}{\int_{0}^{\infty}e^{-\sqrt{p'} r_e/\ve\,\cosh\te}(\sinh\te)^\al\,d\te}.
\end{equation*}
Thus, the fact that $|x-z_e|=d_\Ga(x)+r_e$ gives the first inequality in \eqref{eq:U-V-barriers}, by recalling the definition of $U^\ve$.
\par 
To obtain the second inequality in \eqref{eq:U-V-barriers} we proceed differently whether $x\in\Om_{r_i}$ or not. Indeed, if $x\in\Om_{r_i}$, there exists $z_i\in\Om$ such that \eqref{uniform-balls-condition} holds for some $y\in\Ga$ and $x\in B_{r_i}(z_i)$; moreover, since $\pa B_{d_\Ga(x)}(x)\cap \pa B_{r_i}(z_i)=\{y\}$, we observe that $x$ lies in the segment joining $y$ to $z_i$, and hence $|x-z_i|=r_i-d_\Ga(x)$. Again, by using the comparison principle and the expression in \eqref{ball solution formula}, we get that 
\begin{equation*}
u^\ve(x)\le
\frac{\int_{0}^{\pi}e^{\sqrt{p'}\,\cos\te\,\frac{|x-z_i|}{\ve}}(\sin\te)^\al\,d\te}{\int_{0}^{\pi}e^{\sqrt{p'}\,\cos\te\,\frac{r_i}{\ve}}(\sin\te)^\al\,d\te},
\end{equation*}
that, by using the definition of $V^\ve$ and the fact that $|x-z_i|=r_i-d_\Ga(x)$, leads to the second inequality in \eqref{eq:U-V-barriers}.
\par
If $x\in\Om\setminus\ol\Om_{r_i}$, we just note that the expression of $V^\ve$ was already obtained in Lemma \ref{elliptic-control-from-above}. 
\par
The case $p=\infty$ can be treated with similar arguments. 
\end{proof}

\section{Asymptotics for $q$-means}
\label{sec:second-order}

%Before going on, we fix some notations. For any point $y\in \Ga$, denote by $\ka_1(y),\dots,\ka_{N-1}(y)$ the principal curvatures of $\Ga$ at $y$ with respect to the interior normal direction to $\Ga$. Moreover, we let $\Pi_\Ga$ be the function defined in \eqref{def-function-Pi}:
%$$
%\Pi_\Ga(y)=\prod_{j=1}^{N-1}\left[1-R\,\ka_j(y)\right] \ \mbox{ for } \ y\in\Ga.
%$$

%\subsection{Definition of $q$-means}
%\label{ssec:qmeans}

%\subsection{A geometrical lemma}
%\label{ssec:geometrical-lemma}

Before going on, we recall from \cite{IMW} some preliminary facts about $q$-means and a geometrical lemma from \cite{MS-PRSE}.

For any continuous function $u$ there exists an unique $\mu$ satisfying \eqref{eq:def-qmean} (see \cite[Theorem 2.1]{IMW}). For $1\leq q < \infty$, $\mu$ can be characterized by the equation
\begin{equation*}
\int_B | u(z) - \mu |^{q-2} \left[ u(z) - \mu \right]\, dz =0.
\end{equation*}
This equation is equivalent to 
\begin{equation}
\label{eq:characterization-qmean}
\int_B \left[ u(z) - \mu\right]_+^{q-1}\,dz = \int_B \left[\mu - u(z) \right]_+^{q-1}\,dz,
\end{equation}
where $[s]_+=\max\{0,s\}$, for $s\in\RR$.
\par
The $q$-mean is monotonically increasing with respect to $u$, in the sense that 
\begin{equation}
\label{eq:monotonicty of qmeans}
\mu^u \leq \mu^v \ \mbox{ if } \ u \leq v \ \mbox{ in } \ B. 
\end{equation}
where, $\mu^u$ and $\mu^v$ are respectively the $q$-mean of $u$ and of $v$.

\medskip

The next lemma is a version of \cite[Lemma 2.1]{MS-PRSE} slightly adapted to our notations. For the reader's convenience, we also report its proof. We recall that by $\Pi_\Ga$ we mean the function in \eqref{eq:Pi-gamma}.

\begin{lem}
\label{lem:geometric-asymptotics}
Let $x\in\Om$ and assume that, for $R>0$, there exists $y_x\in\Ga$ such that $\ol{B_R(x)}\cap(\RR^N\setminus\Om)=\{y_x\}$ and that $\ka_j(y_x)<1/R$ for $j=1,\dots,N-1$. Set $\Ga_s=\{y\in \Om:d_\Ga(y)=s\}$, for $s>0$.
\par 
Then, it holds that
$$
\lim_{s\to 0^+}s^{-\frac{N-1}{2}}\cH_{N-1}(\Ga_s\cap B_R(x))=\frac{\om_{N-1}\,(2R)^{\frac{N-1}{2}}}{(N-1)\sqrt{\Pi_\Ga(y_x)}}, 
$$
where $\cH_{N-1}$ denotes $(N-1)$-dimensional Hausdorff measure and $\om_{N-1}$ is the surface area of a unit sphere in $\RR^{N-1}$.
\end{lem}

\begin{proof}[{Proof from \cite{MS-PRSE}}]
Without loss of generality, we can suppose that we are working with a coordinate system $\{z_1,\dots,z_N\}$, such that $y_x=0$, the tangent plane to $\Ga$ at $y_x$ coincides with the plane $\{z_N=0\}$ and $x=(0,\dots,0,R)$. We can also suppose that $z_1,\dots,z_{N-1}$ are chosen such that
\begin{eqnarray}
\label{eq:distance}
d_\Ga(z)= z_N - \frac1{2}\sum_{j=1}^{N-1} \ka_j(y_x) z_j^2 + o(|z|^2)
\\
\label{eq:distance derivative}
\frac{\pa d_\Ga}{\pa z_N}(z)=1 + o(|z|).
\end{eqnarray}
Note that, with these choices, $B_R(x)$ is represented by the inequality $|z'|^2+(z_N-R)^2 < R^2$, where $z'=(z_1,\dots,z_{N-1})$.
Hence, near the origin, $\pa B_R(x)$ is represented by
\begin{equation}
\label{eq:boundary of B}
z_N = \frac1{2}|z'|^2+O(|z'|^3).
\end{equation}
\par
Combining \eqref{eq:distance} with \eqref{eq:boundary of B}, gives
\begin{equation}
\label{eq:new distance}
d_\Ga(z)= \frac1{2}\sum_{j=1}^{N-1}\left(\frac1{R} - \ka_j(y_x)\right) z_j^2 + o(|z'|^2)\ \mbox{ for } \ z\in B_R \cap \pa B_R(x).
\end{equation}
Since $\ol{B}_R(x)\cap \left(\RR^N\setminus \Om \right)=\{0\}$, for every $\ve > 0$, there exists $s_\ve>0$ such that
\begin{equation}
\label{eq:containing}
\Ga_s \cap B_R(x) \subseteq B_\ve \ \mbox{ if } \ 0<s<s_\ve.
\end{equation}
Hence, from \eqref{eq:distance derivative}, if $\ve > 0$ is sufficiently small and $0<s<s_\ve$, $\Ga_s \cap B_R(x)$ is represented by the graph of a smooth function $z_N=\psi(z')$. Differentiating $d_\Ga(z',\psi(z'))=s$ with respect to $z_j$ yields
$$
d_{z_N} \na_{z'} \psi + \na_{z'} d = 0,
$$
which together with $|\na d_\Ga|=1$ implies that
\begin{equation}
\label{eq:last}
\sqrt{1 + |\na_{z'}\psi|^2} =  1/d_{z_N}.
\end{equation}
\par
Projecting $\Ga_s\cap B_R(x)$ orthogonally on $\{z_N = 0\}$ yields a domain $A_s\subseteq \RR^{N-1}$. Let $\eta > 0$ be sufficiently small. From \eqref{eq:new distance} and \eqref{eq:containing}, there exists $\ve_0>0$ such that, for every $0<s<s_{\ve_0}$, we have
\begin{equation}
\label{eq:ellipsoids control}
E_s^+ \subseteq A_s \subseteq E_s^-
\end{equation}
where 
\begin{equation*}
\label{eq:ellipsoids}
E_s^{\pm} = \left\{ z'\in\RR^{N-1}: \frac1{2} \sum_{j=1}^{N-1} \left(\frac1{R} - \ka_j(y_x) \pm \eta\right) z_j^2 < s\right\}.
\end{equation*}
Moreover, combining \eqref{eq:distance derivative} and \eqref{eq:last} yields
\begin{equation}
\label{eq:inequality for psi}
1 \leq \sqrt{1 + |\na_{z'}\psi|^2} \leq 1 + \eta,
\end{equation}
for very $0<s<s_{\ve_0}$. Hence, it follows from  \eqref{eq:ellipsoids control} and \eqref{eq:inequality for psi} that
\begin{equation}
\label{eq:integral over ellipsoids control}
\int_{E_s^+}1 \,dz' \leq \cH_{N-1}\left(\Ga_s\cap B_R(x)\right) \leq \int_{E_s^-}(1+\eta)\,dz',
\end{equation}
for every $0<s<s_{\ve_0}$, since
$$
\cH_{N-1}\left(\Ga_s\cap B_R(x) \right) = \int_{A_s} \sqrt{1 + |\na_{z'}\psi|^2}\,dz'.
$$
Hence, from \eqref{eq:integral over ellipsoids control} we see that
\begin{multline*}
\frac{\om_{N-1}\,2^{\frac{N-1}{2}}}{(N-1)}\left\{\prod_{j=1}^{N-1}\left[\frac1{R} - \ka_j(y_x) +\eta \right] \right\}^{-1/2}
\leq
\\
s^{-\frac{N-1}{2}} \cH_{N-1}\left(\Ga_s\cap B_R(x) \right)
\leq
\\
\frac{\om_{N-1}\,2^{\frac{N-1}{2}}}{(N-1)}\left\{\prod_{j=1}^{N-1}\left[\frac1{R} - \ka_j(y_x) -\eta \right] \right\}^{-1/2}
\end{multline*}
for every $0<s<s_{\ve_0}$. Since $\eta > 0$ is arbitrarily small, we conclude the proof.
\end{proof}

\subsection{Short-time asymptotics for $q$-means}
\label{ssec:short-time q-means}

%Formula \eqref{asymptotics to the curvatures} can also be seen as an asymptotic formula for the {\it mean value} of $u$ on the ball $B_R(x)$. In fact, the following scale invariant formula follows:
%$$
%\lim_{t\to 0^+} \left(\frac{R^2}{t}\right)^\frac{N+1}{4}\dashint_{B_R(x)}u(y,t)\,dy=\frac{c_N}{(p')^\frac{N+1}{4}\sqrt{\Pi_\Ga(y_x)}},
%$$
%with
%$$
%c_N=\frac{2^\frac{N+1}{2}}{\sqrt{\pi}}\,\frac{N}{N+1}\,\frac{\Ga\left(\frac{N}{2}\right)}{\Ga\left(\frac{N+1}{4}\right)}.
%$$
%\par
%Other statistical quantities that seem to be particularly appropriate and interesting in a game-theoretic  context are the so-called {\it $q$-means}. We shall consider for $1<q\le\infty$ the $q$-mean $\mu_q^u(x,t)$ of $u(\cdot,t)$ on $B_R(x)$, as defined in \eqref{p-mean}; this coincides with the mean value when $q=2$.

\begin{lem}[{Asymptotics for the $q$-mean of a barrier, \cite[Lemma 3.4]{BM-JMPA}}]
\label{lem:JMPA-barriers-qmean}
Set $1<q<\infty$, let $x\in\Om$, and assume that, for $R>0$, there exists a point $y_x\in\Ga$ such that $\ol{B_R(x)}\cap(\RR^N\setminus\Om)=\{y_x\}$ and $\ka_j(y_x)<1/R$ for $j=1,\dots,N-1$.
\par
Let $\xi, \eta:(0,\infty)\to(0,\infty)$ be two functions of time such that 
$\xi(t)$ is positive in $(0,\infty)$, and
$$
\lim_{t\to 0^+} \xi(t)=\lim_{t\to 0^+} \eta(t)=0.
$$ 
\par
For a non-negative, decreasing and continuous function $f$ on $\RR$ such that
$$
\int_0^\infty f(\si)^{q-1} \si^\frac{N-1}{2} d\si<\infty,
$$
set
$$
w(z,t)=f\left(\frac{d_\Ga(z)}{\xi(t)}+\eta(t)\right) \ \mbox{ for } \ (z,t)\in\ol{\Om}\times(0,\infty).
$$
\par
If $\mu_q^w(x,t)$ is the $q$-mean of $w\left(\cdot, t\right)$ on $B_R(x)$, 
then the following formula holds:
\begin{equation}
\label{asymptotics-p-mean-w}
\lim_{t\to 0^+} \left(\frac{R}{\xi(t)}\right)^\frac{N+1}{2(q-1)}\!\!\!\!\mu_q^w(x,t)=
\left\{\frac{2^{-\frac{N+1}{2}}N!\,\int_0^\infty f(\si)^{q-1} \si^\frac{N-1}{2} d\si}{\Ga\left( \frac{N+1}{2}\right)^2\sqrt{\Pi_\Ga(y_x)}}\,\right\}^\frac1{q-1} .
\end{equation}
\end{lem}

\begin{proof}
We know from \eqref{eq:characterization-qmean} that $\mu(t)=\mu_q^w(x,t)$ is the unique root of the following equation
\begin{equation}
\label{i-characterization-mu}
\int_{B_R(x)}[w(z,t)-\mu(t)]_+^{q-1} dz=\int_{B_R(x)}[\mu(t)-w(z,t)]_+^{q-1} dz.
\end{equation}
\par
Firstly, we compute the short-time behavior of the left-hand side of \eqref{i-characterization-mu}. Let $\Ga_s=\{z\in B_R(x):d_\Ga(z)=s\}$. By the co-area formula, we get that
\begin{equation*}
  \int_{B_R(x)}[w(z,t)-\mu(t)]_+^{q-1} dz=
\int_{0}^{2R}\left[f\left(\frac{s}{\xi(t)}+\eta(t)\right)-\mu(t)\right]^{q-1}_+\cH_{N-1}\left(\Ga_s\right)\,ds.
\end{equation*}
\par
By the change of variable $s=\xi(t)\left[\si-\eta(t)\right]$, we obtain that
\begin{equation*}
  \int_{B_R(x)}[w(z,t)-\mu(t)]_+^{q-1} dz=
\xi(t)\int_{\eta(t)}^{\be(t)}\left[f\left(\si\right)-\mu(t)\right]^{q-1}_+\cH_{N-1}\left(\Ga_{\xi(t)\left[\si-\eta(t)\right]}\right)\,d\si,
\end{equation*}
where we set $\be(t)=\frac{2R}{\xi(t)}+\eta(t)$. 
\par 
Hence,
\begin{multline*}
  \xi(t)^{-\frac{N-1}{2}}\int_{B_R(x)}[w(z,t)-\mu(t)]_+^{q-1} dz=
  \\
\int_{\eta(t)}^{\be(t)}\frac{\cH_{N-1}\left(\Ga_{\xi(t)\left[\si-\eta(t)\right]}\right)}{\left\{\xi(t)\left[\si-\eta(t)\right]\right\}^{\frac{N-1}{2}}}\left[\si-\eta(t)\right]^{\frac{N-1}{2}}\left\{f(\si)-\mu(t)\right\}^{q-1}\,d\si.
\end{multline*}
Now, as $t\to 0^+$ we have that $\eta(t), \xi(t), \mu(t)\to 0$, $\be(t)\to \infty$ and that $\xi(t)\left[\si-\eta(t)\right]\to 0$ for almost every $\si\ge 0$. Thus, we can infer that
\begin{multline}
\label{eq:left-hand side}
  \lim_{t\to 0^+}
\xi(t)\int_{B_R(x)}[w(z,t)-\mu(t)]_+^{q-1} dz=
\frac{\om_{N-1}(2R)^{\frac{N-1}{2}}}{(N-1)\sqrt{\Pi_\Ga(y_x)}}\int_{0}^{\infty}f(\si)^{q-1}\si^{\frac{N-1}{2}}\,d\si,
\end{multline}
%\begin{multline}
%\label{eq:left-hand side}
%  \lim_{t\to 0^+}
%\xi(t)\int_{B_R(x)}[w(z,t)-\mu(t)]_+^{q-1} dz=
%\\
%\frac{\om_{N-1}(2R)^{\frac{N-1}{2}}}{(N-1)\sqrt{\Pi_\Ga(y_x)}}\int_{0}^{\infty}f(\si)^{q-1}\si^{\frac{N-1}{2}}\,d\si,
%\end{multline}
by Lemma \ref{lem:geometric-asymptotics} and an application of the dominated convergence theorem, as an inspection of the integrand function reveals.
\par 
Secondly, we treat the short-time behavior of the right-hand side of \eqref{i-characterization-mu}. By again performing the co-area formula and after some manipulations, we have that
\begin{multline}
\label{eq:right-hand side}
  \int_{B_R(x)}\left [\mu(t)-w(z,t)\right]^{q-1}_+\,dz=
\\
\mu(t)^{q-1}\int_{0}^{2R}\left[1-f\left(\frac{s}{\xi(t)}+\eta(t)\right)\Big/\mu(t)\right]^{q-1}_+\cH_{N-1}\left(\Ga_s\right)\,ds
\end{multline}
which, on one hand,  leads to 
\begin{equation*}
    \int_{B_R(x)}\left [\mu(t)-w(z,t)\right]^{q-1}_+\,dz\le \mu(t)^{q-1}|B_R(x)|.
\end{equation*}
Notice in particular that, by using both \eqref{characterization-mu} and \eqref{eq:left-hand side}, the last inequality informs us that
$$
\mu(t)\ge c\,\xi(t)^{\frac{N+1}{2(q-1)}},
$$
for some positive constant $c$. Hence, after setting $\be(s,t)=\frac{s}{\xi(t)}+\eta(t)$, the assumptions on $f$ give the following chain of inequalities:
\begin{multline*}
  \int_{\be(s,t)/2}^{\infty}f(\si)^{q-1}\si^{\frac{N-1}{2}}\,d\si\ge\int_{\be(s,t)/2}^{\be(s,t)}f(\si)^{q-1}\si^{\frac{N-1}{2}}\,d\si\ge\\
\frac{2\left(1-2^{-\frac{N+1}{2}}\right)}{N+1}\frac{f\left(\be(s,t)\right)^{q-1}}{\xi(t)^{\frac{N+1}{2}}}\left[s+\eta(t)\,\xi(t)\right]^{\frac{N+1}{2}}\ge\\
\frac{2\left(1-2^{-\frac{N+1}{2}}\right)}{c(N+1)}\left[f\left(\frac{s}{\xi(t)}+\eta(t)\right)\Big/\mu(t)\right]^{q-1} \left[s+\eta(t)\,\xi(t)\right]^{\frac{N+1}{2}}.
\end{multline*}
Since, for almost every $s\ge 0$, the first term of the chain vanishes as $t\to 0^+$, we have that
\begin{equation*}
  \label{eq:pointwise-right-hand}
  \lim_{t\to 0^+}
\frac{f\left(\frac{s}{\xi(t)}+\eta(t)\right)}{\mu(t)}=0,
\end{equation*}
for almost every $s\ge 0$. Thus, \eqref{eq:right-hand side} gives at once that
\begin{equation}
\label{eq:right-hand side formula}
\lim_{t\to 0^+}   
\mu(t)^{1-q}\int_{B_R(x)}\left [\mu(t)-w(z,t)\right]^{q-1}_+\,dz=|B_R(x)|.
\end{equation}
\par 
Finally, \eqref{characterization-mu}, \eqref{eq:left-hand side} and \eqref{eq:right-hand side formula} tell us that
$$
\mu(t)^{q-1}=\xi(t)^\frac{N+1}{2} \frac{\om_{N-1}\,(2R)^{\frac{N-1}{2}}}{(N-1)\sqrt{\Pi_\Ga(y_x)}}\,\frac{\int_0^\infty f(\si)^{q-1}\,
\si^\frac{N-1}{2} d\si+o(1)}{|B_R(x)|+o(1)},
$$
that gives \eqref{asymptotics-p-mean-w}, after straightforward calculations involving Euler's gamma function.
\end{proof}

\begin{rem}
{\rm
If $q=\infty$, we know that
\begin{equation*}
\mu_\infty^w(x,t)=\frac12\,\left\{ \min_{\ol{B_R(x)}} w(\cdot,t)+\max_{\ol{B_R(x)}} w(\cdot,t)\right\}=
\frac12\,\left[f\left(\frac{\ol{d}}{\xi(t)}+\eta(t)\right)+f(\eta(t))\right],
\end{equation*}
where $\ol{d}$ is positive, being the maximum of $d_\Ga$ on $\ol{B_R(x)}$. Hence, it is easy to compute:
$$
\lim_{t\to 0^+} \mu_\infty^w(x,t)=\frac12\,f(0).
$$
Thus, formula \eqref{asymptotics-p-mean-w} does not extend continuously to the case $q=\infty$.
}
\end{rem}

\begin{thm}[{Short-time asymptotics for $q$-means, \cite[Theorem 3.5]{BM-JMPA}}]
\label{th:JMPA-asymptotics-qmean}
Let $x\in\Om$, and assume that, for $R>0$, there exists a point $y_x\in\Ga$ such that $\ol{B_R(x)}\cap(\RR^N\setminus\Om)=\{y_x\}$ and $\ka_j(y_x)<1/R$ for $j=1,\dots,N-1$.
\par
Set $1<p\leq \infty$ and suppose that $u$ is the bounded (viscosity) solution of \eqref{G-heat}-\eqref{boundary} and, for $1<q\le\infty$, $\mu_q(x,t)$ is the $q$-mean of $u\left(\cdot, t\right)$ on $B_R(x)$.
\par
Then, if $1<q<\infty$, the following formulas hold:
%\begin{equation}
%\label{asymptotics-p-mean-u}
%\lim_{t\to 0^+} \left(\frac{R^2}{t}\right)^\frac{N+1}{4(q-1)}\!\!\!\!\mu_q(x,t)=
%\\
%\left\{\frac{N!\,\int_0^\infty \Erfc(\si)^{q-1} \si^\frac{N-1}{2} d\si}{\,\Ga \left( \frac{N+1}{2} \right)^2}\right\}^\frac1{q-1} \left\{ p'^{\frac{N+1}{2}}\,\Pi_\Gamma(y_x)\right\}^{-\frac{1}{2(q-1)}}, 
%\end{equation}
\begin{multline}
\label{asymptotics-p-mean-u}
\lim_{t\to 0^+} \left(\frac{R^2}{t}\right)^\frac{N+1}{4(q-1)}\!\!\!\!\mu_q(x,t)=
\\
\left\{\frac{N!\,\int_0^\infty \Erfc(\si)^{q-1} \si^\frac{N-1}{2} d\si}{\,\Ga \left( \frac{N+1}{2} \right)^2}\right\}^\frac1{q-1} \left\{ p'^{\frac{N+1}{2}}\,\Pi_\Gamma(y_x)\right\}^{-\frac{1}{2(q-1)}}, 
\end{multline}
and
$$
\lim_{t\to 0^+} \mu_\infty(x,t)=\frac12.
$$
\end{thm}

\begin{proof}
By using \eqref{eq:parabolic-barriers} and \eqref{eq:monotonicty of qmeans}, the limit in \eqref{asymptotics-p-mean-u} will result from Lemma \ref{lem:JMPA-barriers-qmean}, where we choose:
$$
w(x,t)=\Erfc\left(\sqrt{\frac{p'}{4t}}\,d_\Ga(y)\pm\eta(t)\right),
$$
that is we choose $\xi(t)=\sqrt{4t/p'}$ and $\eta(t)$ is given by \eqref{eq:eta-K}, with $K=\ol{B}_R(x)$.
Thus, \eqref{asymptotics-p-mean-u} will follow at once from \eqref{asymptotics-p-mean-w}, where
$f(\si)=\Erfc(\si)$.
\par
By the same argument, we also get the case $q=\infty$, since $f(0)=1$.
 \end{proof}

\begin{rem}{\rm
Notice that
$$
\left\{\int_0^\infty \Erfc(\si)^{q-1} \si^\frac{N-1}{2} d\si\right\}^\frac1{q-1}
$$
can be seen as the $(q-1)$-norm of $\Erfc$ in $(0,\infty)$ with respect to the weighed measure $\si^\frac{N-1}{2} d\si$.
}
\end{rem}

\subsection{Asymptotics for $q$-means in the elliptic case}
\label{ssec:elliptic asymptotics q-means}

%From now on, in order to use the function $\Pi_\Ga$ defined in \eqref{eq:Pi-gamma},  we assume that $\Om$ is a domain of class $C^2$ (not necessarily bounded). 
%\par
%First, we recall from \cite[Lemma 2.1]{MS-PRSE} the following geometrical lemma.
%
%\begin{lem}
%\label{lem:geometrical-asymptotics}
%Let $x\in\Om$ and assume that, for $R>0$, there exists $y_x\in\Ga$ such that $\ol{B_R(x)}\cap(\RR^N\setminus\Om)=\{y_x\}$ and that $\ka_j(y_x)<1/R$ for $j=1,\dots,N-1$.
%\par 
%Then, it holds that
%$$
%\lim_{s\to 0^+}\frac{\cH_{N-1}(\Ga_s\cap B_R(x))}{s^{\frac{N-1}{2}}}=\frac{\om_{N-1}\,(2R)^{\frac{N-1}{2}}}{(N-1)}\left[\Pi_\Ga(y_x)\right]^{-\frac1{2}}, 
%$$
%where $\cH_{N-1}$ denotes $(N-1)$-dimensional Hausdorff measure and $\om_{N-1}$ is the surface area of a unit sphere in $\RR^{N-1}$.
%\end{lem}

The next lemma gives the asymptotic formula for $\ve\to 0^+$ for the $q$-mean on $B_R(x)$ of a quite general class of functions, which includes both the barriers $U^\ve$ and $V^\ve$ of Lemma \ref{barriers-for-c2-domain}.

\begin{lem}[{\cite[Lemma 3.3]{BM-AA}}] 
\label{lem:AA-barriers-qmean}
Set $1<q<\infty$. Let $x\in\Om$ and assume that, for $R>0$, there exists $y_x\in\Ga$ such that $\ol{B_R(x)}\cap(\RR^N\setminus\Om)=\{y_x\}$ and that $\ka_j(y_x)<1/R$ for $j=1,\dots,N-1$.
\par
Let $\{\xi_n\}_{n\in\NN}$ and $\{f_n\}_{n\in\NN}$ be sequences such that
\begin{enumerate}[(i)]
\item
$\xi_n>0$ and $\xi_n\to 0$ as $n\to\infty$;
\item 
$f_n:[0,\infty)\to[0,\infty)$ are decreasing functions;
\item
$f_n$ converges to a function $f$ almost everywhere as $n\to\infty$;
\item
it holds that
\begin{equation*}
  \label{integral-condition}
\lim_{n\to\infty}
\int_{0}^{\infty}f_n(\si)^{q-1}\,\si^{\frac{N-1}{2}}\,d\si=\\
\int_{0}^{\infty}f(\si)^{q-1}\,\si^{\frac{N-1}{2}}\,d\si,
\end{equation*}
and the last integral converges.
\end{enumerate}
For some $1<q<\infty$, let $\mu_{q,n}(x)$ be the $q$-mean of $f_n(d_\Ga/\xi_n)$ on $B_R(x)$.
\par
Then we have:
\begin{equation}
\label{eq:q-mean-barrier}
\lim_{n\to\infty}
  \left(\frac{R}{\xi_n}\right)^{\frac{N+1}{2(q-1)}}\mu_{q,n}(x)=\\
\left\{\frac{2^{-\frac{N+1}{2}}N! \int_{0}^{\infty}f(\si)^{q-1}\si^{\frac{N-1}{2}}\,d\si}{\Ga\left(\frac{N+1}{2}\right)^2\, \sqrt{\Pi_\Ga(y_x) }}\right\}^{\frac1{q-1}}.
\end{equation}
\end{lem}

\begin{proof}
  From \eqref{eq:characterization-qmean}, we know that $\mu_n=\mu_{q,n}(x)$ is the only root of the equation
  \begin{equation}
\label{characterization-mu}
    \int_{B_R(x)}\left[f_n(d_\Ga/\xi_n)-\mu_n\right]_+^{q-1}\,dz=
\int_{B_R(x)}\left[\mu_n-f_n(d_\Ga/\xi_n)\right]_+^{q-1}\,dz,
  \end{equation}
where we mean $[t]_+=\max(0, t)$.
\par
Thus, if we set
$$
\Ga_\si=\{y\in B_R:\, d_\Ga(y)=\si\},
$$
by the co-area formula 
we get that
\begin{equation*}
  \int_{B_R(x)}\left[f_n(d_\Ga/\xi_n)-\mu_n\right]_+^{q-1}\,dz=\\
\int_{0}^{2R}\left[f_n(\si/\xi_n)-\mu_n\right]_+^{q-1}\cH_{N-1}\left(\Ga_\si\right)\,d\si,
\end{equation*}
that, after the change of variable $\si=\xi_n \tau$ and easy manipulations, leads to the formula:
\begin{equation*}
  \int_{B_R(x)}\left[f_n(d_\Ga/\xi_n)-\mu_n\right]_+^{q-1}\,dz=
\xi_n^{\frac{N+1}{2}}
\int_{0}^{2R/\xi_n}\left[f_n(\tau)-\mu_n\right]_+^{q-1}\tau^{\frac{N-1}{2}}\left[\frac{\cH_{N-1}\left(\Ga_{\xi_n\tau}\right)}{\left(\xi_n\tau\right)^{\frac{N-1}{2}}}\right]\,d\tau.
\end{equation*}
\par 
Therefore, since $\mu_n\to 0$ as $n\to\infty$, an inspection of the integrand at the right-hand side, assumptions (i)-(iv), and Lemma \ref{lem:geometric-asymptotics} make it clear that we can apply the generalized dominated convergence theorem (see \cite{LL}) to infer that
\begin{multline}
\label{eq:formula-A-plus}
\lim_{n\to\infty}
\xi_n^{-\frac{N+1}{2}}\int_{B_R(x)}\left[f_n(d_\Ga/\xi_n)-\mu_n\right]_+^{q-1}\,dz=\\
\frac{(2R)^{\frac{N-1}{2}}\om_{N-1}}{(N-1)\sqrt{\Pi_\Ga(y_x)}}\int_{0}^{\infty}f(\si)^{q-1}\si^{\frac{N-1}{2}}\,d\si.
\end{multline}
%\begin{equation}
%\label{eq:left-hand side}
%  \lim_{t\to 0^+}
%\xi(t)\int_{B_R(x)}[w(z,t)-\mu(t)]_+^{q-1} dz=
%\frac{\om_{N-1}(2R)^{\frac{N-1}{2}}}{(N-1)\sqrt{\Pi_\Ga(y_x)}}\int_{0}^{\infty}f(\si)^{q-1}\si^{\frac{N-1}{2}}\,d\si,
%\end{equation}
\par 
Next, by employing again the co-area formula, the right-hand side of \eqref{characterization-mu} can be re-arranged as
\begin{equation*}
\label{eq:A-minus-inequality}
  \int_{B_R(x)} [\mu_n-f_n(d_\Ga/\xi_n)]_+^{q-1}\,dz=
\mu_n^{q-1}\int_{0}^{2R} \left[1-\frac{f_n(\si/\xi_n)}{\mu_n}\right]_+^{q-1}\cH_{N-1}\left(\Ga_\si\right)\,d\si,
\end{equation*}
that leads to the formula
\begin{equation}
  \label{eq:formula-A-minus}
  \lim_{\ve\to 0^+}
\mu_n^{1-q}\int_{B_R(x)} [\mu_n-f_n(d_\Ga/\xi_n)]_+^{q-1}\,dz=
|B_R|,
\end{equation}
by dominated convergence theorem, if we can prove that
\begin{equation}
\label{pointwise}
\frac{f_n(\si/\xi_n)}{\mu_n}\to 0 \ \mbox{ as } \ n\to \infty,
\end{equation}
for almost every $\si\ge 0$.
Then, after straightforward computations, \eqref{eq:q-mean-barrier} will follow 
by putting together \eqref{characterization-mu}, \eqref{eq:formula-A-plus} and \eqref{eq:formula-A-minus}.
\par
We now complete the proof by proving that \eqref{pointwise} holds. From \eqref{characterization-mu}, \eqref{eq:formula-A-plus}, and the fact that
$$
\int_{B_R(x)}\left[\mu_n-f_n(d_\Ga/\xi_n)\right]_+^{q-1}\,dz\le \mu_n^{q-1} |B_R|,
$$
we have that there is a positive constant $c$ such that
$$
\mu_n^{1-q}\le c\,\xi_n^{-\frac{N+1}{2}}.
$$
Also, for every $\tau>0$ we have that
\begin{multline*}
\int_{\tau/2\xi_n}^\infty f_n(\si)^{q-1}\,\si^\frac{N-1}{2} d\si\ge
\int_{\tau/2\xi_n}^{\tau/\xi_n} f_n(\si)^{q-1}\,\si^\frac{N-1}{2} d\si\ge \\
\frac{2(1-2^{-\frac{N+1}{2}})}{N+1}\,f_n(\tau/\xi_n)^{q-1}\left(\frac{\tau}{\xi_n}\right)^\frac{N+1}{2}\ge \\
\frac{2(1-2^{-\frac{N+1}{2}})}{c\,(N+1)}\,\tau^\frac{N+1}{2}\left\{\frac{f_n(\tau/\xi_n)}{\mu_n}\right\}^{q-1}.
\end{multline*}
Thus, \eqref{pointwise} follows, since the first term of this chain of inequalities converges to zero as $n\to\infty$, under our assumptions on $f_n$ and $\xi_n$, in virtue of the generalized dominated convergence theorem.
\end{proof}

\begin{rem}
{\rm
The case $q=\infty$ is simpler. From \cite{IMW} and then the monotonicity of $f_n$ we obtain that:
\begin{equation*}
\mu_{\infty,n}(x)=\frac1{2}\left\{\min_{B_R(x)}f_n\left(d_\Ga/\xi_n\right)+\max_{B_R(x)}f_n\left(d_\Ga/\xi_n\right)\right\}= 
\frac1{2}\left\{f_n\left(2R/\xi_n\right)+f_n(0)\right\}.
\end{equation*}
Thus, if we replace the assumptions (iii) and (iv) by $f_n(0)\to f(0)$ as $n\to\infty$, we conclude that
$\mu_{\infty,n}(x)\to f(0)/2$, since $f_n\left(2R/\xi_n\right)\to 0$ as $n\to\infty$. 
}
\end{rem}

\begin{thm}[{\cite[Theorem 3.5]{BM-AA}}]
  \label{th:AA-qmean}
Set $1<p\le \infty$. Let $x\in\Om$ be such that $B_R(x)\subset\Om$ and $\ol{B_R(x)}\cap(\RR^N\setminus\Om)=\{y_x\}$; suppose that $k_j(y_x)<\frac1{R}$, for every $j=1,\dots,N-1$.
\par 
Let $u^\ve$ be the bounded (viscosity) solution of \eqref{G-elliptic}-\eqref{elliptic-boundary} and, for $1<q\le\infty$, let $\mu_{q,\ve}(x)$ be the $q$-mean of $u^\ve$ on $B_R(x)$.
\par 
Then, if $1<q<\infty,$ we have that 
\begin{equation}
  \label{q-mean}
\lim_{\ve\to 0^+}
  \left(\frac{\ve}{R}\right)^{-\frac{N+1}{2(q-1)}}\mu_{q,\ve}(x)=
  \left\{\frac{2^{-\frac{N+1}{2}}N!}{(q-1)^{\frac{N+1}{2}}\Ga\left(\frac{N+1}{2}\right)}\right\}^{\frac1{q-1}}\left\{p'^{\frac{N+1}{2}}
\Pi_\Ga(y_x)\right\}^{-\frac{1}{2(q-1)}}.
\end{equation}
\par 
If $q=\infty$, we simply have that 
$$ 
\lim_{\ve\to 0^+}
\mu_{\infty,\ve}(x) =\frac1{2}.
$$
\end{thm}

\begin{proof}
We have that $\mu_{q,\ve}^{U^\ve}(x)\le \mu_{q,\ve}(x)\le \mu_{q,\ve}^{V^\ve}(x)$ by the monotonicity of the $q$-means, where with $\mu_{q,\ve}^{U_\ve}$ and $\mu_{q,\ve}^{V^\ve}$ we denote the $q$-mean of $U^\ve(d/\ve)$ and $V^\ve(d/\ve)$ on $B_R(x)$.
Hence, in order to prove \eqref{q-mean}, we only need to apply Lemma \ref{lem:AA-barriers-qmean} to  $f_n=U^{\ve_n}$ and $f_n=V^{\ve'_n}$, where the vanishing sequences $\ve_n$ and $\ve'_n$ are chosen so that the $\liminf$ and $\limsup$ of $\left(\ve/R\right)^{-\frac{N+1}{2(q-1)}}\mu_{q,\ve}(x)$
as $\ve\to 0$ are attained along them, respectively.
\par
By an inspection, it is not difficult to check that $f_n=U^{\ve_n}$ and $f_n=V^{\ve'_n}$, with $\xi_\ve=\ve/\sqrt{p'}$ and $f(\si)=e^{-\si}$, satisfy the relevant assumptions of Lemma \ref{lem:AA-barriers-qmean}, by applying, in particular, Lemma \ref{lem:asymptotics for modified bessel} for (iii) and the dominated convergence theorem for (iv).
\end{proof} 

%\afterpage{\blankpage}

%\afterpage{
%\thispagestyle{empty}
%}

\chapter{Geometric and symmetry results}
\label{ch:applications}

The goal of this chapter is to collect some geometric and symmetry results for solutions \eqref{G-heat}-\eqref{boundary} or \eqref{G-elliptic}-\eqref{elliptic-boundary}, in the spirit of those given by Magnanini and Sakaguchi in \cite{ CMS-JEMS, MS-AM, MS-PRSE, MS-IUMJ, MS-JDE2, MS-MMAS}. We obtain characterizations of balls, spheres and hyperplanes as applications of Varadhan-type formulas of Chapter \ref{ch:asymptotics i} and of formulas for $q$-means of Chapter \ref{ch:asymptotics ii}. 
\par
We introduce the problems that we consider. We say that an $(N-1)$-dimensional surface $\Si$ is a {\it time-invariant level surface} for the solution $u$ of \eqref{G-heat}-\eqref{boundary}, if there exists a function $a_{\Si}:(0,\infty)\to (0,\infty)$ such that
\begin{equation}
\label{time-invariant surface}
u(x,t)=a_\Si(t) \ \mbox{ for any } \ (x,t)\in\Si\times(0,\infty).
\end{equation}
%\par
%In the case of the solution of the heat equation, where a time-invariant level surface is commonly called {\it stationary isothermic surface}, we can refer a quite large list of papers, for example \cite{A1, A2, Sa-JAM, MS-AM}. In \cite{A2} it has been shown that, by using Serrin's symmetry result on overdetermined value problems, under the assumption that $\Om$ is bounded and that every point of $\Ga$ is regular with respect to the Laplacian, then if all the isothermic surfaces of $u$ are {\it invariant with time} then $\Om$ must be a ball. In \cite{Sa-JAM}, a different proof has been given with by using a proper classification theorem for hypersurfaces in Euclidean space, due to Levi-Civita and Segre. In \cite{MS-AM}, the authors have shown that, under the assumption that $\Om$ is bounded and satisfies the exterior sphere condition and that there exists a sub-domain $D$ satisfies the interior cone condition, if \eqref{time-invariant surface} holds true, for $\Si =\pa D$, then $\Om$ must be a ball. In Subsection \ref{sec:symmetry for bounded domains}, we give a version of this result for a generic $p\in(1,\infty)$ (see Corollary \ref{cor:TILsym2}). 
\par
In the case of the heat equation, a time-invariant level surface is commonly called {\it stationary isothermic surface}.  We list two results concerning stationary isothermic surfaces, which are relevant for the analysis carried out in this chapter. In \cite[Theorem 1.1]{MS-AM}, for a bounded domain $\Om$ that satisfies  the exterior sphere condition, if \eqref{time-invariant surface} holds for $\Si=\pa D$, where $D$ satisfies the interior cone condition and  $\ol{D}\subset\Om$, then $\Om$ must be a ball. See also \cite{CMS-JEMS} and \cite{MS-MMAS}. In \cite{MS-IUMJ}, the case of domains with non-compact boundary is considered. In fact in \cite[Theorem 3.4]{MS-IUMJ}, the authors have shown that if $\Ga$ is the graph of function defined on the whole $\RR^{N-1}$, satisfying certain sufficient assumptions, and \eqref{time-invariant surface} holds, then $\Ga$ must be a hyperplane. See also \cite{MS-AIHP, MS-JDE, Sa-INDAM}.

% In \cite{MS-AIHP} it has been considered the case of exterior domains. An application of \cite[Theorem 1.2]{MS-AIHP} gives that, if an exterior domain $\Om$ of class $C^2$ admits a stationary isothermic surface of class $C^2$, then $\Ga$ must be a sphere. In \cite{MS-IUMJ}, the case of domains with non-compact boundary is considered. In fact in \cite[Theorem 3.4]{MS-IUMJ}, the authors have shown that if $\Ga$ is the graph of function defined on the whole $\RR^{N-1}$, satisfying certain sufficient assumptions, and \eqref{time-invariant surface} holds, then $\Ga$ must be a hyperplane.

In Section \ref{sec:serrin}, by employing the method of moving planes (see \cite{Ser}) as in \cite{CMS-JEMS} and \cite{MS-MMAS}, we give a proper version of \cite[Theorem 1.1]{MS-AM}  in the case $p\in(1,\infty)$. See Corollary \ref{cor:TILsym2}. This corollary is actually a consequence of a quite more general theorem in which one obtains the spherical symmetry under the weaker condition that there exist $\ol{t}>0$ and $R>0$ such that
\begin{equation*}
\label{eq:symcondition1}
x\mapsto u(x,\ol{t})\ \mbox{ is constant on } \ \Ga_R,
\end{equation*}
where $\Ga_R = \{ x\in \Om:d_\Ga(x)=R\}$.
See Theorem \ref{th:TILsym1}.

%In \cite[Theorem 1.2]{MS-AIHP}, it has been considered the case of exterior domains in which there exists a surface satisfying \eqref{time-invariant surface}. In \cite[Theorem 3.4]{MS-IUMJ}, it has been considered the obviously different case of unbounded boundary domains. In this case, if $\Om$ satisfies a suitable structure condition, the authors have shown that a surface satisfying \eqref{time-invariant surface} implies that $\Ga$ must be a hyperplane. 

In the case \eqref{G-elliptic}-\eqref{elliptic-boundary} we consider surfaces $\Si$ that are level sets of $u^\ve$, for any $\ve > 0$, i.e. $\Si$ satisfies the requirement:
\begin{equation}
\label{epsilon-invariant surface}
u^\ve  \ \mbox{ is constant on } \ \Si, \ \mbox{ for any } \ve > 0.
\end{equation}
We point out that, in the linear case, from \eqref{eq:intro_modified laplace}
%$$
%u^\ve(x)=\ve^{-2}\int_{0}^{\infty}u(x,t)e^{-t/\ve^2}\,dt,
%$$
 it follows that $\Si$ is a stationary isothermic surface if and only if \eqref{epsilon-invariant surface} holds true.
 \par
 In Section \ref{sec:serrin}, we present results which are analogous to those obtained for $p\ne 2$ in the parabolic case, by just adapting the proofs in the elliptic context. 

Section \ref{sec:sliding} contains our result for the case of non-compact boundaries. There, we generalize  the result in \cite[Theorem 3.4]{MS-IUMJ} to a generic $p\in(1,\infty)$ and to the elliptic case.
\par
Section \ref{sec:symmetry-qmeans} contains another type of symmetry result for the solution of \eqref{G-heat}-\eqref{boundary}. It concerns the following condition for the q-mean $\mu_q(x,t)$ of a ball $B_R(x)$ such that $R=d_\Ga(x)$. Let $\Om$ be a domain with bounded and connected boundary in which there exists a parallel surface $\Ga_R$, such that
\begin{equation}
\label{qmean time-invariant}
x\mapsto \mu_q(x,t) \ \mbox{ is constant on } \ \Ga_R,
\end{equation}
for any fixed $t>0$. In the spirit of \cite[Theorem 1.2]{MS-PRSE}, in Theorem \ref{th:radiality-bounded domain} we show that if \eqref{qmean time-invariant} holds, then $\Ga$ must be a sphere. Theorem \ref{th:radiality-bounded domain-elliptic} takes care of the elliptic counterpart of \eqref{qmean time-invariant}.

\section{Parallel surfaces}
\label{sec:parallelism}

An interesting geometric property that invariant surfaces enjoy is that they are parallel to $\Ga$, as the following results show for both the parabolic and elliptic case.

\begin{thm}[{\cite[Theorem 3.6]{BM-JMPA}}]
\label{th:parallelism-invariant level set}
Let $\Om$ be a domain in $\RR^N$ satisfying $\Ga=\pa\left(\RR^N\setminus\ol\Om\right)$ and suppose that, for $1<p\le\infty$, $u$ is the solution of \eqref{G-heat}-\eqref{boundary}. 
\par 
If $\Si\subset\Om$ is a time-invariant level surface for $u$, then there exists $R>0$ such that 
\begin{equation}
\label{constant distance}
d_\Ga(x)=R\ \mbox{ for every }\ x\in\Si.
\end{equation}
\end{thm}

\begin{proof}
Let $R=\dist\left(\Si,\Ga\right)$ and let $x_0$ be a point in $\Si$ such that $d_\Ga(x_0) = R$. If $y\in \Ga$, we have that $u(x_0,t)=u(y,t)$ and hence $4t\,\log u(x_0,t)=4t\,\log u(y,t)$ for every $t>0$. By Theorem \ref{th:parabolic-pointwise},
we infer that $d_\Ga(x_0)=d_\Ga(y)$ and hence we obtain our claim.
\end{proof}

\begin{thm}
\label{th:parallelilsm-invariant level set-elliptic}
Let $\Om$ be a domain in $\RR^N$ satisfying $\Ga=\pa\left(\RR^N\setminus\ol\Om\right)$ and suppose that, for $1<p\le\infty$, $u^\ve$ is the solution of \eqref{G-elliptic}-\eqref{elliptic-boundary}. 
\par 
If $\Si\subset\Om$ is an level surface of $u^\ve$, for every $\ve > 0$, then there exists $R>0$ such that \eqref{constant distance} holds true.
%\begin{equation}}
%d_\Ga(x)=R\ \mbox{ for every }\ x\in\Si.
%\end{equation}}
\end{thm}

\begin{proof}
To conclude, it is sufficient to properly modify the proof of Theorem \ref{th:parallelism-invariant level set}, having in mind Theorem \ref{th:elliptic-pointwise}.
\end{proof}

\section{Spherical symmetry for invariant level surfaces}
\label{sec:serrin}

We recall that bounded viscosity solutions of \eqref{G-heat} and \eqref{G-elliptic} are of class $C^{1,\be}_{loc}$, for some $ 0 < \be < 1$. See \cite[Theorem 2.1]{APR} and \cite[Theorem 2.1]{AP-2018}. 

%\subsection{The case of bounded domains}
%\label{ssec:symmetry for bounded domains}

In Theorems \ref{th:TILsym1} and \ref{th:ELsym1}, we apply the method of moving planes to a subset $D$ of a bounded domain $\Om$. The idea is to show that $D$ is mirror symmetric in every direction. We introduce some notations.

Given a direction $\te\in \SS^{N-1}$ and $\la \in \RR$, let $\pi_\la=\{x\in\RR^N:\lan x, \te\ran = \la\}$. For a fixed point $x\in\RR^N$, we define $x^*$, the reflection of $x$ in $\pi_\la$, by 
$$
x^*= x -2\left( \lan x,\te\ran - \la \right)\te.
$$
\par 
Let $H_\la = \{x\in\RR^N: \lan x, \te \ran > \la\}$ so that $\pi_\la = \pa H_\la$. Also, let $D_\la = D \cap H_\la$ and $D^*_\la$ its reflection in $\pi_\la$. We set $\La= \sup\{\lan x, \te \ran : x\in D\}$. Suppose that $D$ is of class $C^2$. From \cite[Theorem 5.7]{Fra}, for $\la < \La$ sufficiently close to $\La$, we have that $D^*_\la\subseteq D$. Here, we observe that, from \cite[Lemma 2.8]{CMS-JEMS}, if $D^*_{\la} \subset D$, then $\Om_{\la}^* \subseteq \Om$.
\par 
Let $\la^*$ be the number defined by
$$
\la^*=\inf\{\la < \La: D_\mu^* \subseteq D: \ \mbox{ for any } \la< \mu < \La\}.
$$
\par
Eventually, one of the following two cases occurs:
\begin{enumerate}[(1)]
\item the boundary of $D^*_{\la^*}$ becomes tangent to that of $D$, and the set of tangency contains points not belonging to $\pi_{\la^*}$.
\item $\pi_{\la^*}$ is orthogonal to the boundary of $D$ and of $D^*_{\la^*}$, at some point of the intersection.
\end{enumerate}

\begin{thm}
\label{th:TILsym1}
Set $p\in(1,\infty)$. Let $\Om$ be a bounded domain of class $C^2$ and let $u$ be the solution of \eqref{G-heat}-\eqref{boundary}. Suppose that $D$ is a $C^2$ domain such that $\ol{D}\subset \Om$, $\pa D=\Ga_R$, for some $R>0$, and there exists $\ol{t}>0$ for which $u$ is constant on $\Ga_R\times \{\ol{t}\}$.
%Suppose that there exist $R > 0$ and $\ol{t} > 0$ such that $\Ga_R$  is of class $C^2$, $D=\{x\in \Om:d_\Ga(x) > R\}$ is connected and $u(\cdot, \ol{t})$ is constant on $\Ga_R$. 
%\begin{equation}
%\label{parallelism condition}
%x\mapsto u(x,\ol{t})
%\end{equation}
%is constant on $\Ga_r$.
\par 
Then,
$D$ and
$\Om$ must be concentric balls. 
\end{thm}

\begin{proof}
%Let $D$ be the set
%$$
%D=\{y\in\Om:d_\Ga(y)< r\}.
%$$
\par 
%Let $D$ be the set
%$$
%D=\{ x\in \Om: d_\Ga(x) > R\},
%$$
%such that $\pa D = \Ga_R$.
Since $\Om_{\la^*}^*\subset \Om$, we can define the function $u^*:\Om^*_{\la^*}\times(0,\infty)\to\RR$, by $u^*(x,t)=u(x^*,t)$, for $x\in\Om_{\la^*}^*$ and $t>0$. It is easy to see that
\begin{equation*}
\begin{cases}
u^*_t-\De_p^G u^*=0 \ &\mbox{ in } \ \Om^*_{\la^*}\times(0,\infty),\\
u^*=0 \ &\mbox{ on } \ \Om^*_{\la^*}\times\{0\},\\
u^*\ge u \ &\mbox{ on } \ \pa \Om^*_{\la^*}\times(0,\infty).
\end{cases}
\end{equation*}
Moreover, by the (weak) comparison principle (see Corollary \ref{cor:parabolic comparison}), we have that the function $w=u - u^*$ is non-positive on the whole $\ol{\Om^*_{\la^*}}\times(0,\infty)$.
\par
Now, following the proof of \cite[Theorem 1.1]{CMS-JEMS} (see also \cite[Theorem 1]{AG} \cite[Theorem 1.1]{BK}), we can apply the strong maximum principle to $w$, in a proper sub-cylinder of $\Om\times(0,\infty)$. 
Consider $v(x)=u(x,\ol{t})$, for $x\in D$. From Lemma \ref{parabolic monotonicity}, we have that $v$ is a non-constant viscosity subsolution of $-\De_p^G v = 0$ in $D$. Moreover, $v$ equals a constant on $\Ga_R$. Thus, by applying the strong maximum principle (see Remark \ref{maximum principle}) and Corollary \ref{cor:elliptic hopf-oleinik}, we obtain that 
$$
\na v\neq 0\ \mbox{ on } \ \Ga_R,
$$
where we use that $u$ is differentiable in $\Om\times(0,\infty)$.
\par
 Also, we have that there exist $\de > 0$ and $t_1, t_2 > 0$, with $t_1 < \ol{t} < t_2$,  such that 
\begin{equation}
\label{eq:regularity strip}
\na u\neq 0\ \mbox{ in }\ S_\de\times (t_1,t_2),
\end{equation}
where $S_\de=\{y\in\Om:d_{\Ga_R}(y)< \de\}$. Observe that the last set displayed in \eqref{eq:regularity strip} is a neighborhood of $\Ga_R$. 
\par
%Thus, we have that $|\na u|,|\na u^*|>0$ in $\pa D_{\la^*}^* \cap \pa D\times\{\ol{t}\}$. Hence, there exist $\de>0$ and $t_1,t_2 >0$, with $t_1 < \ol{t} < t_2$, such that $|\na u|, |\na u^*| > 0$ in $S_\de \cap \Om_{\la^*}^* \times (t_1,t_2)$, where $S_\de$ is the connected component of $\left(\pa D_{\la^*}^* \cap \pa D\right) + B_\de$ containing $y$ or $z$ in the respective cases (1) or (2).
Hence, by using \eqref{eq:regularity strip}, we have that $|\na u|,|\na u^*|>0$ in $\left(S_\de \cap \Om_{\la^*}^*\right)\times(t_1,t_2)$. Therefore, by a standard procedure, we can infer that $w$ is a (non-positive) solution of the following uniformly parabolic equation with smooth coefficients
 \begin{equation}
 \label{eq:regularized equation}
   w_t-\tr\left\{A'\na^2 w\right\} - b \cdot \na w =0 \ \mbox{ in } \ \left(S_\de \cap \Om_{\la^*}^*\right)\times(t_1,t_2),
 \end{equation}
where $A'$ and $b$ are defined, respectively, by the following expressions:
\begin{equation*}
  A'=\int_{0}^{1}\na_X F\left(\si\na u + (1- \si)\na u^*, \si\na^2 u+(1-\si)\na^2 u^*\right)d\si
\end{equation*}
and
\begin{equation*}
  b=\int_{0}^{1}\na_\xi F\left(\si\na u + (1- \si)\na u^*, \si\na^2 u+(1-\si)\na^2 u^*\right)d\si.
\end{equation*}
Here, $F$ is defined in \eqref{p-laplace-coefficients}. Observe that (as shown in \cite[Theorem 1.1]{BK}) the matrix $A'$ is uniformly elliptic and the coefficients $b$ are bounded in $\left(S_\de \cap \Om_{\la^*}^*\right)\times(t_1,t_2)$.
\par 
Applying the classical strong maximum principle to \eqref{eq:regularized equation} (see \cite{PW}) yields that either $w \equiv 0$ on $\ol{\left(S_\de\cap \Om^*_{\la^*}\right)}\times [t_1,t_2]$ or $w<0$ in $\left(S_\de \cap \Om_{\la^*}^*\right)\times(t_1,t_2)$.
\par  
Now, we conclude as in \cite[Theorem 1.1]{CMS-JEMS}. Suppose that $D^*_{\la^*}\subset D\setminus \ol{D_{\la^*}}$. Then $S_\de \cap \Om_{\la^*}^*$ contains points that are in $\pa D_{\la^*}^*\setminus \pa D$. This implies that
\begin{equation}
\label{strictly negative}
w<0 \ \mbox{ in }\ \left(S_\de\cap \Om_{\la^*}^* \right)\times(t_1,t_2)
\end{equation} 
and that
\begin{equation}
\label{hopf on pi}
\frac{\pa w}{\pa \te} > 0 \ \mbox{ on } \ \left(S_\de
 \cap \pi_{\la^*} \right) \times (t_1,t_2).
\end{equation}
%and, by the other hand, it must occur one of the following:
%\begin{enumerate}[(i)]
%\item 
%there exists $y\in \Ga_r \cap \pa D'\setminus \pi_\la$;
%\item 
%there exists $y\in \Ga_r \cap \pi_\la$ with $\te \in T_y\left(\pa D'\right)$.
%\end{enumerate}
\par
If (1) occurs, then there exists $y \in \pa D_{\la^*}^* \cap \pa D$. This implies that $u(y,\ol{t})= u^*(y,\ol{t})$ since $\pa D$ is a level surface of $u(\cdot, \ol{t})$ and hence that $w(y,\ol{t})=0$ which contradicts \eqref{strictly negative}. 
\par 
If (2) occurs, then $\pa D$ must be orthogonal to $\pi_{\la^*}$ at some $z \in \pa D$, then $\frac{\pa u}{\pa \te}(z,\ol{t}) = 0$. Also, we have that $\frac{\pa u^*}{\pa \te}(z,\ol{t}) = 0$ and then that $\frac{\pa w}{\pa \te} (z,\ol{t}) = 0$. This contradicts \eqref{strictly negative}, since $\frac{\pa w}{\pa \te}(z,\ol{t}) > 0$, by \eqref{hopf on pi}.
\par
Hence, $D$ must be symmetric with respect to every direction and hence $D$ must be a ball. Since $\Om$ and $D$ are $C^2$, then we conclude that $\Om$ must be a ball.
\end{proof}

\begin{rem}
{
\rm
We point out that in Theorem \ref{th:TILsym1}, we can replace the $C^2$ regularity of $\Ga_R$ with a weaker assumption. Indeed, we can apply the Hopf-Oleinik lemma to infer that $\na u \neq 0$ on $\Ga_R$ by assuming that $\Ga_R$ admits an interior $\om$-pseudo ball condition, where $\om$ satisfies the assumptions of Lemma \ref{hopf-oleinik}. For example, if we assume $\Ga_R \in C^{1,\al}$, for some $\al \in(0,1)$, then Theorem \ref{th:TILsym1} still holds true.
}
\end{rem}

Now, we are ready to obtain the following characterization of balls, as a consequence of Theorem \ref{th:TILsym1}.

\begin{cor}
\label{cor:TILsym2}
Set $p\in(1,\infty)$. Let $\Om$ be a bounded domain of class $C^2$ and  $u$ be the solution of \eqref{G-heat}-\eqref{boundary}. Suppose that $D$ is a $C^2$ domain such that $\ol{D}\subset \Om$ and with boundary $\pa D=\Si$ satisfying \eqref{time-invariant surface}.
\par 
Then, $D$ and $\Om$ must be concentric balls.
\end{cor}

\begin{proof}
It is enough to apply Theorem \ref{th:TILsym1}. Indeed, we just note that, from Theorem \ref{th:parallelism-invariant level set}, there exists $R>0$, such that $\Si=\Ga_R$.
\end{proof}

With some adjustments, we obtain the same conclusions of Theorem \ref{th:TILsym1} and Corollary \ref{cor:TILsym2} in the case of the elliptic problem \eqref{G-elliptic}-\eqref{elliptic-boundary}. 

{
\thm
\label{th:ELsym1}
Set $p\in(1,\infty)$. Let $\Om$ be a bounded domain of class $C^2$ and $u^\ve$ be the solution of \eqref{G-elliptic}-\eqref{elliptic-boundary}. Suppose that  $D$ is a $C^2$ domain such that $\ol{D}\subset \Om$, $\pa D = \Ga_R$, for some $R>0$, and there exists $\ol{\ve}>0$ for which $u=u^{\ol{\ve}}$ is constant on $\Ga_R$.
%Suppose that there exist $R > 0$ and $\ol{\ve} > 0$ such that $\Ga_R$  is of class $C^2$ and $u = u^{\ol{\ve}}$ is constant on $\Ga_R$.	
\par
Then, $D$ and $\Om$ must be concentric balls.
}

\begin{proof}
It is well defined the function $u^*:\Om_{\la^*}^*\to\RR$, by $u^*(x)=u(x^*)$, for $x\in \Om_{\la^*}^*$. Then, it is an easy check to prove that
\begin{equation*}
\begin{cases}
u^*-\ol{\ve}^2 \De_p^G u^*=0 \ &\mbox{ in } \ \Om_{\la^*}^*,\\
u^*\ge u \ &\mbox{ on } \ \pa \Om^*_{\la^*}.
\end{cases}
\end{equation*}
Hence, by applying the comparison principle (see Corollary \ref{cor:elliptic comparison}), we have that $w=u - u^*\le 0$ on $\ol{\Om_{\la^*}^*}$. 
\par
Now, proceeding as in the proof of Theorem \ref{th:TILsym1}, by applying Corollary \ref{cor:elliptic maximum} and \ref{cor:elliptic hopf-oleinik}, we have that $|\na u|, |\na u^*|>0$ in $S_\de$, the set  defined in the proof of Theorem \ref{th:TILsym1}. By using standard elliptic regularity theory, we have that $w$ is a (non-positive) solution of a smooth uniformly elliptic equation
\begin{equation}
\label{eq:smooth elliptic equation}
w-\ol{\ve}^2 \tr\left\{A' \na^2 w\right\} - \ol{\ve}^2\lan b, \na w \ran =0 \ \mbox{ in } \ S_\de,
\end{equation}
where $A'$ and $b$ have the same structure of those of the proof of Theorem \ref{th:TILsym1}. Since $w\le 0$, we can apply to $w$ the classical strong maximum principle (see \cite{PW}), which implies that either $w\equiv 0$ on $\ol{S_\de \cap \Om^*_{\la^*} }$ or $w<0$ in $S_\de\cap \Om_{\la^*}^*$. Hence, we conclude as in the proof of Theorem \ref{th:TILsym1}.
\end{proof}

{
\cor
\label{cor:ELsym2}
Set $p\in(1,\infty)$. Let $\Om$ be a bounded domain of class $C^2$ and $u^\ve$ be the solution of \eqref{G-elliptic}-\eqref{elliptic-boundary}. Suppose that $D$ is a domain of class $C^2$, such that $\ol{D} \subset \Om$. Suppose that $\pa D$ is a level surface of $u^\ve$, for any $\ve >0$.
\par
Then, $D$ and $\Om$ must be concentric balls.
}

\begin{proof}
We conclude from Theorems \ref{th:parallelilsm-invariant level set-elliptic} and \ref{th:ELsym1}.
\end{proof}

\section{The case of non-compact boundaries}
\label{sec:sliding}

%Let $f:\RR^{N-1}\to \RR$ be a global Lipschitz continuous function satisfying
%\begin{equation}
%\label{eq:unbounded condition}
%\lim_{|x'|\to\infty}f(x'+\xi)-f(x')=0,
%\end{equation}
%for any $\xi\in\RR^{N-1}$.

Let $f:\RR^{N-1}\to \RR$ be a continuous function. In this Subsection, we consider domains of the following form
\begin{equation}
\label{epigraphic domain}
\Om=\{(x',x_N)\in\RR^{N-1}\times\RR:x_N>f(x')\}.
\end{equation}

\begin{thm}
\label{th:monosym1}
Set $p\in(1,\infty)$. Let $\Om$ be defined by \eqref{epigraphic domain}, with $f$ of class $C^2$. Suppose that there exists a basis $\{\xi^1,\dots,\xi^{N-1}\} \subset \RR^{N-1}$ such that  for every $j=1,\dots, N-1$, the function $f(x'+\xi^j)-f(x')$ has either a maximum or a minimum in $\RR^{N-1}$.  Let $u$ be the bounded solution of \eqref{G-heat}-\eqref{boundary}.  Suppose that there exists $R>0$ such that $\Ga_R$ is of class $C^2$ and $\ol{t}>0$ for which one of the following occurs:
\begin{enumerate}[(i)]
\item $u(\cdot,\ol{t})$ is constant on $\Ga_R$.
\item for some $q\in[1,\infty)$, the function $x\mapsto \mu_q(x,\ol{t})$ is constant on $\Ga_R$.
\end{enumerate}
Then, $f$ is affine and so $\Ga$ is a hyperplane.
\end{thm}

\begin{proof}
From the assumption on $f$, for a fixed $j=1,\dots,N-1$, the function of $x'$ defined by $f(x'+\xi^j)-f(x')$ has either a maximum or a minimum in $\RR^{N-1}$. Say $f(x'+\xi^j)-f(x')$ has a maximum $M$ in $\RR^{N-1}$. Then there exists $z' \in \RR^{N-1}$ such that 
$$
f(x'+\xi^j)-f(x') \le M = f(z'+\xi^j) - f(z')\ \mbox{ for any } \ x'\in\RR^{N-1}.
$$
\par
We apply the sliding method, a variant of the method of moving planes, introduced in \cite{BCN}. Here, we adapt the proof given with that method in \cite[Theorem 1.1]{Sa-INDAM}. Also see \cite{MS-IUMJ, MS-JDE, MS-AIHP}. Set 
$$
\Om_{\xi^j,M}=\{(x',x_N)\in\RR^N: (x'+\xi^j, x_N + M)\in \Om\}.
$$
We have that $\Om \subseteq \Om_{\xi^j,M}$ and that $z=(z',f(z')) \in \Ga \cap \left(\Ga\right)_{\xi^j,M}$. In particular, from the regularity of both $\Ga$ and $\Ga_R$ we can infer that, if $\nu_\Ga$ represents the unit outward normal to $\Ga$, then $y=z - \nu_\Ga(z) R  \in \Ga_R \cap \left(\Ga_R\right)_{\xi^j,M}$, which means that $y + (\xi^j,M)\in \Ga_R$. 
\par
Since $\Om\subseteq \Om_{\xi^j,M}$, we can define the function $u^*:\Om\times(0,\infty)\to \RR$, by $u^*(x,t)=u(x+(\xi^j,M),t)$, for $x\in\Om$ and $t>0$. It easy to see that 
\begin{equation*}
\begin{cases}
u^*_t-\De_p^G u^*=0 \ &\mbox{ in } \ \Om\times(0,\infty),\\
u^*=0 \ &\mbox{ on } \ \Om\times\{0\},\\
u^*\le u \ &\mbox{ on } \ \Ga\times(0,\infty).
\end{cases}
\end{equation*}
Thus, from the comparison principle (Corollary \ref{cor:parabolic comparison}), $u^*\le u$ on $\ol\Om\times(0,\infty)$. 
\par
Now, suppose that $\Om \subset \Om_{\xi^j, M}$. As we have done in the proof of Theorem \ref{th:TILsym1}, we show that we can use the strong comparison, in a proper sub-cylinder of $\Om\times(0,\infty)$. By using Corollaries \ref{cor:elliptic maximum} and \ref{cor:elliptic hopf-oleinik} and the continuity of $\na u$, there exist $\de>0$ and $t_1 < \ol{t} < t_2$ such that
$$
\na u\neq 0\ \mbox{ in } \ S_\de \times (t_1,t_2),
$$
where here $S_\de$ is the compact set $ B_R(y) \cap \left(\Ga_R + B_\de \right)$.
\par 
Thus, $u^*-u$ is a solution in $S_\de \times (t_1,t_2)$ of the uniformly parabolic equation with smooth coefficients \eqref{eq:regularized equation} and hence we can apply to $u^* - u$ the strong maximum principle. In particular, since $\Om$ does not coincide with $\Om_{\xi^j,M}$, we infer that
\begin{equation}
\label{contradiction to ii}
u^*<u \ \mbox{ in } \ S_\de\times(t_1,t_2).
\end{equation} 
\par
Denoting with $\mu^*_q(y,t)$ the $q$-mean of $u^*(\cdot,\ol{t})$ on $B_R(y)$, the last inequality implies that 
\begin{equation}
\label{contradiction to iii}
\mu_q^{*}(y,\ol{t})<\mu_q(y,\ol{t}),
\end{equation}
since $|\{u^* < u\}\cap B_R(y)| > 0$. Indeed, from \eqref{eq:monotonicty of qmeans}, we have that $\mu^* \leq \mu $ and from the fact that the function $u\mapsto |u-\mu^*|^{q-2} (u-\mu^*)$ is strictly increasing, we have that
%\begin{equation*}
%0= \int_{B_R(y)} |u^*(\zeta,\ol{t}) - \mu^*|^{q-2}(u^*(\zeta,\ol{t}) - \mu^*) \,d\zeta <\\
% \int_{B_R(y)} |u(\zeta,\ol{t}) - \mu^*|^{q-2} (u(\zeta,\ol{t}) - \mu^*)\,d\zeta
%\end{equation*}
\begin{multline*}
0= \int_{B_R(y)} |u^*(\zeta,\ol{t}) - \mu^*|^{q-2}(u^*(\zeta,\ol{t}) - \mu^*) \,d\zeta <\\
 \int_{B_R(y)} |u(\zeta,\ol{t}) - \mu^*|^{q-2} (u(\zeta,\ol{t}) - \mu^*)\,d\zeta
\end{multline*}
which implies that $\mu^* \neq \mu$.
\par 
Now, we prove that if either $(i)$ or $(ii)$ holds, we find a contradiction to \eqref{contradiction to ii} or \eqref{contradiction to iii}. If $(i)$ holds then we have that $u(y,\ol{t}) = u(y+(\xi^j,M), \ol{t}) = u^*(y,\ol{t})$, which contradicts \eqref{contradiction to ii} at once. If $(ii)$ holds, then
\begin{equation}
\label{equations for mu}
\mu_q(y,\ol{t})=\mu_q(y + (\xi^j,M),\ol{t})=\mu_q^*(y,\ol{t}),
\end{equation}
where the first equality is due to $(ii)$ and the latter one  is based on the following argument. Employing a change of variable $\zeta = \zeta' + (\xi^j,M)$ yields that
%\begin{equation*}
%\int_{B_R(y + (\xi^j,M))} |u(\zeta , \ol{t})-\la|^{q-2} \left(u(\zeta,\ol{t})-\la\right)\, d\zeta =
%\int_{B_R(y)} |u^*(\zeta',\ol{t})-\la|^{q-2} \left(u^*(\zeta', \ol{t})-\la\right)\, d\zeta',
%\end{equation*}
\begin{multline*}
\int_{B_R(y + (\xi^j,M))} |u(\zeta , \ol{t})-\la|^{q-2} \left(u(\zeta,\ol{t})-\la\right)\, d\zeta =\\
\int_{B_R(y)} |u^*(\zeta',\ol{t})-\la|^{q-2} \left(u^*(\zeta', \ol{t})-\la\right)\, d\zeta',
\end{multline*}
for any $\la\in\RR$, which implies the latter equation in \eqref{equations for mu}, by the definition of $q$-means. Hence, we have found again a contradiction. Hence, we must have $\Om_{\xi^j,M}=\Om$.
\par
Now, we conclude as in \cite[Theorem 1.1]{Sa-INDAM}. Indeed, we have that for any $1\le j \le N-1$, for any $x'\in \RR^{N-1}$,
$$
f(x'+\xi^j)-f(x') = a_j, 
$$
for some $a_j\in \RR$. The continuity of $f$, the fact that $\{\xi^1,\dots, \xi^{N-1}\}$ is a basis of $\RR^{N-1}$ and an iteration of the sliding method imply that
$$
f(x'+y')-f(x')=f(z'+y')-f(z') \ \mbox{ for any }\ x',y',z' \in \RR^{N-1}.
$$
Since $f$ is continuous, solving the latter system of functional equations yields that $f$ must be affine.
%Now we conclude as in \cite[Theorem 3.4]{MS-IUMJ}, or \cite{BCN}. We have that $\Om \subseteq \Om_{\xi,h}$, for any $h\geq 0$, which implies that $\Om = \Om_{\xi,0}$ and hence that $f$ is periodic with period $\xi$. Since $\xi\in\RR^{N-1}$ is arbitrary, we conclude that $f$ is constant and hence $\Ga$ is a hyperplane.
\end{proof}

%\begin{rem}
%{
%\rm
%The condition \eqref{eq:unbounded condition} can be weakened. See \cite{MS-AIHP, MS-JDE, Sa-INDAM}.
%}
%\end{rem}

\begin{cor}
\label{cor:monosym2}
Set $p\in (1,\infty)$. Let $\Om$ be defined by \eqref{epigraphic domain}, with $f$ of class $C^2$. Suppose that there exists a basis $\{\xi^1,\dots,\xi^{N-1}\} \subset \RR^{N-1}$ such that  for every $j=1,\dots, N-1$, the function $f(x'+\xi^j)-f(x')$ has either a maximum or a minimum in $\RR^{N-1}$. Let $u$ be the bounded solution of \eqref{G-heat}-\eqref{boundary}. Suppose that there exists a surface $\Si$ of class $C^2$ such that \eqref{time-invariant surface} holds true.
\par
Then $f$ is affine and $\Ga$ is a hyperplane.
\end{cor}

\begin{proof}
It is sufficient to observe that, from Theorem \ref{th:parallelism-invariant level set}, there exists $R>0$ such that $\Si=\Ga_R$. Hence, we apply Theorem \ref{th:monosym1}.
\end{proof}

Also in this case, we can give the corresponding theorem in the case \eqref{G-elliptic}-\eqref{elliptic-boundary}.

\begin{thm}
\label{th:monosym3}
Set $p\in(1,\infty)$. Let $\Om$ be defined by \eqref{epigraphic domain}, with $f$ of class $C^2$. Suppose that there exists a basis $\{\xi^1,\dots,\xi^{N-1}\} \subset \RR^{N-1}$ such that  for every $j=1,\dots, N-1$, the function $f(x'+\xi^j)-f(x')$ has either a maximum or a minimum in $\RR^{N-1}$. Let $u^\ve$ be the bounded solution of \eqref{G-elliptic}-\eqref{elliptic-boundary}. Suppose that there exists $R>0$ such that $\Ga_R$ is of class $C^2$ and $\ol{\ve}>0$ for which one of the following occurs:
\begin{enumerate}[(i)]
\item $u^{\ol{\ve}}$ is constant on $\Ga_R$.
\item for some $q\in[1,\infty)$, the function $x\mapsto \mu_{q,\ol{\ve}}(x)$ is constant on $\Ga_R$.
\end{enumerate}
Then, $f$ is affine and so $\Ga$ is a hyperplane.
\end{thm}

\begin{proof}
With the same notations as in the proof of Theorem \ref{th:monosym1}, we have that, for $j=1, \dots, N-1$, we can define the function $u^*:\Om\to \RR$, by $u^*(x)=u^{\ol{\ve}}\left(x+\left(\xi^j,M\right)\right)$. For the comparison principle and the maximum principle (Corollaries \ref{cor:parabolic comparison} and \ref{cor:parabolic maximum}), we have that $u^*-u^{\ol{\ve}}\le 0$, on $\ol\Om$. 
\par
Moreover, there exists $\de>0$ such that $|\na u^*|, |\na u|>0$ in $S_\de$. Thus, $u^*-u^{\ol{\ve}}$ satisfies the uniformly elliptic equation with smooth coefficients \eqref{eq:smooth elliptic equation} in $S_\de$.
\par
Supposing that $\Om \subset \Om_{\xi^j,M}$ yields that $u^*-u^{\ol{\ve}}<0$, in $S_\de$, from the strong maximum principle. Thus, we conclude as in the proof of Theorem \ref{th:monosym1}.
\end{proof}

\begin{cor}
\label{cor:monosym4}
Set $p\in(1,\infty)$. Let $\Om$ be defined by \eqref{epigraphic domain}, with $f\in C^2$. Suppose that there exists a basis $\{\xi^1,\dots,\xi^{N-1}\} \subset \RR^{N-1}$ such that  for every $j=1,\dots, N-1$, the function $f(x'+\xi^j)-f(x')$ has either a maximum or a minimum in $\RR^{N-1}$. Let $u^\ve$ be the bounded solution of \eqref{G-elliptic}-\eqref{elliptic-boundary}. Suppose that there exists a $C^2$ surface $\Si$ that is level surface of $u^\ve$, for any $\ve > 0$.
\par
Then, $\Ga$ must be a hyperplane.
\end{cor}

\begin{proof}
We apply together Theorems \ref{th:monosym3} and \ref{th:parallelilsm-invariant level set-elliptic} .
\end{proof}

\section{Spherical symmetry for $q$-means-invariant surfaces} 
\label{sec:symmetry-qmeans}

In this section we give applications of Theorems \ref{th:JMPA-asymptotics-qmean} and   \ref{th:AA-qmean}. We give characterizations of spheres, based on $q$-means, in the spirit of \cite[Theorem 1.2]{MS-PRSE}. These results are new, even for the case $p=2$. 
\par

Here, we consider those parallel surfaces $\Ga_R$ sufficiently near $\Ga$, such that for every $x\in\Ga_R$ there exists an unique $y_x$ such that $\ol{B_R(x)} \cap \Ga =\{y_x\}$, for every $y_x$. Also, we suppose that $\ka_1(y_x),\cdots \ka_{N-1}(y_x) < \frac{1}{R}$.

\begin{thm}[{\cite[Theorem 3.7]{BM-JMPA}}]
\label{th:radiality-bounded domain}
Set $1<p\le\infty$ and let $\Om$ be a domain of class $C^2$ with bounded and connected boundary $\Ga$. 
Let $u$ be the bounded (viscosity) solution of \eqref{G-heat}-\eqref{boundary}. 
\par 
Suppose that $\Si$ is a $C^2$-regular surface in $\Om$, that is a parallel surface to $\Ga$ at distance $R>0$.\par
If, for some $1<q<\infty$ and every $t>0$, the function
$$
\Si\ni x\mapsto \mu_q (x,t)
$$
is constant, then $\Ga$ must be a sphere.
\end{thm}

\begin{proof}
Since $\Si$ is of class $C^2$ and is parallel to $\Ga$, for every $y\in\Ga$, there is a unique $x\in\Si$ at distance $R$ from $y$.  Thus, owing to Theorem \ref{th:JMPA-asymptotics-qmean}, we can infer that 
$$
\Pi_\Ga=\mbox{constant on $\Ga$.}
$$ 
Our claim then follows from a  variant of Alexandrov's Soap Bubble Theorem (see \cite{Al}),  \cite[Theorem 1.2]{MS-PRSE}, or \cite[Theorem 1.1]{MS-AM}.  
\end{proof}

The next is the elliptic counterpart to Theorem \ref{th:radiality-bounded domain}. Here, we intend that $\mu_{q,\ve}$ is the $q$-mean of $u^\ve$ on $B_R(x)$.

\begin{thm}
\label{th:radiality-bounded domain-elliptic}
Set $1<p\le\infty$ and let $\Om$ be a domain of class $C^2$ with bounded and connected boundary $\Ga$. 
Let $u^\ve$ be the bounded (viscosity) solution of \eqref{G-elliptic}-\eqref{elliptic-boundary}. 
\par 
Suppose that $\Si$ is a $C^2$-regular surface in $\Om$, that is a parallel surface to $\Ga$ at distance $R>0$.\par
If, for some $1<q<\infty$ and every $\ve > 0$, the function
$$
\Si\ni x\mapsto \mu_{q,\ve} (x)
$$
is constant, then $\Ga$ must be a sphere.
\end{thm}

\begin{proof}
The proof runs similarly to that of Theorem \ref{th:radiality-bounded domain}, once we replace Theorem \ref{th:JMPA-asymptotics-qmean} by Theorem \ref{th:AA-qmean}.
\end{proof}

%\afterpage{
%\thispagestyle{empty}
%}
%\afterpage{\blankpage}

\bibliography{mybib}{}
  \bibliographystyle{abbrv} 

\end{document}